\documentclass{article}
\usepackage{amsmath, amssymb, graphicx, algpseudocode, empheq}
\usepackage{mathtools}
\usepackage{algorithm}
\usepackage{subcaption}
\usepackage{empheq}
\usepackage[english]{babel}
\usepackage{amssymb}
\usepackage[letterpaper,top=2cm,bottom=2cm,left=3cm,right=3cm,marginparwidth=1.75cm]{geometry}
\usepackage{amsthm}
\usepackage{booktabs}
\usepackage{mathrsfs}
\usepackage{algorithm}  
\usepackage{algpseudocode}
\usepackage{amsmath}
\usepackage{bm}
\usepackage{graphicx}
\usepackage[colorlinks=true, allcolors=blue]{hyperref}
\usepackage{booktabs,tabularx,array} \usepackage[numbers,sort&compress]{natbib}
\newtheorem{theorem}{Theorem}
\numberwithin{theorem}{section}

\usepackage[makeroom]{cancel}
\newtheorem{remark}{Remark}[section]

\newcommand{\omb}{{\boldsymbol\omega}}
\newcommand{\omg}{{\hat{\Omega}^s}}
\usepackage{authblk}
\newenvironment{keywords}
{\par\medskip\noindent\textbf{Keywords:}\ }
{\par\medskip}

\newenvironment{MSCcodes}
{\par\medskip\noindent
\textbf{Mathematics Subject Classification:}\ }
{\par\medskip}

\title{A Coupled Fluid–Naghdi Shell–Windkessel Model for Pulmonary Artery Banding: Partitioned Numerical Scheme and Parameter Estimation}

\author[1]{Yifan Wang}
\author[2]{Josip Tamba\v{c}a}
\author[3]{Sun\v{c}ica \v{C}ani\'{c}*}

\affil[1]{
Department of Mathematics and Statistics,
Texas Tech University, Lubbock, TX, USA\\
\texttt{yifan.wang@ttu.edu}
}

\affil[2]{
Department of Mathematics, Faculty of Science,
University of Zagreb, Croatia\\
\texttt{tambaca@math.hr}
}

\affil[3]{
*Corresponding Author: Department of Mathematics,
University of California, Berkeley, CA, USA\\
\texttt{canics@berkeley.edu}
}

\date{}

\begin{document}
\maketitle

\begin{abstract}
We develop a coupled three-dimensional (3D) fluid-structure interaction model for blood flow through deformable vessels with localized constriction, motivated by pulmonary artery banding in neonates with congenital heart disease.
The model couples the incompressible Navier-Stokes equations in an Arbitrary Lagrangian Eulerian (ALE) framework with a 3D Naghdi Shell for the arterial wall, defined on a two-dimensional (2D) reference domain, and a zero-dimensional three-element Windkessel model for downstream circulation.
A central feature of the formulation is the use of a low-regularity Naghdi shell model that retains membrane, bending, and transverse-shear deformation while requiring
only first derivatives of the reference midsurface. This allows the model to accommodate piecewise-smooth vascular geometries, including localized constrictions with discontinuous surface slopes. 

We derive the coupled variational formulation and develop a non-iterative pratitioned scheme to solve the coupled system. The coupling incorporates the inertia of the arterial wall through a Robin-type interface condition. For a linearized fixed-domain problem, we establish a discrete energy estimate under suitable assumptions.
To account for uncertainty in the downstream circulation, we further couple the model with a reduced-state unscented
Kalman filter to estimate the proximal resistance, distal resistance, and compliance from inlet-pressure
observations. Our numerical experiments demonstrate consistent pressure-wave propagation and wall deformation under spatial and temporal
refinement. Comparisons between the standard Naghdi shell and a modified, predominantly radial shell model illustrate
the influence of structural deformation mechanisms on pressure-wave speed and displacement. Moreover, an idealized pulmonary artery banding example further
demonstrates the usage and application of the coupled Fluid-Naghdi shell-Windkessel model with parameter estimation. The resulting framework provides a basis for investigating the hemodynamic and mechanical effects of localized vascular constrictions while accounting for the influence of downstream vascular properties.

\end{abstract}


\section{Introduction}
We focus on neonates with congenital heart disease involving a large, nonrestrictive ventricular septal defect, as illustrated in Figure~\ref{fig:vsd}. Because the defect offers little resistance to interventricular flow, the systolic pressures in the two ventricles may become nearly equal. Consequently, the direction and magnitude of the shunt depend primarily on the balance between pulmonary and systemic vascular resistance rather than simply on a pressure difference between the left and right ventricles. As pulmonary vascular resistance decreases after birth, left-to-right shunting generally increases, directing excessive blood flow into the pulmonary circulation. The resulting pulmonary overcirculation can cause pulmonary congestion and volume overload of the left atrium and left ventricle. Therefore, changes in pulmonary or systemic vascular resistance can substantially alter the pulmonary-to-systemic flow ratio and the distribution of cardiac output between the two circulations.
\begin{figure}[htbp]
\centering
\includegraphics[width=0.75\textwidth]{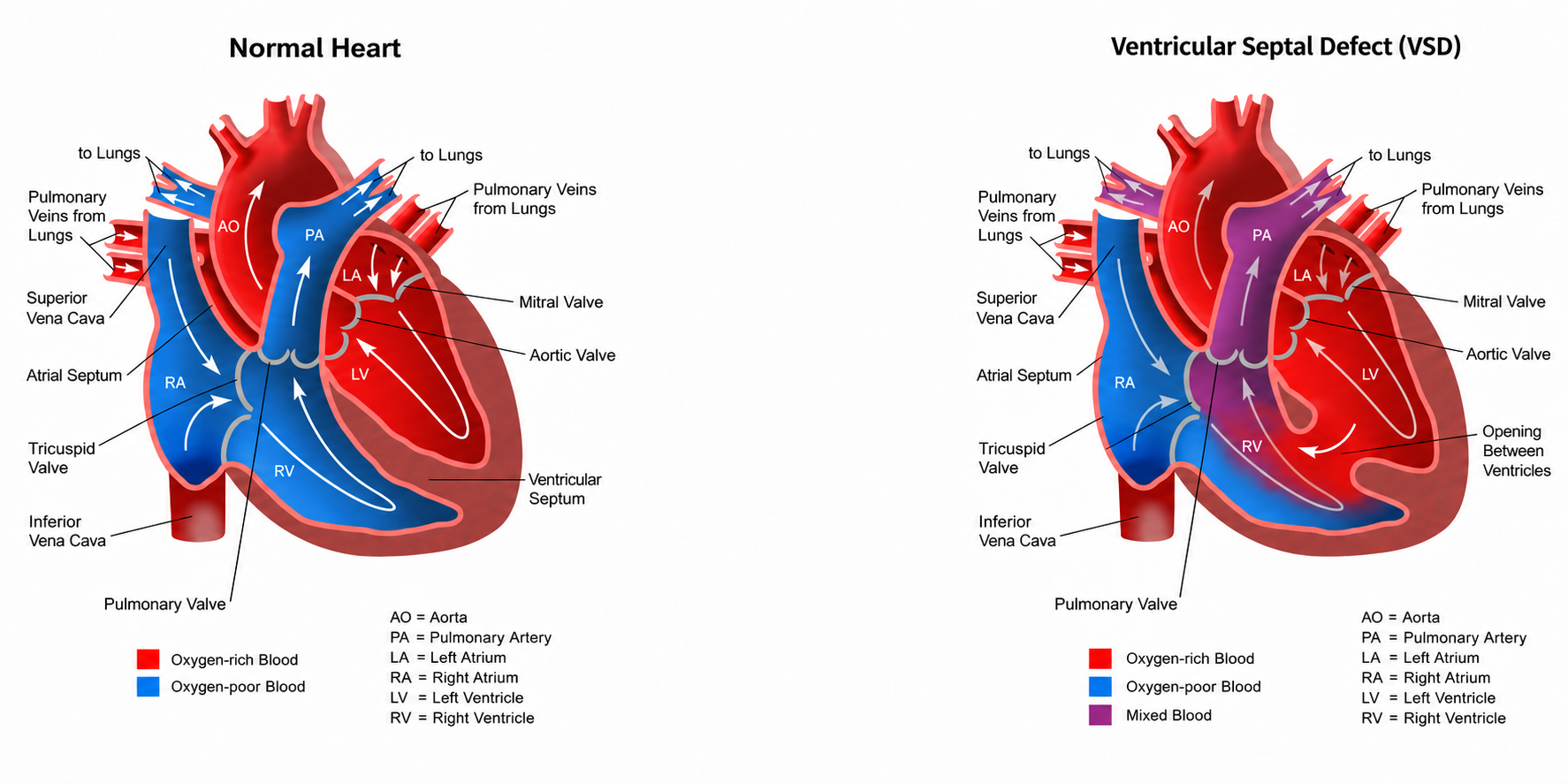}
    \caption{Comparison of a normal heart and a ventricular septal defect \cite{PHC4_2019_PediatricHeartSurgery}.}
    \label{fig:vsd}
\end{figure}

Pulmonary artery banding is a palliative intervention used to limit excessive pulmonary blood flow in selected neonates with congenital heart disease when immediate complete repair is not feasible or is not clinically preferred. The procedure creates a localized constriction in the main pulmonary artery. Its hemodynamic effect, however, depends not only on the band diameter and axial width, but also on the shape and position of the constriction, the native vessel geometry and wall properties, the incoming flow, and the resistance and compliance of the downstream pulmonary circulation. A band that is too loose may fail to reduce pulmonary overcirculation sufficiently, whereas a band that is too tight may produce an excessive pressure drop, increase right-ventricular afterload, and reduce pulmonary blood flow and systemic oxygen delivery. In clinical practice, empirical sizing rules and intraoperative measurements of pressure and oxygen saturation are used to guide band adjustment, but these limited measurements cannot uniquely determine the optimal patient-specific band configuration or the properties of the distal pulmonary circulation \cite{sharma2012pulmonary,trusler1972method,durandy2013pulmonary}. This limitation motivates a mathematical framework that distinguishes the local mechanical effects of the band from the influence of the unresolved downstream pulmonary vasculature.
\begin{figure}[htbp]
    \centering
    \begin{subfigure}[c]{0.32\linewidth}
        \centering
        \includegraphics[width=\linewidth]
        {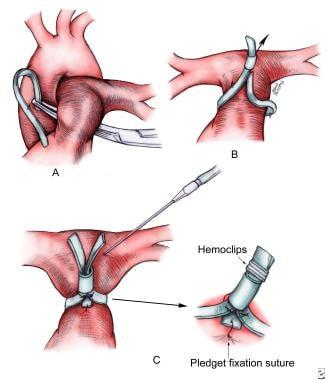}
        \caption{Pulmonary artery banding procedure \cite{chaudhary_pulmonary_artery_banding}.}
        \label{fig:pab-procedure}
    \end{subfigure}
    \hfill
    \begin{subfigure}[c]{0.58\linewidth}
        \centering
        \includegraphics[width=\linewidth]
        {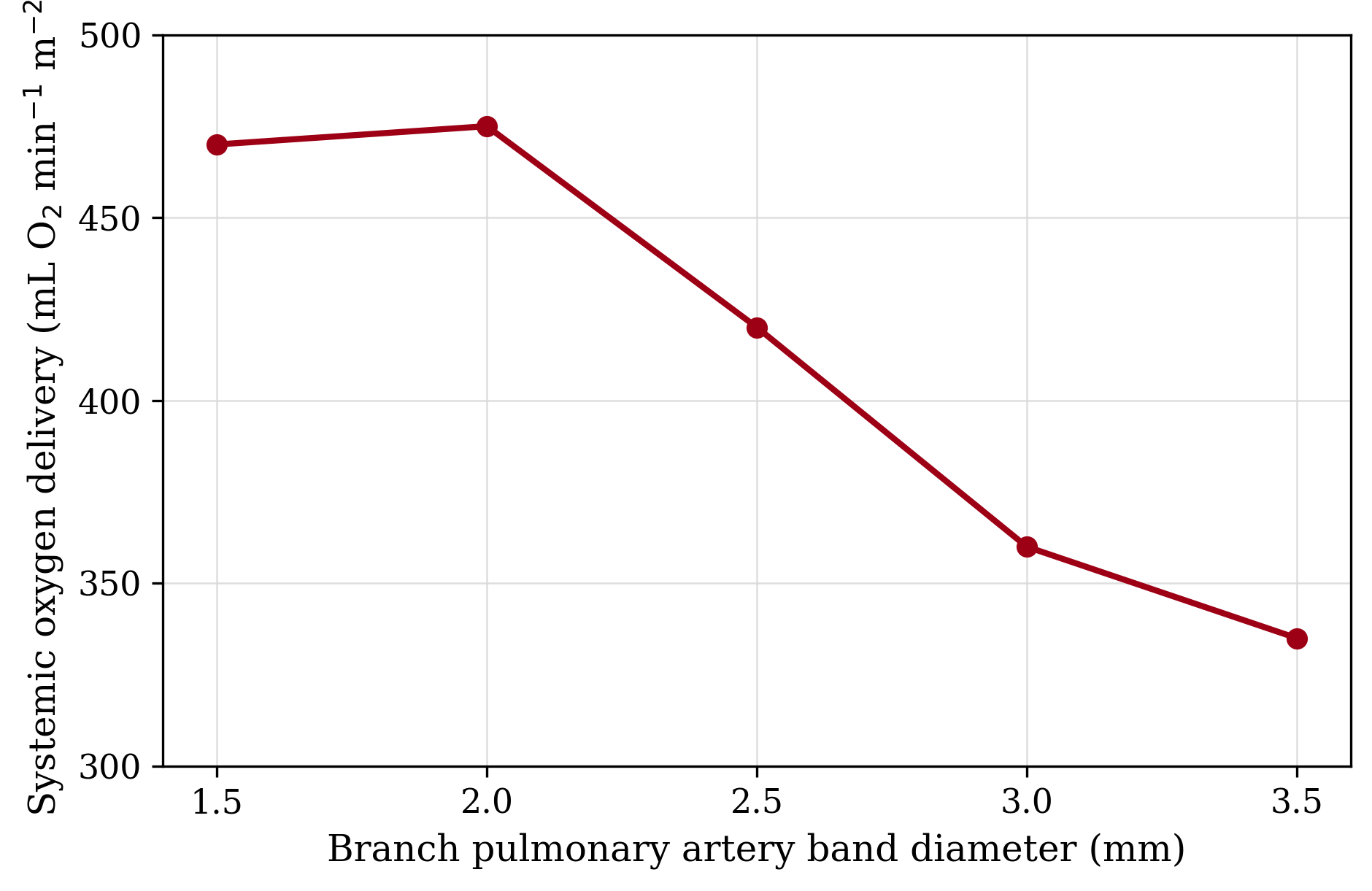}
        \caption{Systemic oxygen delivery as a function of pulmonary
        artery band diameter, provided by Stephen John, MD at UTHealth.}
        \label{fig:pab-oxygen-delivery}
    \end{subfigure}
    \caption{Pulmonary artery banding procedure and its predicted effect on
systemic oxygen delivery. Panel~(a) illustrates the placement, tightening,
and fixation of a band around the main pulmonary artery to restrict excessive
pulmonary blood flow. The band is secured using hemoclips and a pledgeted
fixation suture to reduce the risk of migration ~\cite{angeli2012pulmonary}. Panel~(b) shows the
clinically observed relationship between branch pulmonary artery band diameter
and systemic oxygen delivery. Within the investigated range, systemic oxygen
delivery reaches its maximum near a band diameter of $2.0\,\mathrm{mm}$ and
decreases as the diameter increases further.  }
    \label{fig:pab-comparison}
\end{figure}

Previous computational studies of pulmonary artery banding have investigated the effects of band shape and position on jet formation, flow distribution, and pressure loss using rigid-wall computational fluid dynamics and multiscale circulation models \cite{corno2006noncircular,liu2015effects,baker2013effects}. These studies demonstrate the potential value of virtual banding, but two important coupled effects remain insufficiently resolved. First, a rigid-wall model cannot describe arterial-wall deformation and stress or the feedback of wall motion on the surrounding blood flow. Second, prescribed or manually adjusted outlet conditions may obscure the distinction between pressure changes caused locally by the band and those caused by the impedance of the distal pulmonary circulation. To address these limitations, we therefore consider a three-dimensional (3D) moving-domain fluid-structure interaction (FSI) model coupled to a three-element Windkessel outlet. The Windkessel model represents the unresolved downstream circulation through a proximal resistance, a distal resistance, and a compliance, thereby providing a dynamic relationship between outlet flow and pressure \cite{vignonclementel2010outflow,zambrano2018image}. Complementary reduced-dimensional approaches have also been developed
for blood flow through arteries with spatially varying radius.
In particular, an extended 1D model incorporating the axial variation
of vessel radius was derived and assessed against 3D simulations
for stenotic arteries \cite{CanicEtAl2024Stenosis}.
The present framework resolves the 3D flow together with arterial-wall
displacement and shell-director rotation in a localized constriction.

Because the arterial wall is thin relative to the vessel radius, it can be represented by a two-dimensional (2D) shell model without explicitly resolving the 3D wall thickness. Different shell theories retain different deformation mechanisms. For instance, a membrane model accounts primarily for in-surface stretching, whereas a classical Koiter shell includes both membrane and bending effects while imposing the Kirchhoff–Love constraint of zero transverse shear \cite{ciarlet2000shells, chapelle2003finite}. A Naghdi shell introduces the midsurface displacement and the shell-director rotation as independent unknowns and therefore retains membrane, bending, and transverse-shear deformation \cite{Naghdi1972}. These distinctions are summarized in Table \ref{tab:shell-model-comparison}.
The Naghdi-type formulation is particularly useful when the band transition is represented by a piecewise-smooth mid-surface with limited regularity. In particular, the model of Tambača and Tutek is formulated using only first derivatives of the reference midsurface and remains well defined for globally continuous, piecewise-smooth surfaces with limited regularity, including surfaces with corner-like transitions  \cite{TambacaTutek2016}. Its natural weak formulation involves first-order Sobolev spaces, which permits the use of standard continuous Lagrange finite elements \cite{canic2017dimension}. We therefore represent the pulmonary-artery wall by a Naghdi-type shell defined on an axisymmetric, piecewise-smooth reference midsurface and couple it to an incompressible Newtonian fluid governed by the Navier–Stokes equations. Earlier investigations of vascular fluid--structure interaction
include the analysis of a linear 3D fluid--mesh--shell problem
with a Koiter shell admitting displacement in all three spatial
directions \cite{CanicEtAl2019FluidMeshShell}, as well as
simulations of pulsatile blood flow through curved stented
coronary arteries subject to heart motion
\cite{BukacEtAl2019Stents}.
These studies provide related settings for investigating the
interaction between blood flow, arterial-wall deformation,
and vascular interventions.

\begin{table*}[!t]
\centering
\caption{Comparison of structural wall models used in shell mechanics and
fluid--structure interaction. Here, ``independent rotations'' means that the
rotation of the shell director is an independent unknown rather than being
determined by the midsurface displacement.}
\label{tab:shell-model-comparison}
\footnotesize
\renewcommand{\arraystretch}{1.18}
\setlength{\tabcolsep}{4pt}
\begin{tabularx}{\textwidth}{
    >{\raggedright\arraybackslash}p{0.175\textwidth}
    >{\centering\arraybackslash}p{0.08\textwidth}
    >{\centering\arraybackslash}p{0.10\textwidth}
    >{\centering\arraybackslash}p{0.10\textwidth}
    >{\centering\arraybackslash}p{0.10\textwidth}
    >{\raggedright\arraybackslash}X}
\toprule
Model & Membrane deformation & Bending
& Transverse shear
& Independent rotations
& Best physical regime and representative references \\
\midrule
Naghdi shell
& Yes
& Yes
& Yes
& Yes
& Thin to moderately thick shells, especially when transverse shear,
independent director rotation, or low-regularity midsurfaces are important
\citep{Naghdi1972,BlouzaLeDret2001,TambacaTutek2016}. \\

Classical cylindrical Koiter shell
& Yes
& Yes
& No
& No
& Thin cylindrical shells for which the Kirchhoff--Love constraint is
appropriate; widely used in vascular fluid--structure interaction
\citep{Koiter1960,Ciarlet2000,BukacEtAl2013,MuhaCanic2013}. \\

Reduced Koiter shell
& Yes
& Reduced or model-dependent
& No
& Usually no
& Thin vessels with restricted, often axisymmetric or radial-dominant,
wall motion; the retained bending and displacement components depend on the
particular reduction \citep{QuarteroniTuveriVeneziani2000,
FormaggiaQuarteroniVeneziani2009}. \\

Membrane model enhanced with transverse shear
& Yes
& No full shell bending energy
& Yes
& Model dependent
& Membrane-dominated vascular deformation in which transverse-shear modes
are added to improve the response and stability under normal loading
\citep{OshimaEtAl2001,FigueroaEtAl2006}. \\
\bottomrule
\end{tabularx}
\end{table*}

We discretize the coupled problem using a partitioned operator-splitting method from the family of kinematically coupled fluid–structure interaction schemes. Conventional explicit Dirichlet–Neumann partitioning may suffer from added-mass instability, particularly when a light structure interacts with an incompressible fluid, as occurs in cardiovascular applications. Kinematically coupled schemes mitigate this instability by incorporating the translational inertia of the shell into the fluid subproblem through a Robin interface condition \cite{bukac2013fluid,canic2015stability}. At each time step, the Naghdi shell is first advanced using a selected fraction of the previous fluid traction; the resulting displacement updates the Arbitrary Lagrangian Eulerian (ALE) map and mesh velocity; the fluid velocity, fluid pressure, and Windkessel pressure are then solved together on the updated domain. Related partitioned schemes support stable nonlinear moving-domain computations without fluid-structure sub-iterations at each time step \cite{bukac2015partitioned}. Other related explicit interface-coupling strategies include a fully
parallelizable loosely coupled method for fluid-poroelastic
structure interaction \cite{GuoEtAl2025FPSI, WangLeeCanic2026ErrorAnalysis, WangLeeCanic2026SecondOrder} and a Robin--Robin
partitioned method for an incompressible fluid interacting with
a thick elastic structure \cite{GuoEtAl2025RobinRobin}.
Both approaches advance the coupled subproblems without
fluid--structure subiterations.
Here, we use the translational inertia of the Naghdi shell
in a Robin condition and incorporate the Windkessel outlet
into the fluid subproblem.

Patient-specific outlet calibration introduces a second computational challenge. The finite-element fluid–structure interaction state contains a large number of degrees of freedom, whereas the unknown Windkessel parameter vector contains only the proximal resistance, distal resistance, and compliance. We therefore use a reduced-state unscented Kalman filter in which only these three parameters belong to the filter state \cite{deus2024efficient}. At each measurement time, seven sigma points are generated in the parameter space, propagated through separate fluid–structure interaction simulations, and compared with the observed inlet pressure. The resulting discrepancy is then used to update the parameter estimates and their covariance \cite{pant2014methodological,bertoglio2012sequential,arthurs2020flexible, deus2024efficient}. This approach avoids the tangent and adjoint equations that would be required by derivative-based estimation methods \cite{moireau2011reduced}. The filtering is performed using the logarithms of the three Windkessel parameters, thereby ensuring that all estimated resistances and compliances remain positive.

The framework is designed to support a two-stage calibration-and-prediction workflow. In a future patient-specific application, the downstream Windkessel parameters would first be estimated from measured pulmonary-artery flow and pressure data. The estimated proximal resistance, distal resistance, and compliance would then be held fixed while alternative band geometries and arterial-wall properties are evaluated. This strategy is intended to separate the local hemodynamic effects of the band from uncertainty in the unresolved distal pulmonary circulation. Subject to the assumption that the downstream vascular properties remain unchanged over the time scale of the virtual intervention, the framework could be used to compare candidate band configurations in terms of pressure reduction, pulmonary flow, peak jet velocity, wall deformation, and wall stress.

In the present work, we establish the mathematical and computational components required for this workflow. Section 2 introduces the 3D fluid domain, the variable-radius shell geometry, the Arbitrary Lagrangian–Eulerian fluid formulation, the Naghdi shell model, and the fluid–structure coupling conditions. Section 3 presents the coupled weak formulation. Section 4 develops the non-iterative partitioned scheme, including the structure, mesh-update, and fluid–Windkessel subproblems, while Section 5 derives an energy estimate for a linearized fixed-domain problem. Section 6 formulates the reduced-state unscented Kalman filter for estimating the three positive Windkessel parameters from inlet-pressure observations. The numerical studies are presented in Section 7. In Section 7.1, we verify the 3D fluid–Naghdi-shell implementation using a straight compliant tube subjected to a short pressure pulse. The cross-sectionally averaged pressure and mean-diameter profiles exhibit consistent behavior under both spatial and temporal refinement. In Section 7.2, we examine the influence of the additional Naghdi deformation mechanisms by comparing the standard model with a scaled and penalized formulation designed to produce a predominantly radial, Koiter-like response. Both formulations reproduce the expected pressure-wave and wall-expansion behavior, while the modified model produces a faster-moving pressure front than the standard Naghdi model. Finally, Section 7.3 introduces an idealized banded pulmonary-artery configuration, couples it to a three-element Windkessel outlet, and specifies the physiological and computational setting for pressure-based parameter estimation. Collectively, these studies provide an initial proof of concept for the coupled solver and parameter-estimation framework rather than a completed patient-specific band-optimization study.

\section{Mathematical Model} 

\subsection{Fluid and shell domains}
\label{shelldomain}
We consider a fluid–structure interaction problem in which an incompressible viscous fluid is coupled to a thin elastic shell. The reference midsurface of the shell is axisymmetric and is described by a radius function \(R(x)\), where \(x\in[0,L]\) is the axial coordinate and \(L\) is the vessel length. We assume that \(R\) is strictly positive and Lipschitz continuous, with the same reference radius \(R_0\) at the inlet and outlet:
$$
R\in W^{1,\infty}(0,L),
\qquad
\inf_{x\in[0,L]}R(x)>0,
\qquad
R(0)=R(L)=R_0.
\label{def:R}
$$

The unbanded configuration corresponds to a tube of constant radius,
$$
R(x)=R_0
\qquad\text{for all }x\in[0,L].
$$
While a banded configuration is represented by a profile satisfying
$$
0<R(x)\leq R_0,
$$
where \(R(x)<R_0\) only within a localized axial region and returns continuously to \(R_0\) outside that region. The band radius is defined by
$$
R_b=\min_{x\in[0,L]}R(x),
\label{def:R_b}
$$
whereas the axial extent of the region in which \(R(x)<R_0\) defines the total band width, including any transition zones. To accommodate band profiles composed of finitely many smooth geometric segments, we assume that there exists a partition
$$
0=x_0<x_1<\cdots <x_N=L
$$
such that \(R\) is \(C^1\) on each closed subinterval \([x_{j-1},x_j]\). At each transition point \(x_j\), the radius remains continuous, but its axial derivative may have a jump:
$$
R(x_j^-)=R(x_j^+)=R(x_j),
\qquad
R'(x_j^-)\neq R'(x_j^+).
$$
Consequently, the reference midsurface is a connected Lipschitz surface that may contain circular geometric corners at the transition locations. This regularity is compatible with the Naghdi-type shell formulation for midsurfaces described by \(W^{1,\infty}\) parameterizations; see \cite{TambacaTutek2016}.

\paragraph{Fluid domain}
For a prescribed radius profile $R(x)$, let
$\hat{\Omega}_f\subset\mathbb R^3$ denote the reference fluid domain,
defined by
\[
\hat{\Omega}_f
=
\left\{
\hat{\boldsymbol x}
=
(\hat x,\hat y,\hat z)\in\mathbb R^3:
0<\hat x_1<L,\quad
\hat y^2+\hat z^2<R(\hat x)^2
\right\}.
\]
Its boundary is decomposed as
\[
\partial\hat{\Omega}_f
=
\overline{\hat{\Sigma}_f^D}
\cup
\overline{\hat{\Sigma}_f^N}
\cup
\overline{\hat{\Gamma}_s},
\]
where $\hat{\Sigma}_f^D$ is the portion on which essential boundary
conditions are imposed, $\hat{\Sigma}_f^N$ is the portion on which
natural boundary conditions are imposed, and $\hat{\Gamma}_s$ is the
reference fluid-structure interface. These boundary portions are
assumed to be relatively open and pairwise disjoint, except possibly
along their common edges.

For each $t\in[0,T]$, the corresponding current fluid domain is
denoted by $\Omega_f(t)$. Its motion is described by the Arbitrary
Lagrangian-Eulerian mapping
\[
\hat{\varphi}_f(\cdot,t):
\hat{\Omega}_f
\longrightarrow
\Omega_f(t),
\qquad
\hat{\varphi}_f(\hat{\boldsymbol x},t)
=
\hat{\boldsymbol x}
+
\hat{\boldsymbol\eta}_f(\hat{\boldsymbol x},t),
\]
where $\hat{\boldsymbol\eta}_f$ is the fluid-mesh movement. We
assume that $\hat{\varphi}_f(\cdot,t)$ is a sufficiently regular,
orientation-preserving bijection for every $t\in[0,T]$. We define the Jacobian determinant and assume the uniform nondegeneracy condition
\[
\hat J_f(\hat{\boldsymbol x},t)
=
\det
\left(
\hat{\nabla}
\hat{\varphi}_f(\hat{\boldsymbol x},t)
\right),\qquad
\operatorname*{ess\,inf}_{
(\hat{\boldsymbol x},t)
\in
\hat{\Omega}_f\times(0,T)}
\hat J_f(\hat{\boldsymbol x},t)
>0.
\]

The corresponding current boundary portions are
\[
\Sigma_f^D(t)
=
\hat{\varphi}_f(\hat{\Sigma}_f^D,t),
\qquad
\Sigma_f^N(t)
=
\hat{\varphi}_f(\hat{\Sigma}_f^N,t),
\qquad
\Gamma(t)
=
\hat{\varphi}_f(\hat{\Gamma}_s,t).
\]
For a scalar or vector valued function $g$ defined on
$\Omega_f(t)$, its pullback to the reference fluid domain is
\[
\hat g(\hat{\boldsymbol x},t)
=
g\left(
\hat{\varphi}_f(\hat{\boldsymbol x},t),t
\right).
\]
The fluid-mesh velocity in the current configuration is
\[
\boldsymbol w_f(\boldsymbol x,t)
=
\partial_t\hat{\varphi}_f
\left(
\hat{\varphi}_f^{-1}(\boldsymbol x,t),t
\right),
\qquad
\boldsymbol x\in\Omega_f(t).
\]

Since we assume that
$R\in W^{1,\infty}(0,L),~
\inf_{x\in[0,L]}R(x)>0$, this
imply that the reference fluid domain $\hat{\Omega}_f$ is a Lipschitz
domain. We further assume that, for every $t\in[0,T]$, the ALE mapping
is a bi-Lipschitz homeomorphism:
\[
\hat{\varphi}_f(\cdot,t)
\in
W^{1,\infty}(\hat{\Omega}_f;\mathbb R^3),
\qquad
\hat{\varphi}_f^{-1}(\cdot,t)
\in
W^{1,\infty}(\Omega_f(t);\mathbb R^3),
\]
with bounds that are uniform in time. Consequently, the current fluid
domain $\Omega_f(t)$ is also Lipschitz, and its outward unit normal
$\boldsymbol n_f$ is defined almost everywhere on
$\partial\Omega_f(t)$.

\begin{figure}[htbp]
    \centering
    \includegraphics[width=0.7\textwidth]{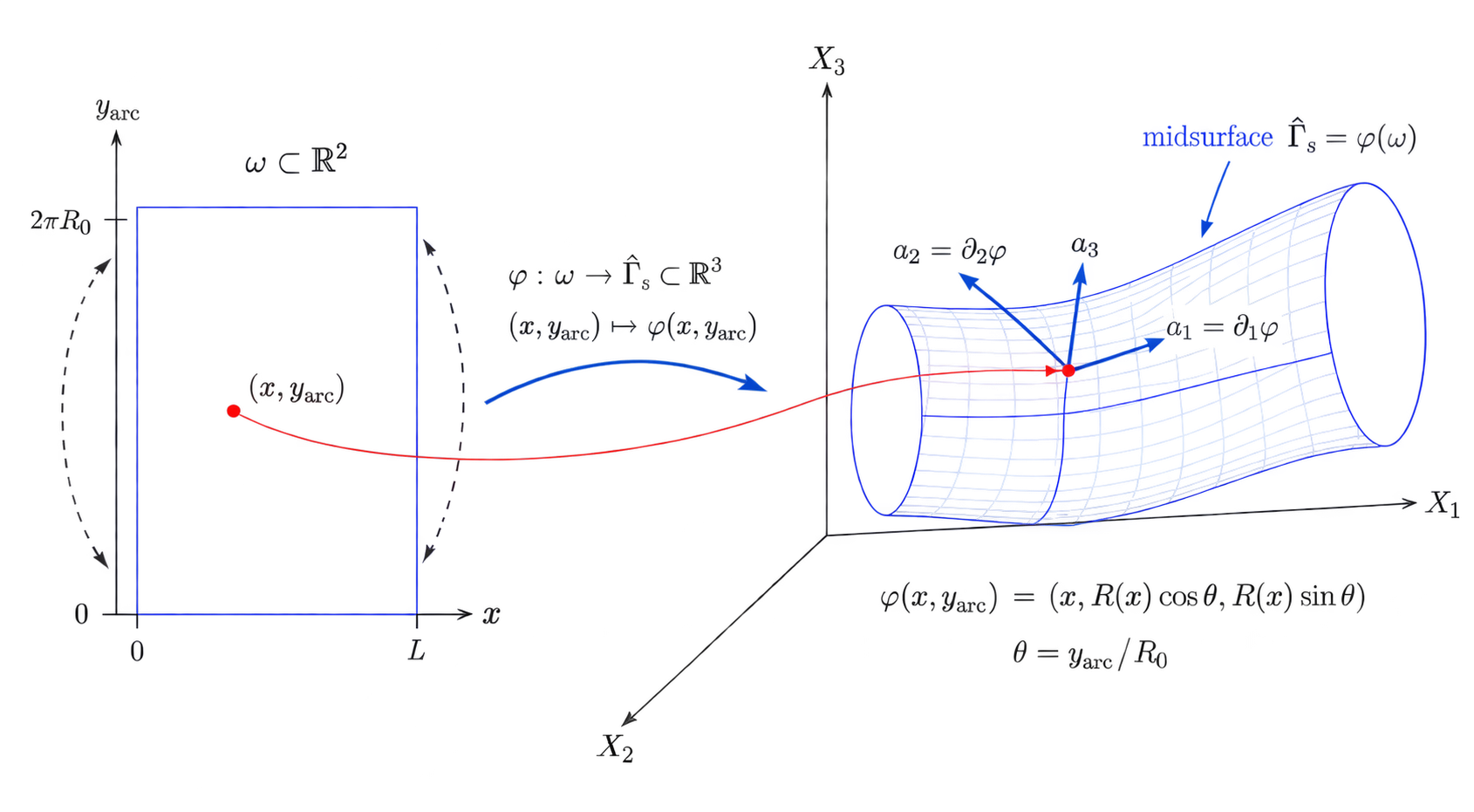}
    \caption{Schematic representation of the Naghdi shell midsurface. The two-dimensional parameter domain, with periodic circumferential coordinate $y_{\mathrm{arc}}$, is mapped onto the midsurface  $\hat{\Gamma}_s={\varphi}(\omg)\subset\mathbb{R}^3$ by ${\varphi}(x,y_{\mathrm{arc}})=\bigl(x,R(x)\cos\theta,R(x)\sin\theta\bigr)$. The vectors $\boldsymbol{a}_1=\partial_1{\varphi}$ and $\boldsymbol{a}_2=\partial_2{\varphi}$ span the tangent plane, while $\boldsymbol{a}_3$ denotes the unit normal to the midsurface.}
    \label{SketchNaghdi}
\end{figure}

\paragraph{Shell domain}
The arterial wall is represented by a two-dimensional (2D) Naghdi-type
shell embedded in $\mathbb R^3$, as illustrated in Figure.~\ref{SketchNaghdi}. The shell is assumed to be cylindrical, and has length $L$, thickness
$\varepsilon$, with the reference radius profile $R(x)$. Its reference
midsurface is parameterized over the fixed domain
\[
\omg =
(0,L)\times(0,2\pi R_0).
\]
The coordinates on $\omg$ are denoted by
$(x,y_{\mathrm{arc}})$, where $x$ is the axial coordinate and
$y_{\mathrm{arc}}$ is the circumferential reference coordinate. We assume the top and bottom
boundaries
$
y_{\mathrm{arc}}=0~
\text{and}~ 
y_{\mathrm{arc}}=2\pi R_0
$
are periodic.

We define an angular variable 
$\theta=\frac{y_{\mathrm{arc}}}{R_0}$, then the unified reference midsurface map is
\[
\varphi:
\omg
\longrightarrow
\mathbb R^3,
\qquad
\varphi(x,y_{\mathrm{arc}})
=
\left(
x,
R(x)\cos\theta,
R(x)\sin\theta
\right).
\]
The reference fluid structure interface is therefore
\[
\hat{\Gamma}_s
=
\varphi(\omg).
\]

Throughout this paper, we write $\partial_1=\partial_x,
~
\partial_2=\partial_{y_{\mathrm{arc}}}$.
Then, the covariant basis vectors associated with the reference midsurface map are
defined almost everywhere on $\omg$ by
\[
\boldsymbol a_1
=
\partial_1\varphi
=
\left(
1,
R'(x)\cos\theta,
R'(x)\sin\theta
\right),
\quad
\boldsymbol a_2
=
\partial_2\varphi
=
\left(
0,
-\frac{R(x)}{R_0}\sin\theta,
\frac{R(x)}{R_0}\cos\theta
\right).
\]
Their cross product is
\begin{equation}
\boldsymbol a_1\times\boldsymbol a_2
=
\frac{R(x)}{R_0}
\left(
R'(x),-\cos\theta,-\sin\theta
 \right)
,
\quad
 \left|\boldsymbol a_1\times\boldsymbol a_2\right|
 =
 \sqrt a
 =
 \frac{R(x)}{R_0}\sqrt{1+R'(x)^2}
.
\label{def:a}
\end{equation}
With this orientation, we can define the outward unit normal to the reference fluid domain
on the lateral interface as $-\boldsymbol a_3$, where
\[
\boldsymbol a_3
=
\frac{\boldsymbol a_1\times\boldsymbol a_2}
{\left|\boldsymbol a_1\times\boldsymbol a_2\right|}
=
\frac{
\left(
R'(x),-\cos\theta,-\sin\theta
\right)
}{
\sqrt{1+R'(x)^2}
}.
\]
The contravariant tangent vectors are determined by
\[
\boldsymbol a^\alpha\cdot\boldsymbol a_\beta
=
\delta^\alpha_\beta,
\qquad
\alpha,\beta\in\{1,2\},
\]
and are given explicitly by
\[
\boldsymbol a^1
=
\frac{
\left(
1,
R'(x)\cos\theta,
R'(x)\sin\theta
\right)
}{
1+R'(x)^2
},
~
\boldsymbol a^2
=
\left(
0,
-\frac{R_0}{R(x)}\sin\theta,
\frac{R_0}{R(x)}\cos\theta
\right),
~
\boldsymbol a^3=\boldsymbol a_3.
\]
Therefore, we can obtain
the corresponding covariant metric tensor and the contravariant metric tensor
\[
\boldsymbol a_c
=
\begin{pmatrix}
1+R'(x)^2 & 0\\[1mm]
0 & \left(\dfrac{R(x)}{R_0}\right)^2
\end{pmatrix},\qquad
\boldsymbol a^c
=
\begin{pmatrix}
\dfrac{1}{1+R'(x)^2} & 0\\[3mm]
0 & \left(\dfrac{R_0}{R(x)}\right)^2
\end{pmatrix}.
\]
The metric determinant is
\begin{equation}
a
=
\det(\boldsymbol a_c)
=
\left(
1+R'(x)^2
\right)
\left(
\frac{R(x)}{R_0}
\right)^2.
\end{equation}

For the banded case, the reference surface measure is
\[
dS_0
=
\sqrt{a}\,dx\,dy_{\mathrm{arc}}
=
\frac{R(x)}{R_0}
\sqrt{1+R'(x)^2}
\,dx\,dy_{\mathrm{arc}}.
\]
While for the unbanded case, we have:
$R(x)=R_0,~
R'(x)=0,$
and thus
$$\boldsymbol a_c=\boldsymbol a^c=\mathbf I_2,
~
a=1,
~
dS_0=dx\,dy_{\mathrm{arc}}.
$$

Let
$\boldsymbol\eta:
\omg\times[0,T]
\longrightarrow
\mathbb R^3$
denote the shell midsurface displacement in Cartesian coordinates.
The current shell midsurface map is
\[
\Phi_s(x,y_{\mathrm{arc}},t)
=
\varphi(x,y_{\mathrm{arc}})
+
\boldsymbol\eta(x,y_{\mathrm{arc}},t),
\]
and the current fluid structure interface is
\[
\Gamma(t)
=
\Phi_s(\omg,t).
\]
We assume that $\Phi_s(\cdot,t)$ is injective and regular, up to the
periodic identification, and define
\[
J_s(x,y_{\mathrm{arc}},t)
=
\left|
\partial_x\Phi_s(x,y_{\mathrm{arc}},t)
\times
\partial_{y_{\mathrm{arc}}}
\Phi_s(x,y_{\mathrm{arc}},t)
\right|.
\]
We also assume the uniform nondegeneracy condition
$$
\operatorname*{ess\,inf}_{
(x,y_{\mathrm{arc}},t)
\in
\omg\times(0,T)}
J_s(x,y_{\mathrm{arc}},t)
>0.
\label{nondegeneracy}
$$
The current surface measure is therefore
\[
dS_t
=
J_s(x,y_{\mathrm{arc}},t)
\,dx\,dy_{\mathrm{arc}}.
\]

\paragraph{Fluid structure interface}

The fluid and shell mappings are required to coincide on the
fluid-structure interface $\Gamma(t)$:
\[
\hat{\varphi}_f
\left(
\varphi(x,y_{\mathrm{arc}}),t
\right)
=
\Phi_s(x,y_{\mathrm{arc}},t).
\]
Equivalently,
\[
\hat{\boldsymbol\eta}_f
\left(
\varphi(x,y_{\mathrm{arc}}),t
\right)
=
\boldsymbol\eta(x,y_{\mathrm{arc}},t).
\]
For the interior points of the fluid domain, the fluid mesh displacement is
obtained using the adaptive ALE extension proposed by Masud and
Hughes~\cite{masud1997space}. In particular, the prescribed shell
displacement is extended componentwise into the fluid domain through the
weighted harmonic problem
\begin{equation}
\begin{cases}
-\hat{\nabla}\cdot
\left(
\left[1+\tau_m(\hat{\boldsymbol x},t)\right]
\hat{\nabla}\hat{\boldsymbol\eta}_f
\right)
=
\boldsymbol 0
& \text{in }\hat{\Omega}_f,\\[2mm]
\hat{\boldsymbol\eta}_f
=
\boldsymbol\eta\circ\varphi^{-1}
& \text{on }\hat{\Gamma}_s,\\[2mm]
\hat{\boldsymbol\eta}_f
=
\boldsymbol 0
& \text{on }
\partial\hat{\Omega}_f
\setminus
\overline{\hat{\Gamma}_s}.
\end{cases}
\label{adaptive_ALE_ext}
\end{equation}
The weighting parameter $\tau_m$ is piecewise constant on each reference
fluid element $\hat K_e$ and is defined by
\begin{equation}
\left.\tau_m(\hat{\boldsymbol x},t)\right|_{\hat K_e}
=
\frac{
1-\Delta_{\min}(t)/\Delta_{\max}(t)
}{
\Delta^e(t)/\Delta_{\max}(t)
},
\qquad e=1,\ldots,n_{\mathrm{el}},
\label{adaptive_ALE_weight}
\end{equation}
where
\[
K_e(t)=\hat{\varphi}_f(\hat K_e,t),
\qquad
\Delta^e(t)=\lvert K_e(t)\rvert,
\]
and
\[
\Delta_{\min}(t)
=
\min_{1\leq e\leq n_{\mathrm{el}}}\Delta^e(t),
\qquad
\Delta_{\max}(t)
=
\max_{1\leq e\leq n_{\mathrm{el}}}\Delta^e(t).
\]
Thus, in the present 3D setting, $\Delta^e(t)$ denotes the
volume of the corresponding tetrahedral element in the current fluid mesh.
The quantities $\Delta_{\min}(t)$, $\Delta_{\max}(t)$, and $\tau_m$ are
recomputed before each ALE extension solve. Smaller elements are assigned with a larger stiffness, which prevents them to under
excessive deformation.

\subsection{Fluid equations in the ALE formulation}

The fluid is modeled as an incompressible Newtonian fluid occupying the
time-dependent domain $\Omega_f(t)$. For each $t\in(0,T)$, let
\[
\boldsymbol u(\cdot,t):\Omega_f(t)\rightarrow\mathbb R^3,
\qquad
p(\cdot,t):\Omega_f(t)\rightarrow\mathbb R
\]
denote the fluid velocity and pressure, respectively. The Cauchy stress
tensor is
\[
\boldsymbol\sigma_f(\boldsymbol u,p)
=
-p\mathbf I+2\mu_f\mathbb D(\boldsymbol u),
\qquad
\mathbb D(\boldsymbol u)
=
\frac{1}{2}
\left(
\nabla\boldsymbol u+
\left(\nabla\boldsymbol u\right)^T
\right),
\]
where $\mu_f>0$ is the fluid dynamic viscosity. All spatial derivatives
in this subsection are taken with respect to the current spatial
coordinate $\boldsymbol x\in\Omega_f(t)$.

Let
$\hat{\boldsymbol u}(\hat{\boldsymbol x},t)
=
\boldsymbol u
\left(
\hat{\varphi}_f(\hat{\boldsymbol x},t),t
\right)$
denote the pullback of the fluid velocity to the reference domain
$\hat{\Omega}_f$. Differentiating at fixed reference coordinate
$\hat{\boldsymbol x}$ gives
\[
\partial_t\hat{\boldsymbol u}(\hat{\boldsymbol x},t)
=
\left[
\partial_t\boldsymbol u
+
(\boldsymbol w_f\cdot\nabla)\boldsymbol u
\right]
\left(
\hat{\varphi}_f(\hat{\boldsymbol x},t),t
\right).
\]
Then the ALE time derivative in current coordinates is defined
by
\[
\left.
\partial_t\boldsymbol u
\right|_{\hat{\Omega}_f}
(\boldsymbol x,t)
:=
\partial_t\hat{\boldsymbol u}
\left(
\hat{\varphi}_f^{-1}(\boldsymbol x,t),t
\right)
=
\partial_t\boldsymbol u(\boldsymbol x,t)
+
\left(
\boldsymbol w_f(\boldsymbol x,t)\cdot\nabla
\right)
\boldsymbol u(\boldsymbol x,t),
\]
and therefore the incompressible Navier-Stokes equations in ALE form are
\begin{equation}
\left\{
\begin{aligned}
\rho_f
\left(
\left.
\partial_t\boldsymbol u
\right|_{\hat{\Omega}_f}
+
\left(
(\boldsymbol u-\boldsymbol w_f)\cdot\nabla
\right)
\boldsymbol u
\right)
&=
\nabla\cdot\boldsymbol\sigma_f(\boldsymbol u,p)
+
\boldsymbol F_f
&&
\text{in }\Omega_f(t),
\\
\nabla\cdot\boldsymbol u
&=
0
&&
\text{in }\Omega_f(t),
\end{aligned}
\right.
\qquad t\in(0,T).
\label{Navier-Stokes}
\end{equation}
Here $\rho_f>0$ is the constant fluid density and
$\boldsymbol F_f$ denotes the external body force.

\paragraph {Boundary conditions} 
The current non-interface boundary portions can be obtained from their
reference counterparts through the ALE mapping:
\[
\Sigma_f^D(t)
=
\hat{\varphi}_f(\hat{\Sigma}_f^D,t),
\qquad
\Sigma_f^N(t)
=
\hat{\varphi}_f(\hat{\Sigma}_f^N,t).
\]
Together with the fluid-structure interface, they satisfy
\[
\partial\Omega_f(t)
=
\overline{\Sigma_f^D(t)}
\cup
\overline{\Sigma_f^N(t)}
\cup
\overline{\Gamma(t)},
\]
where their relative interiors are pairwise disjoint. The interface
compatibility imposed in Section~\ref{shelldomain} gives
\[
\Gamma(t)
=
\hat{\varphi}_f(\hat{\Gamma}_s,t)
=
\Phi_s(\omg,t).
\]

\begin{figure}[htbp]
\centering
\includegraphics[width=0.75\textwidth]{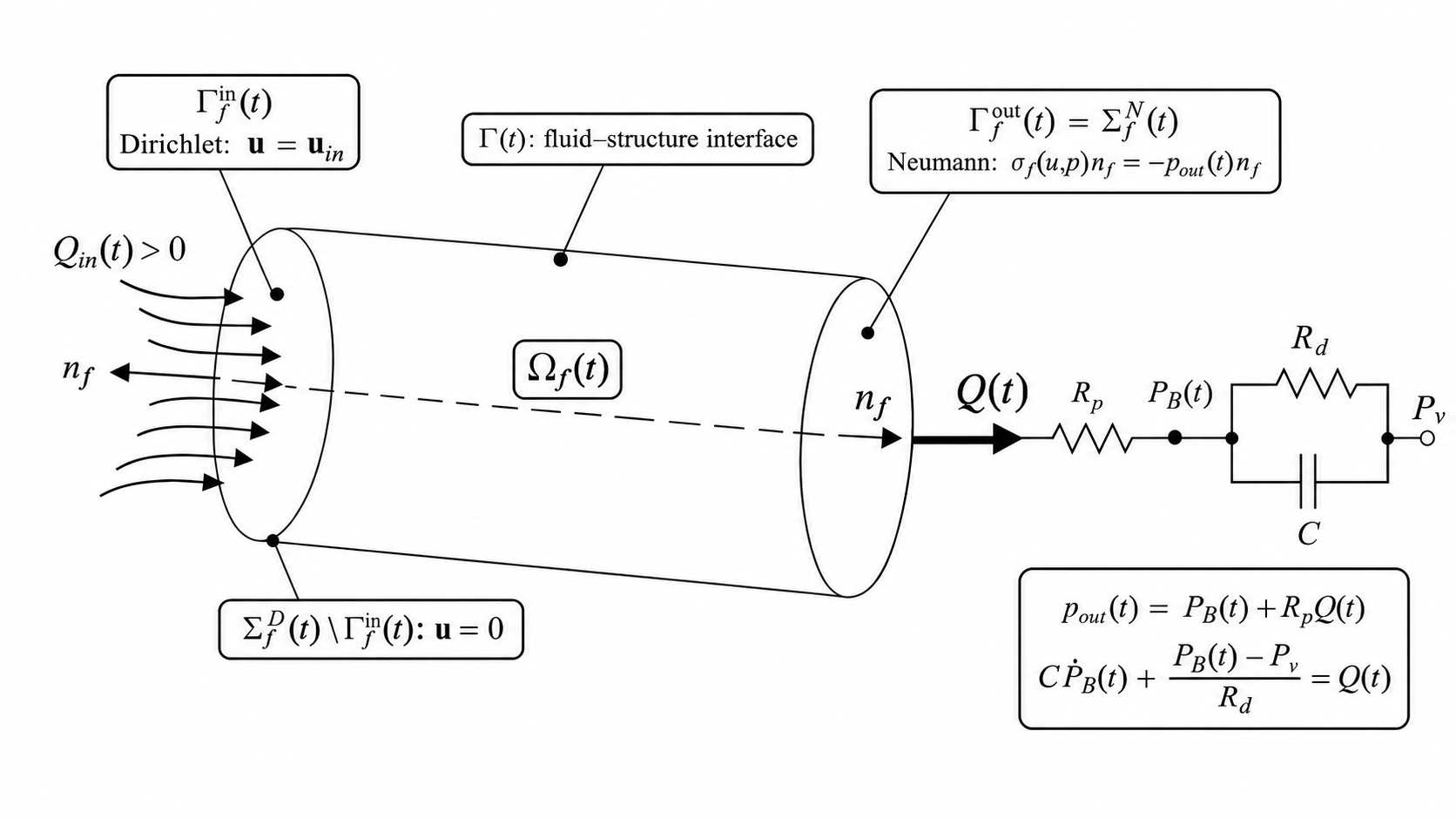}
    \caption{Schematic of the fluid domain, boundary conditions, and three-element Windkessel outlet model. }
    \label{FluidBC}
\end{figure}

As illustrated in Figure.~\ref{FluidBC}, for the specific medical problem we considered in this work, a Dirichlet boundary condition is imposed at the inlet, while a Neumann boundary condition is prescribed at the outlet:
\[
\Gamma_f^{\mathrm{in}}(t)
\subset
\Sigma_f^D(t),
\qquad
\Gamma_f^{\mathrm{out}}(t)
=
\Sigma_f^N(t).
\]
And the non-interface boundary conditions are
\begin{equation}
\left\{
\begin{aligned}
\boldsymbol u
&=
\boldsymbol u_{\mathrm{in}}
&&
\text{on }\Gamma_f^{\mathrm{in}}(t),
\\
\boldsymbol\sigma_f(\boldsymbol u,p)\boldsymbol n_f
&=
-p_{\mathrm{out}}(t)\boldsymbol n_f
&&
\text{on }\Gamma_f^{\mathrm{out}}(t),
\\
\boldsymbol u
&=
\boldsymbol 0
&&
\text{on }
\Sigma_f^D(t)
\setminus
\overline{\Gamma_f^{\mathrm{in}}(t)},
\end{aligned}
\right.
\qquad t\in(0,T).
\label{NS_BC}
\end{equation}
Here $\boldsymbol n_f$ denotes the outward unit normal to the current
fluid domain.

At the inlet $\Gamma_f^{\mathrm{in}}(t)$, we let $\boldsymbol u_{\mathrm{in}}$ be a prescribed inlet velocity profile
whose flux equals the prescribed inlet-flow waveform. Because
$\boldsymbol n_f$ points outward from the fluid domain, the positive
inflow rate is defined by
\[
Q_{\mathrm{in}}(t)
=
-\int_{\Gamma_f^{\mathrm{in}}(t)}
\boldsymbol u_{\mathrm{in}}(\boldsymbol x,t)
\cdot
\boldsymbol n_f(\boldsymbol x,t)
\,dS
>0.
\]
In particular, both the spatial distribution and the time dependence of
$\boldsymbol u_{\mathrm{in}}$ are prescribed so that its flux equals
$Q_{\mathrm{in}}(t)$. 

At the outlet boundary, we employ a three-element Windkessel model, in which the outlet pressure $p_{\mathrm{out}}(t)$ is determined by the outlet flow rate $Q(t)$ rather than prescribed independently.
The three-element Windkessel model is given by the following:
\[
C\dot P_B(t)
+
\frac{P_B(t)-P_v}{R_d}
=
Q(t),
\]
where $C>0$ is compliance,  $R_d>0$ is distal resistance, and $P_v$ is downstream venous pressure. They are all prescribed. The right hand-side, which is the outflow rate, is calculated from 
\[
Q(t)
=
\int_{\Gamma_f^{\mathrm{out}}(t)}
\boldsymbol u(\boldsymbol x,t)
\cdot
\boldsymbol n_f(\boldsymbol x,t)
\,dS.
\]
Then, the spatially uniform outlet pressure is determined by
\[
p_{\mathrm{out}}(t)
=
P_B(t)+R_pQ(t),
\]
where $R_p$ is the proximal resistance, which is prescribed.

As we shall see in Section~\ref{sec:ukf_parameter_estimation}, the Windkessel parameters $C,R_p$ and $R_d$ will be estimated to match the inlet pressure experimental measurements using an unscented Kalman filter. 

The kinematic and dynamic coupling conditions on $\Gamma(t)$ are
specified in the fluid-structure coupling section \ref{interface_conditions}.

\paragraph{Initial condition}
We assume the initial fluid velocity satisfies $\boldsymbol u(\boldsymbol x,0)
=
\boldsymbol 0,~
\text{in }\Omega_f(0)$.

\subsection{Naghdi shell model}
\label{NaghdiModel}

The elastic wall is modeled by the linear Naghdi-type shell model for
midsurfaces with limited regularity~\cite{TambacaTutek2016}. In particular, this low-regularity Naghdi formulation depends on the reference
midsurface map $\varphi$ only through its first weak derivatives.
It is therefore well defined almost everywhere for the globally
$W^{1,\infty}$, piecewise-$C^1$ regular reference parametrization
introduced in Section~\ref{shelldomain}. 
Classical curvature tensor requires second derivatives of the midsurface map, while all geometric coefficients in the current formulation are constructed from the first weak derivatives of $\varphi$ and are interpreted almost everywhere.

The Naghdi shell unknowns are
$
(\boldsymbol\eta,\omb)$,
where $\boldsymbol\eta:\omg\times(0,T)\rightarrow\mathbb R^3$ is the
midsurface displacement and
$\omb:\omg\times(0,T)\rightarrow\mathbb R^3$ is the
infinitesimal rotation of the shell cross-section. Both fields are expressed
in Cartesian coordinates. 

\paragraph{Notation:} We would like to clarify that throughout this paper, $\omg$ denotes the reference domain of the Naghdi shell, whereas $\omb$ denotes its rotation vector.

\paragraph{Membrane and transverse shear strain}
Define the Cartesian membrane-shear geometric quantity
\[
\mathbf E_{\mathrm{geo}}(\boldsymbol\eta,\omb)
=
\begin{bmatrix}
\partial_1\boldsymbol\eta+\boldsymbol a_1\times\omb
&
\partial_2\boldsymbol\eta+\boldsymbol a_2\times\omb
\end{bmatrix}
\in\mathbb R^{3\times2}.
\]
We define the transform matrix $\mathbf Q$ to convert Cartesian coordinates into components relative to the local surface basis $(\boldsymbol a_1,\boldsymbol a_2,\boldsymbol a_3)$:
\[
\mathbf Q
=
\begin{pmatrix}
\boldsymbol a^1\\
\boldsymbol a^2\\
\boldsymbol a^3
\end{pmatrix}.
\]
For vector $\boldsymbol r\in\mathbb R^3$ in Cartesian coordinates, we have the coordinate respresentation of $\boldsymbol{r}$ in the local surface basis
\[
\mathbf Q\boldsymbol r
=
\begin{pmatrix}
\boldsymbol a^1\cdot\boldsymbol r\\
\boldsymbol a^2\cdot\boldsymbol r\\
\boldsymbol a^3\cdot\boldsymbol r
\end{pmatrix}.
\]
Therefore, to express each column of $\mathbf E_{\mathrm{geo}}(\boldsymbol\eta,\omb)$ in the local surface frame, we multiply
from the left by $\mathbf Q$. Since the
first two rows of $\mathbf Q$ are the dual tangent vectors
$\boldsymbol a^1$ and $\boldsymbol a^2$, while the third row is the
unit normal $\boldsymbol a_3$, the first two rows of the resulting
matrix contain the tangential components and the third row contains
the normal components. Accordingly, we define
\[
\widehat{\mathbf E}_{\mathrm{mem}}
(\boldsymbol\eta,\omb)
=
\mathbf Q
\mathbf E_{\mathrm{geo}}(\boldsymbol\eta,\omb)
=
\begin{pmatrix}
\mathbf A(\boldsymbol\eta,\omb)\\
\boldsymbol e(\boldsymbol\eta,\omb)^T
\end{pmatrix},
\]
where
$\mathbf A(\boldsymbol\eta,\omb)\in\mathbb R^{2\times2}$
collects the tangential membrane-rotation components, and
$\boldsymbol e(\boldsymbol\eta,\omb)\in\mathbb R^2$
collects the normal, or transverse-shear, components. Since
the metric is diagonal matrix, the tangential block can be written as
\[
\mathbf A(\boldsymbol\eta,\omb)
=
\begin{pmatrix}
\boldsymbol a^1\cdot\partial_1\boldsymbol\eta
&
\boldsymbol a^1\cdot\partial_2\boldsymbol\eta
+\sqrt a\,a^{11}\omb\cdot\boldsymbol a_3
\\
\boldsymbol a^2\cdot\partial_1\boldsymbol\eta
-\sqrt a\,a^{22}\omb\cdot\boldsymbol a_3
&
\boldsymbol a^2\cdot\partial_2\boldsymbol\eta
\end{pmatrix},
\]
where $a^{11}$ and $a^{22}$ are the diagonal entries of
$\boldsymbol a^c$. The opposite signs in the off-diagonal entries come from $\boldsymbol a^2 \times \boldsymbol a^1 = - \boldsymbol a^1 \times \boldsymbol a^2$.


{\color{black}The matrix $\mathbf A(\boldsymbol\eta,\omb)$ is a
generalized, generally nonsymmetric tangential strain.
Multiplying it by the covariant metric tensor $\boldsymbol{a}_c$ we obtain, see Lemma 3.2 in \cite{TambacaTutek2016},
\[
\begin{aligned}
\boldsymbol a_c
\mathbf A(\boldsymbol\eta,\omb)
&=
\begin{pmatrix}
\boldsymbol a_1\cdot\partial_1\boldsymbol\eta
&
\dfrac12
\left(
\boldsymbol a_1\cdot\partial_2\boldsymbol\eta
+
\boldsymbol a_2\cdot\partial_1\boldsymbol\eta
\right)
\\
\dfrac12
\left(
\boldsymbol a_1\cdot\partial_2\boldsymbol\eta
+
\boldsymbol a_2\cdot\partial_1\boldsymbol\eta
\right)
&
\boldsymbol a_2\cdot\partial_2\boldsymbol\eta
\end{pmatrix}\\
&\quad + \sqrt{a} \left(\omb\cdot\boldsymbol{a}_3 - \frac{1}{2\sqrt{a}}(\boldsymbol{a}_2\cdot\partial_1\boldsymbol{\eta}-\boldsymbol{a}_1\cdot\partial_2\boldsymbol{\eta})\right)\begin{pmatrix}
    0 & -1 \\ 1 & 0
\end{pmatrix}.
\end{aligned}
\]
The first term in the right hand side is symmetric and recovers the classical covariant membrane strain. It depends only on the displacement. The second term in the right hand side is skew-symmetric and measures the drilling shear. The present low-regularity Naghdi model,
however, uses the full matrix
$\mathbf A(\boldsymbol\eta,\omb)$ and therefore retains also the drilling shear.
}

The transverse shear row is
\[
\boldsymbol e(\boldsymbol\eta,\omb)^T
=
\begin{pmatrix}
\boldsymbol a_3\cdot\partial_1\boldsymbol\eta
+
\sqrt a\,\omb\cdot\boldsymbol a^2
&
\boldsymbol a_3\cdot\partial_2\boldsymbol\eta
-
\sqrt a\,\omb\cdot\boldsymbol a^1
\end{pmatrix}.
\]

For a test pair $(\boldsymbol v,\boldsymbol w)$, define analogously
\[
\widehat{\mathbf D}_{\mathrm{mem}}(\boldsymbol v,\boldsymbol w)
=
\mathbf Q
\begin{bmatrix}
\partial_1\boldsymbol v+\boldsymbol a_1\times\boldsymbol w
&
\partial_2\boldsymbol v+\boldsymbol a_2\times\boldsymbol w
\end{bmatrix}
=
\begin{pmatrix}
\mathbf D(\boldsymbol v,\boldsymbol w)\\
\boldsymbol d(\boldsymbol v,\boldsymbol w)^T
\end{pmatrix},
\]
where $\mathbf D(\boldsymbol v,\boldsymbol w)\in\mathbb R^{2\times2}$
is given by
\[
\mathbf D(\boldsymbol v,\boldsymbol w)
=
\begin{pmatrix}
\boldsymbol a^1\cdot\partial_1\boldsymbol v
&
\boldsymbol a^1\cdot\partial_2\boldsymbol v
+\sqrt a\,a^{11}\boldsymbol w\cdot\boldsymbol a_3
\\
\boldsymbol a^2\cdot\partial_1\boldsymbol v
-\sqrt a\,a^{22}\boldsymbol w\cdot\boldsymbol a_3
&
\boldsymbol a^2\cdot\partial_2\boldsymbol v
\end{pmatrix},
\]
and
\[
\boldsymbol d(\boldsymbol v,\boldsymbol w)^T
=
\begin{pmatrix}
\boldsymbol a_3\cdot\partial_1\boldsymbol v
+
\sqrt a\,\boldsymbol w\cdot\boldsymbol a^2
&
\boldsymbol a_3\cdot\partial_2\boldsymbol v
-
\sqrt a\,\boldsymbol w\cdot\boldsymbol a^1
\end{pmatrix}.
\]

\paragraph{Bending strain}
The bending strain measures the variation of the infinitesimal rotation field along
the surface directions. In Cartesian coordinates, we define the bending quantity
\[
\mathbf E_{\mathrm{bend}}(\omb)
=
\begin{bmatrix}
\partial_1\omb
&
\partial_2\omb
\end{bmatrix}
\in\mathbb R^{3\times2},
\]
then the corresponding local bending quantity is
\[
\widehat{\mathbf E}_{\mathrm{bend}}(\omb)
=
\mathbf Q
\mathbf E_{\mathrm{bend}}(\omb)
= \begin{pmatrix}
\boldsymbol{a}_1 \cdot \partial_1\omb & \boldsymbol{a}_1 \cdot \partial_2\omb \\
\boldsymbol{a}_2 \cdot \partial_1\omb & \boldsymbol{a}_2 \cdot \partial_2\omb \\
\boldsymbol{a}_3 \cdot \partial_1\omb & \boldsymbol{a}_3 \cdot \partial_2\omb
\end{pmatrix}
=
\begin{pmatrix}
\mathbf E_b(\omb)\\
\boldsymbol e_b(\omb)^T
\end{pmatrix},
\]
where $\mathbf E_b(\omb)\in\mathbb R^{2\times2}$ and
$\boldsymbol e_b(\omb)\in\mathbb R^2$. For a test rotation
$\boldsymbol w$, we set
\[
\widehat{\mathbf D}_{\mathrm{bend}}(\boldsymbol w)
=
\mathbf Q
\begin{bmatrix}
\partial_1\boldsymbol w
&
\partial_2\boldsymbol w
\end{bmatrix}
=
\begin{pmatrix}
\mathbf D_b(\boldsymbol w)\\
\boldsymbol d_b(\boldsymbol w)^T
\end{pmatrix}.
\]

\paragraph{Elasticity tensors}
Let $\lambda$ and $\mu$ be the Lam\'e coefficients and assume $\mu>0,~
3\lambda+2\mu>0$.
The membrane-shear elasticity operator is defined through
\begin{equation}
    \begin{aligned}
\mathcal C_m
\widehat{\mathbf E}_{\mathrm{mem}}
(\boldsymbol\eta,\omb)
\cdot
\widehat{\mathbf D}_{\mathrm{mem}}
(\boldsymbol v,\boldsymbol w)
&=
\frac{2\lambda\mu}{\lambda+2\mu}
\operatorname{tr}
\left(
\mathbf A(\boldsymbol\eta,\omb)
\right)
\operatorname{tr}
\left(
\mathbf D(\boldsymbol v,\boldsymbol w)
\right)
\\
&~~
+2\mu
\left(
\boldsymbol a_c
\mathbf A(\boldsymbol\eta,\omb)
\boldsymbol a^c
\right)
:
\mathbf D(\boldsymbol v,\boldsymbol w)
+
\mathcal B_m
\boldsymbol e(\boldsymbol\eta,\omb)
\cdot
\boldsymbol d(\boldsymbol v,\boldsymbol w),
\end{aligned}
\label{membrane_tensor}
\end{equation}
where, for the isotropic low-regularity model,
$\mathcal B_m
=
\mu\boldsymbol a^c$.

For the bending term, let
\[
\mathbf J
=
\begin{pmatrix}
0&-1\\
1&0
\end{pmatrix}.
\]
For $\mathbf S\in\mathbb R^{2\times2}$, define
\[
\mathcal A\mathbf S
=
\frac{2\lambda\mu}{\lambda+2\mu}
\left(
\boldsymbol a^c:\mathbf S
\right)
\boldsymbol a^c
+
2\mu
\boldsymbol a^c
\mathbf S
\boldsymbol a^c.
\]
Then
\[
\begin{aligned}
\mathcal C_b
\widehat{\mathbf E}_{\mathrm{bend}}(\omb)
\cdot
\widehat{\mathbf D}_{\mathrm{bend}}(\boldsymbol w)
=
a\,
\mathcal A
\left(
\mathbf J\mathbf E_b(\omb)
\right)
:
\left(
\mathbf J\mathbf D_b(\boldsymbol w)
\right)
+
a\,
\mathcal B_b
\boldsymbol e_b(\omb)
\cdot
\boldsymbol d_b(\boldsymbol w).
\end{aligned}
\]
Here, the matrix $\mathcal B_b$ is assumed to be symmetric and uniformly positive
definite. The sign convention used in the definition of $\mathbf J$ does not
affect this bilinear form, because $\mathbf J$ occurs in both factors.

\paragraph{Weak shell equation}
Let the clamped part of the shell boundary be
\[
\gamma_0
=
\left(
\{0\}\times(0,2\pi R_0)
\right)
\cup
\left(
\{L\}\times(0,2\pi R_0)
\right).
\]
Define
\[
\begin{aligned}
V_N
=
\big\{
&(\boldsymbol\eta,\omb)
\in H^1(\omg;\mathbb R^3)^2:
\boldsymbol\eta|_{\gamma_0}=\boldsymbol0,
\quad
\omb|_{\gamma_0}=\boldsymbol0,
\\
&\boldsymbol\eta(x,0)
=
\boldsymbol\eta(x,2\pi R_0),
~
\omb(x,0)
=
\omb(x,2\pi R_0)
\quad\text{for a.e. }x\in(0,L)
\big\}.
\end{aligned}
\]
The membrane-shear and bending bilinear forms are
\[
\begin{aligned}
&B_m
\big(
(\boldsymbol\eta,\omb),
(\boldsymbol v,\boldsymbol w)
\big)
=
\varepsilon
\int_\omg
\mathcal C_m
\widehat{\mathbf E}_{\mathrm{mem}}
(\boldsymbol\eta,\omb)
\cdot
\widehat{\mathbf D}_{\mathrm{mem}}
(\boldsymbol v,\boldsymbol w)
\sqrt a
\,dx\,dy_{\mathrm{arc}},
\end{aligned}
\]
and
\[
\begin{aligned}
&B_b
\big(
(\boldsymbol\eta,\omb),
(\boldsymbol v,\boldsymbol w)
\big)
=
\frac{\varepsilon^3}{12}
\int_\omg
\mathcal C_b
\widehat{\mathbf E}_{\mathrm{bend}}(\omb)
\cdot
\widehat{\mathbf D}_{\mathrm{bend}}(\boldsymbol w)
\sqrt a
\,dx\,dy_{\mathrm{arc}}.
\end{aligned}
\]

The weak shell problem is to find
$(\boldsymbol\eta,\omb)\in V_N$ such that, for every
$(\boldsymbol v,\boldsymbol w)\in V_N$,
\begin{equation}
\begin{aligned}
&\varepsilon\rho_s
\int_\omg
\partial_{tt}\boldsymbol\eta
\cdot
\boldsymbol v
\sqrt a
\,dx\,dy_{\mathrm{arc}}
+
\frac{\varepsilon^3\rho_s}{12}
\int_\omg
\partial_{tt}\omb
\cdot
\boldsymbol w
\sqrt a
\,dx\,dy_{\mathrm{arc}}
\\
&
+B_m
\big(
(\boldsymbol\eta,\omb),
(\boldsymbol v,\boldsymbol w)
\big)
+
B_b
\big(
(\boldsymbol\eta,\omb),
(\boldsymbol v,\boldsymbol w)
\big)
=
\int_\omg
\boldsymbol f^\varepsilon
\cdot
\boldsymbol v
\sqrt a
\,dx\,dy_{\mathrm{arc}}.
\end{aligned}
\label{Naghdi_shell}
\end{equation}
Here $\rho_s$ is the shell density and $\boldsymbol f^\varepsilon$ is an
external force per unit physical reference surface area. No distributed
external couple is included.

\paragraph{Boundary conditions}
The matrices $\mathcal{B}_m$ and $\mathcal{B}_b$ are assumed to be symmetric positive definite. The shell is clamped at $x = 0$ and $x = L$:
\[
\boldsymbol{\eta}(0, y_{\mathrm{arc}}, t) = \boldsymbol{\eta}(L, y_{\mathrm{arc}}, t) = \mathbf{0}, \qquad \omb(0, y_{\mathrm{arc}}, t) = \omb(L, y_{\mathrm{arc}}, t) = \mathbf{0}.
\]
Periodic boundary conditions are imposed in the circumferential direction:
\[
\boldsymbol{\eta}(x, 0, t) = \boldsymbol{\eta}(x, 2\pi R, t), \qquad \omb(x, 0, t) = \omb(x, 2\pi R, t).
\]

\paragraph{Initial conditions}
We assume the following compatible initial conditions
\[
\boldsymbol\eta(\cdot,\cdot,0)=\boldsymbol\eta^0,
\qquad
\partial_t\boldsymbol\eta(\cdot,\cdot,0)=\dot{\boldsymbol\eta}^0,
\qquad
\omb(\cdot,\cdot,0)=\omb^0,
\qquad
\partial_t\omb(\cdot,\cdot,0)=\dot{\omb}^0.
\]

\begin{remark}
Because the coefficients of the shell model depend only on first derivatives
of $\varphi$, they are defined almost everywhere and remain bounded under the
assumptions of Section~\ref{shelldomain}. No separate boundary condition is
imposed at an internal band-transition curve. The global
$H^1(\omg;\mathbb R^3)^2$ regularity of
$(\boldsymbol\eta,\omb)$ enforces continuity across each such
curve in the trace sense and therefore represents a connected, welded shell.
\end{remark}

\subsection{Fluid-structure interface coupling conditions}
\label{interface_conditions}

Although the shell displacement belongs to an $H^1$ space, not every
such displacement produces a geometrically admissible current interface:
the resulting surface could self-intersect, fold onto itself, or collapse
locally. We therefore restrict the coupled problem to displacements for
which the current shell map
\[
\Phi_s(\cdot,t):\omg\rightarrow\Gamma(t)
\]
has a bi-Lipschitz representative onto its image, after identifying the
periodic boundaries $y_{\mathrm{arc}}=0$ and
$y_{\mathrm{arc}}=2\pi R_0$. The bi-Lipschitz property prevents distinct
shell points from being mapped to the same spatial point and guarantees
that the inverse surface map $\Phi_s(\cdot,t)^{-1}$ is well defined.

We further assume that $\Phi_s$ preserves the orientation inherited from
the reference map $\varphi$ and that its surface Jacobian
$
J_s(x,y_{\mathrm{arc}},t)$
is bounded away from zero almost everywhere \eqref{nondegeneracy}.
The orientation preservation ensures that the interface normal has a
consistent direction relative to the reference configuration. The lower
bound on $J_s$ prevents the tangent vectors
$\partial_1\Phi_s$ and $\partial_2\Phi_s$ from becoming linearly
dependent and therefore prevents the surface from collapsing locally.
Together, these assumptions ensure that the current interface, its unit
normal, the inverse surface map, and the surface change-of-variables
formula used below are well defined.

\paragraph{Geometric coupling}
Given the reference fluid-structure interface is
$\hat\Gamma_s=\varphi(\omg)$. The geometric condition states that the fluid-mesh boundary and shell midsurface occupy the same current location
\[
\hat\varphi_f
\left(
\varphi(x,y_{\mathrm{arc}}),t
\right)
=
\Phi_s(x,y_{\mathrm{arc}},t),
\qquad
(x,y_{\mathrm{arc}})\in\omg,
\quad t\in(0,T).
\]
Since
$\hat\varphi_f(\hat{\boldsymbol x},t)
=\hat{\boldsymbol x}+\hat{\boldsymbol\eta}_f(\hat{\boldsymbol x},t)$
and
$\Phi_s=\varphi+\boldsymbol\eta$, this condition is equivalent to
\[
\hat{\boldsymbol\eta}_f
\left(
\varphi(x,y_{\mathrm{arc}}),t
\right)
=
\boldsymbol\eta(x,y_{\mathrm{arc}},t).
\]
Since the reference map $\varphi$ is independent of time,
differentiating the geometric coupling condition at fixed
$(x,y_{\mathrm{arc}})$ gives
\[
\partial_t\hat\varphi_f
\left(
\varphi(x,y_{\mathrm{arc}}),t
\right)
=
\partial_t\Phi_s(x,y_{\mathrm{arc}},t)
=
\partial_t\boldsymbol\eta(x,y_{\mathrm{arc}},t).
\]
Using the definition of the current fluid-mesh velocity, this is
equivalent to
\[
\boldsymbol w_f
\left(
\Phi_s(x,y_{\mathrm{arc}},t),t
\right)
=
\partial_t\boldsymbol\eta(x,y_{\mathrm{arc}},t).
\]

\paragraph{Kinematic coupling}
The no-slip
condition requires the fluid velocity to equal the velocity of the
current shell midsurface
\[
\begin{aligned}
&\boldsymbol u
\left(
\Phi_s(x,y_{\mathrm{arc}},t),t
\right)
=
\partial_t\boldsymbol\eta(x,y_{\mathrm{arc}},t), \qquad
(x,y_{\mathrm{arc}})\in\omg,
\quad t\in(0,T).
\end{aligned}
\]
Equivalently,
\[
\boldsymbol u(\boldsymbol x,t)
=
\partial_t\boldsymbol\eta
\left(
\Phi_s(\cdot,t)^{-1}(\boldsymbol x),t
\right),
\qquad
\boldsymbol x\in\Gamma(t),
\quad t\in(0,T).
\]
Notice that, although the full Naghdi shell unknowns are
$(\boldsymbol\eta,\omb)$, there is no
independent fluid kinematic condition for the rotation field
$\omb$.

\paragraph{Dynamic coupling}

The orientation chosen for $\boldsymbol a_3$ in
Section~\ref{shelldomain} and retained in the Naghdi shell formulation gives
the current outward unit normal to the fluid domain as
\[
\boldsymbol n_f
\left(
\Phi_s(x,y_{\mathrm{arc}},t),t
\right)
=
-
\frac{
\partial_1\Phi_s(x,y_{\mathrm{arc}},t)
\times
\partial_2\Phi_s(x,y_{\mathrm{arc}},t)
}{
J_s(x,y_{\mathrm{arc}},t)
}.
\]
In particular, in the reference configuration,
\[
J_s\big|_{\boldsymbol\eta=\boldsymbol0}
=
\sqrt a,
\qquad
\boldsymbol n_f^0
=
-\boldsymbol a_3.
\]
Let $\boldsymbol t_f$ denote the traction vector appearing in the fluid
momentum balance on the interface:
\[
\boldsymbol t_f(\boldsymbol x,t)
=
\boldsymbol\sigma_f
\left(
\boldsymbol u(\boldsymbol x,t),p(\boldsymbol x,t)
\right)
\boldsymbol n_f(\boldsymbol x,t),
\qquad
\boldsymbol x\in\Gamma(t).
\]
Then, the force per unit current surface area
exerted by the fluid on the shell is
\[
\boldsymbol f_s^\Gamma(\boldsymbol x,t)
=
-\boldsymbol t_f(\boldsymbol x,t)
=
-\boldsymbol\sigma_f
\left(
\boldsymbol u(\boldsymbol x,t),p(\boldsymbol x,t)
\right)
\boldsymbol n_f(\boldsymbol x,t),
\qquad
\boldsymbol x\in\Gamma(t).
\]
Let $\boldsymbol v^\Gamma$ denote the test function for the shell displacement on the current
interface
\[
\boldsymbol v^\Gamma
\left(
\Phi_s(x,y_{\mathrm{arc}},t),t
\right)
=
\boldsymbol v(x,y_{\mathrm{arc}}),
\]
where we recall that $\boldsymbol{v}$ is a test function on the reference Naghdi shell domain. 
The fluid test function $\boldsymbol{\tilde v}$ and the shell test function are compatible when
\[
\boldsymbol{\tilde v}
=
\boldsymbol v^\Gamma
\qquad
\text{on }\Gamma(t).
\]
Under this compatibility condition, the work at the interface should be balanced, namely
\[
\int_{\Gamma(t)}
\boldsymbol\sigma_f(\boldsymbol u,p)
\boldsymbol n_f
\cdot
\boldsymbol{\tilde v}
\,dS_t
+
\int_{\Gamma(t)}
\boldsymbol f_s^\Gamma
\cdot
\boldsymbol v^\Gamma
\,dS_t
=0.
\]
Next, we can pull the shell-side interface work back to $\omg$, which gives
\[
\begin{aligned}
&\int_{\Gamma(t)}
\boldsymbol f_s^\Gamma
\cdot
\boldsymbol v^\Gamma
\,dS_t
=
\int_\omg
\boldsymbol f_s^\Gamma
\left(
\Phi_s(x,y_{\mathrm{arc}},t),t
\right)
\cdot
\boldsymbol v(x,y_{\mathrm{arc}})
J_s(x,y_{\mathrm{arc}},t)
\,dx\,dy_{\mathrm{arc}}.
\end{aligned}
\]
Therefore, we can define the pulled-back shell load per unit parameter area by
\[
\boldsymbol f_s^{\mathrm{ref}}(x,y_{\mathrm{arc}},t)
=
\boldsymbol f_s^\Gamma
\left(
\Phi_s(x,y_{\mathrm{arc}},t),t
\right)
J_s(x,y_{\mathrm{arc}},t).
\]
Equivalently,
\[
\begin{aligned}
\boldsymbol f_s^{\mathrm{ref}}(x,y_{\mathrm{arc}},t)
={}&-
\boldsymbol\sigma_f
\Bigl(
\boldsymbol u
\left(
\Phi_s(x,y_{\mathrm{arc}},t),t
\right),
p
\left(
\Phi_s(x,y_{\mathrm{arc}},t),t
\right)
\Bigr)
\boldsymbol n_f
\Bigl(
\Phi_s(x,y_{\mathrm{arc}},t),t
\Bigr)
J_s(x,y_{\mathrm{arc}},t).
\end{aligned}
\]
\begin{remark}
Recall that the surface force $\boldsymbol f_s^{\mathrm{ref}}$ is related to the surface force $\boldsymbol f^\varepsilon$ appearing in  the shell equation \eqref{Naghdi_shell} as follows:
\[
\boldsymbol f_s^{\mathrm{ref}}
=
\boldsymbol f^\varepsilon\sqrt a,
\]
where $a$ is defined in \eqref{def:a}, and 
$
dS_0
=
\sqrt a\,dx\,dy_{\mathrm{arc}}
$.


Therefore, the fluid induced force appearing in the shell equation is
\[
\begin{aligned}
\boldsymbol f^\varepsilon(x,y_{\mathrm{arc}},t)
={}&-
\boldsymbol\sigma_f
\Bigl(
\boldsymbol u
\left(
\Phi_s(x,y_{\mathrm{arc}},t),t
\right),
p
\left(
\Phi_s(x,y_{\mathrm{arc}},t),t
\right)
\Bigr)
\boldsymbol n_f
\Bigl(
\Phi_s(x,y_{\mathrm{arc}},t),t
\Bigr)
\frac{J_s(x,y_{\mathrm{arc}},t)}{\sqrt a}.
\end{aligned}
\]
Consequently,
\[
\begin{aligned}
\int_\omg
\boldsymbol f^\varepsilon
\cdot
\boldsymbol v
\sqrt a
\,dx\,dy_{\mathrm{arc}}
=
\int_\omg
\boldsymbol f_s^{\mathrm{ref}}
\cdot
\boldsymbol v
\,dx\,dy_{\mathrm{arc}}
=
\int_{\Gamma(t)}
\boldsymbol f_s^\Gamma
\cdot
\boldsymbol v^\Gamma
\,dS_t.
\end{aligned}
\]
\end{remark}

\begin{remark}
In a general Naghdi shell model, the external interface virtual work may
contain contributions from both a distributed surface force and a
distributed surface couple:
\[
\delta W_{\Gamma}
=
\int_{\Gamma(t)}
\boldsymbol f_s^\Gamma
\cdot
\boldsymbol v^\Gamma
\,dS_t
+
\int_{\Gamma(t)}
\boldsymbol m_s^\Gamma
\cdot
\boldsymbol w^\Gamma
\,dS_t,
\]
where $\boldsymbol m_s^\Gamma$ is the couple per unit current surface
area, and
$
\boldsymbol w^\Gamma\circ\Phi_s=\boldsymbol w$.
In the present midsurface-coupling approximation, the fluid traction is
applied directly to the shell midsurface and is represented only by the
surface-force resultant
$\boldsymbol f_s^\Gamma
=
-\boldsymbol\sigma_f(\boldsymbol u,p)\boldsymbol n_f$.
We neglect an independent moment resultant, including any moment produced
by the offset between the physical fluid-contacting surface and the shell
midsurface, and therefore set
\[
\boldsymbol m_s^\Gamma=\boldsymbol0.
\]
Consequently, the interface virtual work reduces to
\[
\delta W_{\Gamma}
=
\int_{\Gamma(t)}
\boldsymbol f_s^\Gamma
\cdot
\boldsymbol v^\Gamma
\,dS_t
=
\int_\omg
\boldsymbol f_s^{\mathrm{ref}}
\cdot
\boldsymbol v
\,dx\,dy_{\mathrm{arc}},
\]
and no direct interface-work term involving the rotation test function
$\boldsymbol w$ appears. Nevertheless, this does not mean that $\omb$ remains unchanged. The applied force can still generate rotation indirectly because $\boldsymbol\eta$ and $\omb$ are coupled through the shell bilinear forms. It only means that the fluid does not apply a separate external couple directly to $\omb$.
\end{remark}

\begin{remark}
At a band-transition curve, $R'$ may jump. Consequently, the tangent
vectors and unit normal may have different one-sided values, so a single
classical normal may not exist exactly on the transition curve.
Nevertheless, these quantities are defined almost everywhere on each
smooth part of the surface.

A transition curve is one-dimensional, whereas the shell surface is
two-dimensional, so the curve has zero surface measure. Therefore, it
does not contribute to an ordinary surface integral such as
\[
\int_{\Gamma(t)}
\boldsymbol f_s^\Gamma
\cdot
\boldsymbol v^\Gamma
\,dS_t.
\]
No additional condition or integral is required there for a regular
distributed traction.
If the physical model includes a force or couple concentrated along a
band-transition curve $\mathcal C_j$, its virtual work must instead be
included explicitly through line terms of the form
\[
\sum_j
\int_{\mathcal C_j}
\boldsymbol f_{\mathrm{line},j}
\cdot
\boldsymbol v
\,ds
+
\sum_j
\int_{\mathcal C_j}
\boldsymbol m_{\mathrm{line},j}
\cdot
\boldsymbol w
\,ds.
\]
Such concentrated line forces $\boldsymbol f_{\mathrm{line},j}$ and line couples $\boldsymbol m_{\mathrm{line},j}$ are not included in our
present model.
\end{remark}

\section{Weak Formulation}
\label{sec:weak-formulation}

We now formulate the coupled fluid-shell-Windkessel problem in weak
form. The fluid equations are posed on the current time-dependent domain
$\Omega_f(t)$, whereas the shell equations are posed on the fixed
parameter domain $\omg$. The prescribed inlet velocity is imposed strongly, as is
the kinematic coupling condition on the fluid-structure interface. The dynamic coupling condition
is incorporated through the interface tractions. When the fluid and shell
equations are tested with functions having compatible interface traces,
their interface work contributions cancel, yielding the combined
variational formulation without an explicit interface work term.

\paragraph{Function spaces}
Because the inlet velocity is prescribed, the fluid trial set is affine. We
define
\[
\begin{aligned}
V_f(t)
=
\big\{
&\boldsymbol u:\Omega_f(t)\rightarrow\mathbb R^3,
\quad
\boldsymbol u
=
\hat{\boldsymbol u}
\circ
\hat\varphi_f(\cdot,t)^{-1}
:
\hat{\boldsymbol u}
\in
H^1(\hat\Omega_f;\mathbb R^3),
\\
&
\boldsymbol u=\boldsymbol u_{\mathrm{in}}
\text{ on }\Gamma_f^{\mathrm{in}}(t),
\quad
\boldsymbol u=\boldsymbol0
\text{ on }
\Sigma_f^D(t)
\setminus
\overline{\Gamma_f^{\mathrm{in}}(t)}
\big\}.
\end{aligned}
\]
The corresponding homogeneous fluid test space is
\[
\begin{aligned}
V_{f,0}(t)
=
\big\{
&\boldsymbol{\tilde v}:\Omega_f(t)\rightarrow\mathbb R^3,
\quad
\boldsymbol{\tilde v}
=
\hat{\boldsymbol v}
\circ
\hat\varphi_f(\cdot,t)^{-1}
:
\hat{\boldsymbol v}
\in
H^1(\hat\Omega_f;\mathbb R^3),\quad
\boldsymbol{\tilde v}=\boldsymbol0
\text{ on }\Sigma_f^D(t)
\big\}.
\end{aligned}
\]
The pressure space is
\[
Q_f(t)
=
\left\{
q:\Omega_f(t)\rightarrow\mathbb R,
\quad
q
=
\hat q
\circ
\hat\varphi_f(\cdot,t)^{-1}
:
\hat q\in L^2(\hat\Omega_f)
\right\}.
\]
For the shell displacement and rotation, we use the space $V_N$
defined in the Naghdi shell model section \ref{NaghdiModel}, including its clamped and
circumferential periodicity conditions. We continue to denote the shell displacement
velocity and rotational velocity as
$\dot{\boldsymbol\eta}
=
\partial_t\boldsymbol\eta,
~
\dot{\omb}
=
\partial_t\omb$.

The coupled trial space is
\[
\begin{aligned}
\mathcal V(t)
=
\big\{
&(\boldsymbol u(t),\boldsymbol\eta(t),\omb(t))
 \in V_f(t)\times V_N:
~
\boldsymbol u
\left(
\Phi_s(x,y_{\mathrm{arc}},t),t
\right)
=
\partial_t\boldsymbol\eta(x,y_{\mathrm{arc}},t)
\text{ for }
(x,y_{\mathrm{arc}})\in\omg
\big\}.
\end{aligned}
\]
The corresponding coupled test space is
\[
\begin{aligned}
\mathcal V_0(t)
=
\big\{
&(\boldsymbol{\tilde v}(t),\boldsymbol v,\boldsymbol w)
\in V_{f,0}(t) \times V_N:
~\boldsymbol{\tilde v}
\left(
\Phi_s(x,y_{\mathrm{arc}},t),t
\right)
=
\boldsymbol v(x,y_{\mathrm{arc}})
\text{ for }
(x,y_{\mathrm{arc}})\in\omg
\big\}.
\end{aligned}
\]
We note that the interface equalities in these definitions are understood in the trace
sense. Moreover, the compatibility condition imposed on the test functions is the
homogeneous counterpart of the strong kinematic coupling condition.
Because the fluid and shell share a single translational velocity on
the interface, their admissible virtual velocities must also have the same interface trace. 

\paragraph{Fluid weak form}
For
$\boldsymbol{\tilde v}(t)\in V_{f,0}(t)$ and $q(t)\in Q_f(t)$, multiplying the
equation (\ref{Navier-Stokes}) by
$\boldsymbol{\tilde v}(t)$ and integrating over $\Omega_f(t)$ gives
\begin{equation}
    \begin{aligned}
&\rho_f
\left(
\left.\partial_t\boldsymbol u\right|_{\hat\Omega_f},
\boldsymbol{\tilde v}
\right)_{\Omega_f(t)}
+
\rho_f
\left(
\big((\boldsymbol u-\boldsymbol w_f)\cdot\nabla\big)\boldsymbol u,
\boldsymbol{\tilde v}
\right)_{\Omega_f(t)}
\\
&\quad
+
2\mu_f
\left(
\mathbb D(\boldsymbol u),
\mathbb D(\boldsymbol{\tilde v})
\right)_{\Omega_f(t)}
-
\left(
p,
\nabla\cdot\boldsymbol{\tilde v}
\right)_{\Omega_f(t)}
+\left(
\nabla\cdot\boldsymbol u,
q
\right)_{\Omega_f(t)}
\\
&\quad=
\left(
\boldsymbol F_f,
\boldsymbol{\tilde v}
\right)_{\Omega_f(t)}
+
\left\langle
\boldsymbol\sigma_f(\boldsymbol u,p)\boldsymbol n_f,
\boldsymbol{\tilde v}
\right\rangle_{\Gamma(t)}
-
\left\langle
p_{\mathrm{out}}(t)\boldsymbol n_f,
\boldsymbol{\tilde v}
\right\rangle_{\Gamma_f^{\mathrm{out}}(t)}.
\end{aligned}
\label{fluid_weak_form}
\end{equation}
Here $(\cdot,\cdot)_{\Omega_f(t)}$ denotes the volume $L^2$ inner
product, and $\langle\cdot,\cdot\rangle$ denotes the appropriate surface
duality pairing, or the corresponding surface integral for sufficiently
regular fields. There is no inlet-traction term because
$\boldsymbol{\tilde v}=\boldsymbol0$ on
$\Gamma_f^{\mathrm{in}}(t)$. 

\paragraph{Shell weak form}
Using the shell model in equation~\eqref{Naghdi_shell}, 
we obtain that for every test function 
$(\boldsymbol v,\boldsymbol w)\in V_N$ the weak solution
$(\boldsymbol\eta(t),\omb(t))\in V_N$
must satisfy
\begin{equation}
\begin{aligned}
&\varepsilon\rho_s
\int_\omg
\partial_{tt}\boldsymbol\eta
\cdot
\boldsymbol v
\sqrt a
\,dx\,dy_{\mathrm{arc}}
+
\frac{\varepsilon^3\rho_s}{12}
\int_\omg
\partial_{tt}\omb
\cdot
\boldsymbol w
\sqrt a
\,dx\,dy_{\mathrm{arc}}
\\
&\quad
+
B_m
\big(
(\boldsymbol\eta,\omb),
(\boldsymbol v,\boldsymbol w)
\big)
+
B_b
\big(
(\boldsymbol\eta,\omb),
(\boldsymbol v,\boldsymbol w)
\big)
=
\int_\omg
\boldsymbol f_s^{\mathrm{ref}}
\cdot
\boldsymbol v
\,dx\,dy_{\mathrm{arc}}.
\end{aligned}
\label{shell_weak_form}
\end{equation}
The pulled-back shell load
$\boldsymbol f_s^{\mathrm{ref}}
=\boldsymbol f^\varepsilon\sqrt a$, as established in
Section~\ref{interface_conditions}. The factor $\sqrt a$ has already been
included in the definitions of $B_m$ and $B_b$.

\paragraph{Coupled weak formulation}
For compatible test functions, the dynamic coupling condition from
Section~\ref{interface_conditions} gives
\[
\left\langle
\boldsymbol\sigma_f(\boldsymbol u,p)\boldsymbol n_f,
\boldsymbol{\tilde v}
\right\rangle_{\Gamma(t)}
+
\int_\omg
\boldsymbol f_s^{\mathrm{ref}}
\cdot
\boldsymbol v
\,dx\,dy_{\mathrm{arc}}
=0.
\]
Hence the explicit interface-work terms cancel when the fluid and shell
equations are added. 

Therefore, by adding \eqref{fluid_weak_form} and \eqref{shell_weak_form} the coupled weak formulation reads as follows. Find a scalar
$P_B\in H^1(0,T)$
and time-dependent fields
$(\boldsymbol u,\boldsymbol\eta,\omb,p)$
such that, for almost every $t\in(0,T)$,
$
(\boldsymbol u(t),\boldsymbol\eta(t),\omb(t))
\in\mathcal V(t),
~
p(t)\in Q_f(t),
$
and for all
$(\boldsymbol{\tilde v}(t),\boldsymbol v,\boldsymbol w)
\in\mathcal V_0(t),~
q(t)\in Q_f(t)$, we have
\begin{equation}
\begin{aligned}
&\rho_f
\left(
\left.\partial_t\boldsymbol u\right|_{\hat\Omega_f},
\boldsymbol{\tilde v}
\right)_{\Omega_f(t)}
+
\rho_f
\left(
\big((\boldsymbol u-\boldsymbol w_f)\cdot\nabla\big)\boldsymbol u,
\boldsymbol{\tilde v}
\right)_{\Omega_f(t)}
\\
&\quad
+
2\mu_f
\left(
\mathbb D(\boldsymbol u),
\mathbb D(\boldsymbol{\tilde v})
\right)_{\Omega_f(t)}
-
\left(
p,
\nabla\cdot\boldsymbol{\tilde v}
\right)_{\Omega_f(t)}
+
\left(
\nabla\cdot\boldsymbol u,
q
\right)_{\Omega_f(t)}
+
\left\langle
\left(
P_B(t)+R_pQ(t)
\right)
\boldsymbol n_f,
\boldsymbol{\tilde v}
\right\rangle_{\Gamma_f^{\mathrm{out}}(t)}
\\
&\quad
+
\varepsilon\rho_s
\int_\omg
\partial_{tt}\boldsymbol\eta
\cdot
\boldsymbol v
\sqrt a
\,dx\,dy_{\mathrm{arc}}
+
\frac{\varepsilon^3\rho_s}{12}
\int_\omg
\partial_{tt}\omb
\cdot
\boldsymbol w
\sqrt a
\,dx\,dy_{\mathrm{arc}}
\\
&\quad
+
B_m
\big(
(\boldsymbol\eta,\omb),
(\boldsymbol v,\boldsymbol w)
\big)
+
B_b
\big(
(\boldsymbol\eta,\omb),
(\boldsymbol v,\boldsymbol w)
\big)
=
\left(
\boldsymbol F_f,
\boldsymbol{\tilde v}
\right)_{\Omega_f(t)}.
\end{aligned}
\label{coupled}
\end{equation}
This variational equation is supplemented by
\[
Q(t)
=
\int_{\Gamma_f^{\mathrm{out}}(t)}
\boldsymbol u\cdot\boldsymbol n_f
\,dS_t,
\qquad
C\dot P_B(t)
+
\frac{P_B(t)-P_v}{R_d}
=
Q(t).
\]

\section{The loosely coupled splitting scheme}
\label{sec:betascheme}
We now describe a loosely coupled time-splitting scheme for the coupled fluid-shell system. Let $\Delta t>0$ be the time step and set $t^n=n\Delta t, ~n=0,\ldots,N,~
N\Delta t=T$. We denote by $\boldsymbol u^n, p^n, P_B^n, \boldsymbol\eta^n, \dot{\boldsymbol\eta}^n, \omb^n, \dot{\omb}^n$ the solutions at time $t^n$. The current shell map and interface at time $t^n$ are 
$$
\Phi_s^n(x,y_{\mathrm{arc}}) = \varphi(x,y_{\mathrm{arc}}) + \boldsymbol\eta^n(x,y_{\mathrm{arc}}), \quad \Gamma^n=\Phi_s^n(\omg),
\quad
J_s^n
=
\left|
\partial_1\Phi_s^n
\times
\partial_2\Phi_s^n
\right|.$$  
At each completed time level, we write
\[
\boldsymbol\xi^n
=
\left(
\boldsymbol\eta^n,
\omb^n
\right),
\qquad
\dot{\boldsymbol\xi}^n
=
\left(
\dot{\boldsymbol\eta}^n,
\dot{\omb}^n
\right).
\]
where
$\dot{\boldsymbol\eta}^n
=
\partial_t\boldsymbol\eta(t^n),
~
\dot{\omb}^n
=
\partial_t\omb(t^n)$.
Given the solution at time $t^n$, the variables at time $t^{n+1}$ are
computed through the following three substeps.

\subsection*{Step 1: Solve the structure subproblem}

The force per unit current surface area exerted by the fluid on the shell is
\[
\boldsymbol f_s^{\Gamma,n}
=
-\boldsymbol t_f^n
=
-\boldsymbol\sigma_f
\left(
\boldsymbol u^n(\boldsymbol x),
p^n(\boldsymbol x)
\right)
\boldsymbol n_f^n(\boldsymbol x),
\qquad
\boldsymbol x\in\Gamma^n.
\]
Its pullback to the fixed shell parameter domain is
\begin{equation}
\begin{aligned}
\boldsymbol f_s^{\mathrm{ref},n}
&=
\left(
\boldsymbol f_s^{\Gamma,n}
\circ
\Phi_s^n
\right)
J_s^n
=
-
\left(
\boldsymbol t_f^n
\circ
\Phi_s^n
\right)
J_s^n
=
\boldsymbol f^{\varepsilon,n}\sqrt a.
\end{aligned}
\label{discrete_pulled_back_load}
\end{equation}
The structure subproblem applies the fraction
$\beta\in[0,1]$ of this previously computed shell load. It produces the
updated shell displacement and rotation, the structure-stage translational
velocity $\dot{\boldsymbol\eta}^{n+1}$, and the updated
rotational velocity $\dot{\omb}^{n+1}$. Define the completed fluid-interface velocity by
\[
\boldsymbol u_\Gamma^n
=
\boldsymbol u^n\circ\Phi_s^n.
\]
The discrete kinematic condition at every completed time level is
\begin{equation}
\dot{\boldsymbol\eta}^n
=
\boldsymbol u_\Gamma^n
=
\boldsymbol u^n\circ\Phi_s^n.
\label{completed_velocity_at_n}
\end{equation}
Thus, the translational velocity supplied to the next structure subproblem
is the completed interface velocity obtained from the preceding fluid
subproblem.

Find
$\boldsymbol\eta^{n+1},~
\omb^{n+1},~
\dot{\boldsymbol\eta}^{n+1},~
\dot{\omb}^{n+1}$
such that, for every
$(\boldsymbol v,\boldsymbol w)\in V_N$,
\begin{equation}
\begin{aligned}
&\varepsilon\rho_s
\int_\omg
\frac{
\dot{\boldsymbol\eta}^{n+1}
-
\boldsymbol u_\Gamma^n
}{\Delta t}
\cdot
\boldsymbol v
\sqrt a
\,dx\,dy_{\mathrm{arc}}
+
\frac{\varepsilon^3\rho_s}{12}
\int_\omg
\frac{
\dot{\omb}^{n+1}
-
\dot{\omb}^n
}{\Delta t}
\cdot
\boldsymbol w
\sqrt a
\,dx\,dy_{\mathrm{arc}}
\\
&\quad
+
B_m
\big(
(\boldsymbol\eta^{n+1},\omb^{n+1}),
(\boldsymbol v,\boldsymbol w)
\big)
+
B_b
\big(
(\boldsymbol\eta^{n+1},\omb^{n+1}),
(\boldsymbol v,\boldsymbol w)
\big)
=-\beta
\int_\omg
\left(
\boldsymbol t_f^n
\circ
\Phi_s^n
\right)
J_s^n
\cdot
\boldsymbol v
\,dx\,dy_{\mathrm{arc}}.
\end{aligned}
\label{structure_beta_subproblem}
\end{equation}
The shell displacement and rotation are updated by
\begin{equation}
\boldsymbol\eta^{n+1}
=
\boldsymbol\eta^n
+
\Delta t\,
\dot{\boldsymbol\eta}^{n+1},
\qquad
\omb^{n+1}
=
\omb^n
+
\Delta t\,
\dot{\omb}^{n+1}.
\label{structure_stage_update}
\end{equation}
These update relations are part of the structure subproblem and close
equation~\eqref{structure_beta_subproblem}. The updated shell map, interface, and surface Jacobian are
\[
\Phi_s^{n+1}
=
\varphi+\boldsymbol\eta^{n+1},
\qquad
\Gamma^{n+1}
=
\Phi_s^{n+1}(\omg),
\]
and
\[
J_s^{n+1}
=
\left|
\partial_1\Phi_s^{n+1}
\times
\partial_2\Phi_s^{n+1}
\right|.
\]
We note that the calculation of angular velocity $\dot{\omb}^{n+1}$ is completed in this
subproblem. It is not modified by the fluid subproblem because there is no
independent fluid kinematic condition or distributed fluid couple for the
rotation field.

\subsection*{Step 2: Calculate the ALE update}

Let
$\hat\Gamma_s=\varphi(\omg)$
denote the reference fluid-structure interface. We determine the fluid mesh movement
$\hat{\boldsymbol\eta}_f^{n+1}$ by solving the harmonic mesh-extension
problem in the reference fluid domain, given by \eqref{adaptive_ALE_ext}.
Then we can update the ALE map and fluid domain by
\[
\hat\varphi_f^{n+1}
=
\hat{\boldsymbol x}
+
\hat{\boldsymbol\eta}_f^{n+1},
\qquad
\Omega_f^{n+1}
=
\hat\varphi_f^{n+1}(\hat\Omega_f).
\]

The discrete mesh velocity, represented on $\Omega_f^{n+1}$, can be obtained by
\[
\boldsymbol w_f^{n+1}
=
\left(
\frac{
\hat\varphi_f^{n+1}
-
\hat\varphi_f^n
}{\Delta t}
\right)
\circ
\left(
\hat\varphi_f^{n+1}
\right)^{-1}.
\]
Geometric compatibility at the two consecutive time levels gives
\[
\hat\varphi_f^k\circ\varphi
=
\Phi_s^k,
\qquad
k=n,n+1.
\]
In particular,
\[
\left(
\hat\varphi_f^{n+1}
\right)^{-1}
\circ
\Phi_s^{n+1}
=
\varphi.
\]
Therefore, on the updated fluid-structure interface,
\[
\begin{aligned}
\boldsymbol w_f^{n+1}
\circ
\Phi_s^{n+1}
&=
\left(
\frac{
\hat\varphi_f^{n+1}
-
\hat\varphi_f^n
}{\Delta t}
\right)
\circ
\left(
\hat\varphi_f^{n+1}
\right)^{-1}
\circ
\Phi_s^{n+1}
=
\frac{
\hat\varphi_f^{n+1}\circ\varphi
-
\hat\varphi_f^n\circ\varphi
}{\Delta t}
\\
&=
\frac{
\Phi_s^{n+1}
-
\Phi_s^n
}{\Delta t}
=
\frac{
\boldsymbol\eta^{n+1}
-
\boldsymbol\eta^n
}{\Delta t}
=
\dot{\boldsymbol\eta}^{n+1}.
\end{aligned}
\]
Thus, the updated fluid mesh follows the structure-stage velocity. It is worth to mention that this
identity concerns the mesh motion only. The completed translational shell
velocity is determined from the fluid trace in Step~3.

\subsection*{Step 3: Solve the Fluid subproblem}

The updated fluid domain $\Omega_f^{n+1}$ and mesh velocity
$\boldsymbol w_f^{n+1}$ obtained in Step~2 are held fixed during the fluid
subproblem.
The fluid velocity time derivative is discretized at the reference coordinates. For
$\hat{\boldsymbol x}\in\hat\Omega_f$, we define the following notation
\[
\hat{\boldsymbol u}^k(\hat{\boldsymbol x})
=
\boldsymbol u^k
\left(
\hat\varphi_f^k(\hat{\boldsymbol x})
\right),
\qquad
k=n,n+1.
\]
The backward Euler discretization of the ALE time derivative is 
\begin{equation}
\left(
\left.
\delta_t\boldsymbol u^{n+1}
\right|_{\hat\Omega_f}
\right)
\circ
\hat\varphi_f^{n+1}
=
\frac{
\hat{\boldsymbol u}^{n+1}
-
\hat{\boldsymbol u}^n
}{\Delta t}.
\label{discrete_ALE_derivative}
\end{equation}
Equivalently,
\[
\left.
\delta_t\boldsymbol u^{n+1}
\right|_{\hat\Omega_f}
=
\left(
\frac{
\boldsymbol u^{n+1}\circ\hat\varphi_f^{n+1}
-
\boldsymbol u^n\circ\hat\varphi_f^n
}{\Delta t}
\right)
\circ
\left(
\hat\varphi_f^{n+1}
\right)^{-1}.
\]
Foe the convective term, we treat it in a semi-implicit fashion, as follows
$t^{n+1}$:
\[
\Big(
\big(
(\boldsymbol u-\boldsymbol w_f)\cdot\nabla
\big)
\boldsymbol u
\Big)^{n+1}
\approx
\left(
\left(
\boldsymbol u^n
-
\boldsymbol w_f^{n+1}
\right)
\cdot\nabla
\right)
\boldsymbol u^{n+1}.
\]
Whenever $\boldsymbol u^n$ appears in an integral over the current domain
$\Omega_f^{n+1}$, it is treated as the value of the previous fluid
velocity at the same ALE reference coordinate, namely
\[
\boldsymbol u^n
\left(
\hat\varphi_f^{n+1}(\hat{\boldsymbol x})
\right)
:=
\boldsymbol u^n
\left(
\hat\varphi_f^n(\hat{\boldsymbol x})
\right),
\qquad
\hat{\boldsymbol x}\in\hat\Omega_f.
\]
For the fluid velocity defined on $\Omega_f^{n+1}$, we define its trace pulled back to the fixed shell parameter domain by
\[
\boldsymbol u_\Gamma^{n+1}
=
\boldsymbol u^{n+1}\circ\Phi_s^{n+1},
\qquad
\boldsymbol{\tilde v}_\Gamma^{n+1}
=
\boldsymbol{\tilde v}\circ\Phi_s^{n+1}.
\]
And the new fluid traction on the current fluid-structure interface $\Gamma^{n+1}$ is
\[
\boldsymbol t_f^{n+1}
=
\boldsymbol\sigma_f
\left(
\boldsymbol u^{n+1},
p^{n+1}
\right)
\boldsymbol n_f^{n+1}
\qquad
\text{on }\Gamma^{n+1}.
\]

For our loosely coupled $\beta$ splitting scheme, the fluid subproblem incorporates the remaining translational shell
inertia through a Robin interface condition as follows
\begin{equation}
\begin{aligned}
&
\frac{\varepsilon\rho_s\sqrt a}{\Delta t}
\left(
\boldsymbol u_\Gamma^{n+1}
-
\dot{\boldsymbol\eta}^{n+1}
\right)
+
\left(
\boldsymbol t_f^{n+1}
\circ
\Phi_s^{n+1}
\right)
J_s^{n+1}
=
\beta
\left(
\boldsymbol t_f^n
\circ
\Phi_s^n
\right)
J_s^n
\qquad
\text{on }\omg.
\end{aligned}
\label{discrete_beta_robin_condition}
\end{equation}
Therefore, for every
$\boldsymbol{\tilde v}\in V_{f,0}(t^{n+1})$, the new interface-traction
work is represented by
\[
\begin{aligned}
\left\langle
\boldsymbol\sigma_f(\boldsymbol u,p)^{n+1}\boldsymbol n_f^{n+1},
\boldsymbol{\tilde v}
\right\rangle_{\Gamma^{n+1}}
&=
\left\langle
\boldsymbol t_f^{n+1},
\boldsymbol{\tilde v}
\right\rangle_{\Gamma^{n+1}}
\\
&=
-
\frac{\varepsilon\rho_s}{\Delta t}
\int_\omg
\left(
\boldsymbol u_\Gamma^{n+1}
-
\dot{\boldsymbol\eta}^{n+1}
\right)
\cdot
\boldsymbol{\tilde v}_\Gamma^{n+1}
\sqrt a
\,dx\,dy_{\mathrm{arc}}
+
\beta
\int_\omg
\left(
\boldsymbol t_f^n
\circ
\Phi_s^n
\right)
J_s^n
\cdot
\boldsymbol{\tilde v}_\Gamma^{n+1}
\,dx\,dy_{\mathrm{arc}}.
\end{aligned}
\]
Thus, in the fluid variational formulation, the interface integral involving the fluid normal traction is replaced by the RHS of the above identity.

The outlet flow rate and spatially uniform outlet pressure are defined by
\begin{equation}
Q^{n+1}
=
\int_{\Gamma_f^{\mathrm{out},n+1}}
\boldsymbol u^{n+1}
\cdot
\boldsymbol n_f^{n+1}
\,dS,
\quad \text{and}\quad
p_{\mathrm{out}}^{n+1}
=
P_B^{n+1}
+
R_pQ^{n+1}.
\label{discrete_outlet_flow}
\end{equation}

Together, we define the fluid subproblem by solving
$
\boldsymbol u^{n+1}\in V_f(t^{n+1}),
~
p^{n+1}\in Q_f(t^{n+1}),
~
P_B^{n+1}\in\mathbb R$, 
such that, for every
$
\boldsymbol{\tilde v}\in V_{f,0}(t^{n+1}),
~
q\in Q_f(t^{n+1}),
$ sastisfies
\begin{equation}
\begin{aligned}
&\rho_f
\left(
\left.
\delta_t\boldsymbol u^{n+1}
\right|_{\hat\Omega_f},
\boldsymbol{\tilde v}
\right)_{\Omega_f^{n+1}}
+
\rho_f
\left(
\left(
\left(
\boldsymbol u^n
-
\boldsymbol w_f^{n+1}
\right)
\cdot\nabla
\right)
\boldsymbol u^{n+1},
\boldsymbol{\tilde v}
\right)_{\Omega_f^{n+1}}
\\
&\quad
+
2\mu_f
\left(
\mathbb D(\boldsymbol u^{n+1}),
\mathbb D(\boldsymbol{\tilde v})
\right)_{\Omega_f^{n+1}}
-
\left(
p^{n+1},
\nabla\cdot\boldsymbol{\tilde v}
\right)_{\Omega_f^{n+1}}
+
\left(
\nabla\cdot\boldsymbol u^{n+1},
q
\right)_{\Omega_f^{n+1}}
\\
&\quad
+
\frac{\varepsilon\rho_s}{\Delta t}
\int_\omg
\left(
\boldsymbol u_\Gamma^{n+1}
-
\dot{\boldsymbol\eta}^{n+1}
\right)
\cdot
\boldsymbol{\tilde v}_\Gamma^{n+1}
\sqrt a
\,dx\,dy_{\mathrm{arc}}
+
\left\langle
\left(
P_B^{n+1}
+
R_pQ^{n+1}
\right)
\boldsymbol n_f^{n+1},
\boldsymbol{\tilde v}
\right\rangle_{\Gamma_f^{\mathrm{out},n+1}}
\\
&=
\left(
\boldsymbol F_f^{n+1},
\boldsymbol{\tilde v}
\right)_{\Omega_f^{n+1}}
+
\beta
\int_\omg
\left(
\boldsymbol t_f^n
\circ
\Phi_s^n
\right)
J_s^n
\cdot
\boldsymbol{\tilde v}_\Gamma^{n+1}
\,dx\,dy_{\mathrm{arc}}.
\end{aligned}
\label{fluid_beta_subproblem}
\end{equation}
While the Windkessel state is updated by the backward-Euler equation
\begin{equation}
C
\frac{
P_B^{n+1}
-
P_B^n
}{\Delta t}
+
\frac{
P_B^{n+1}
-
P_v
}{R_d}
=
Q^{n+1}.
\label{discrete_windkessel}
\end{equation}

\section{Stability analysis of the splitting scheme}
\label{sec:beta_stability}
In this section, we derive an energy estimate for our kinematically coupled $\beta$ scheme for a linearized problem in the fixed domain. The estimate identifies the fluid, shell, interface and Windkessel energy contributions controlled by the operator splitting. We emphasize that this analysis does not, by itself, establish the stability of the scheme for the fully nonlinear problem on the moving domain. Nevertheless, the numerical results presented in Section \ref{numerical_result} indicate that the scheme remains stable in this setting.

We consider the fixed fluid domain  $\hat \Omega_f$ and evaluate the coupling on the
reference interface
$\hat\Gamma=\varphi(\omg)$.
Thus,
\[
\Phi_s^n\approx\varphi,
\qquad
J_s^n\approx\sqrt a,
\qquad
\boldsymbol w_f^n=\boldsymbol0.
\]
Here, the reference geometry can still be considered as banded geometry, namely, $\sqrt a$ is not assumed to be equal to one. We define the pulled-back outward fluid
traction by
\begin{equation}
\boldsymbol\tau_f^n
=
\left(
\boldsymbol t_f^n\circ\varphi
\right)
\sqrt a.
\label{fixed_pulled_back_traction}
\end{equation}
Then, the corresponding load exerted by the fluid on the shell is
$-\boldsymbol\tau_f^n$.

To simplify the notation in the latter proof,  we denote the translational and rotational shell mass densities per unit
parameter area be
\[
m_t
=
\varepsilon\rho_s\sqrt a,
\qquad
m_r
=
\frac{\varepsilon^3\rho_s}{12}\sqrt a.
\]
For vector fields on domain $\omg$, define
\[
\begin{aligned}
M_t(\boldsymbol v,\boldsymbol z)
=
\int_\omg
m_t\,
\boldsymbol v\cdot\boldsymbol z
\,dx\,dy_{\mathrm{arc}},
\qquad
M_r(\boldsymbol w,\boldsymbol r)
=
\int_\omg
m_r\,
\boldsymbol w\cdot\boldsymbol r
\,dx\,dy_{\mathrm{arc}},
\end{aligned}
\]
and the associated norms as
\[
\|\boldsymbol v\|_{M_t}^2
=
M_t(\boldsymbol v,\boldsymbol v),
\qquad
\|\boldsymbol w\|_{M_r}^2
=
M_r(\boldsymbol w,\boldsymbol w).
\]
For a load $\boldsymbol g$ per unit parameter area, define the inverse
translational mass norm by
\begin{equation}
\|\boldsymbol g\|_{M_t^{-1}}^2
=
\int_\omg
\frac{|\boldsymbol g|^2}{m_t}
\,dx\,dy_{\mathrm{arc}}.
\label{inverse_shell_mass_norm}
\end{equation}
In particular, equation~\eqref{fixed_pulled_back_traction} gives
\[
\|\boldsymbol\tau_f^n\|_{M_t^{-1}}^2
=
\frac{1}{\varepsilon\rho_s}
\|\boldsymbol t_f^n\|_{L^2(\Gamma)}^2.
\]
This equality remains valid for the banded reference geometry because we included the
factor $\sqrt a$ in both the traction pullback and the surface measure.

\begin{remark}
To make the interface-traction energy well defined, we assume that the
pulled-back fluid traction satisfies
\[
\boldsymbol\tau_f^n
\in
L^2(\omg;\mathbb R^3),
\qquad
n=0,\ldots,N.
\]
This assumption may be interpreted either as an additional regularity
assumption on the time discrete solution or, at the fully discrete level,
as the use of an $L^2$-projected traction in a common finite element
interface space. Because of the assumption of the geometry given in \eqref{def:R}, the translational surface mass
density is uniformly nondegenerate:
\[
0<m_*
\le
m_t(x,y_{\mathrm{arc}})
\le
m^*
<\infty
\qquad
\text{for a.e. }
(x,y_{\mathrm{arc}})\in\omg.
\]
Consequently, $m_t^{-1}\in L^\infty(\omg)$, and the inverse mass norm
\[
\|\boldsymbol g\|_{M_t^{-1}}^2
=
\int_\omg
\frac{|\boldsymbol g|^2}{m_t}
\,dx\,dy_{\mathrm{arc}}
\]
is well defined.
\end{remark}

We define the total shell elastic bilinear form by
\[
A_s
\big(
(\boldsymbol\eta,\omb),
(\boldsymbol v,\boldsymbol w)
\big)
=
B_m
\big(
(\boldsymbol\eta,\omb),
(\boldsymbol v,\boldsymbol w)
\big)
+
B_b
\big(
(\boldsymbol\eta,\omb),
(\boldsymbol v,\boldsymbol w)
\big),
\]
where one can show that $A_s$ is symmetric and positive definite on $V_N$; see Lemma 2.1 in \cite{TambacaTutek2016}.

At every completed time level, the discrete kinematic condition gives
\[
\dot{\boldsymbol\eta}^n
=
\boldsymbol u_\Gamma^n
=
\boldsymbol u^n\circ\varphi.
\]
Therefore, in the fixed-domain setting, the Robin relation from the fluid subproblem
becomes
\begin{equation}
\frac{m_t}{\Delta t}
\left(
\boldsymbol u_\Gamma^{n+1}
-
\dot{\boldsymbol\eta}^{n+1}
\right)
+
\boldsymbol\tau_f^{n+1}
=
\beta\boldsymbol\tau_f^n
\qquad
\text{on }\omg.
\label{fixed_beta_robin_relation}
\end{equation}

\paragraph{Discrete energy}
The shell elastic energy is
\[
\mathcal E_s^n
=
\frac12
A_s
\big(
(\boldsymbol\eta^n,\omb^n),
(\boldsymbol\eta^n,\omb^n)
\big).
\]
The total discrete energy at the time $t^n$ is defined by
\begin{equation}
\begin{aligned}
\mathcal E^n
={}&
\frac{\rho_f}{2}
\|\boldsymbol u^n\|_{L^2(\Omega_f)}^2
+
\frac12
\|\boldsymbol u_\Gamma^n\|_{M_t}^2
+
\frac12
\|\dot{\omb}^n\|_{M_r}^2
+
\mathcal E_s^n
+
\frac{C}{2}
|P_B^n-P_v|^2
+
\frac{\Delta t^2}{2}
\|\boldsymbol\tau_f^n\|_{M_t^{-1}}^2.
\end{aligned}
\label{beta_total_discrete_energy}
\end{equation}
Where the last term is a numerical interface traction energy generated by the Robin splitting.

The physical dissipation over the time step is
\begin{equation}
\begin{aligned}
\mathcal D_{\mathrm{phys}}^{n+1}
={}&
2\mu_f\Delta t
\|\mathbb D(\boldsymbol u^{n+1})\|_{L^2(\Omega_f)}^2
+
\Delta t\,R_p|Q^{n+1}|^2
+
\frac{\Delta t}{R_d}
|P_B^{n+1}-P_v|^2.
\end{aligned}
\label{beta_physical_dissipation}
\end{equation}
The numerical dissipation is
\begin{equation}
\begin{aligned}
\mathcal D_{\mathrm{num}}^{n+1}
=&
\frac{\rho_f}{2}
\|\boldsymbol u^{n+1}-\boldsymbol u^n\|_{L^2(\Omega_f)}^2
+
\frac12
\|\dot{\boldsymbol\eta}^{n+1}
-\boldsymbol u_\Gamma^n\|_{M_t}^2
+
\frac12
\|\dot{\omb}^{n+1}
-\dot{\omb}^n\|_{M_r}^2
\\
&
+
\frac12
A_s
\big(
(\boldsymbol\eta^{n+1}-\boldsymbol\eta^n,
 \omb^{n+1}-\omb^n),
(\boldsymbol\eta^{n+1}-\boldsymbol\eta^n,
 \omb^{n+1}-\omb^n)
\big)
\\
&
+
\frac{C}{2}
|P_B^{n+1}-P_B^n|^2.
\end{aligned}
\label{beta_numerical_dissipation}
\end{equation}
Finally, the numerical dissipation associated with the traction split is
\begin{equation}
\mathcal D_\beta^{n+1}
=
\frac{\Delta t^2}{2}
(1-\beta^2)
\|\boldsymbol\tau_f^n\|_{M_t^{-1}}^2.
\label{beta_interface_dissipation}
\end{equation}
For $0\leq\beta\leq1$, all terms in
equations~\eqref{beta_physical_dissipation}-\eqref{beta_interface_dissipation} are nonnegative.

\paragraph{Semi-implicit convection}
For the semi-implicit convective term used in the fluid subproblem, define
\begin{equation}
\mathcal C_f^{n+1}
=
\rho_f\Delta t
\left(
\big(
(\boldsymbol u^n\cdot\nabla)
\boldsymbol u^{n+1}
\big),
\boldsymbol u^{n+1}
\right)_{\Omega_f}.
\label{convective_energy_contribution}
\end{equation}
Integration by parts gives
\begin{equation}
\mathcal C_f^{n+1}
=
\frac{\rho_f\Delta t}{2}
\int_{\partial\Omega_f}
\left(
\boldsymbol u^n\cdot\boldsymbol n_f
\right)
\left|\boldsymbol u^{n+1}\right|^2
\,dS
-
\frac{\rho_f \Delta t}{2}
\int_{\Omega_f}
\left(\nabla \cdot \boldsymbol{u}^n\right)
\left|\boldsymbol{u}^{n+1}\right|^2\,d\boldsymbol{x}.
\label{convective_boundary_flux}
\end{equation}
Consequently, the standard semi-implicit convective form is not
automatically energy neutral on an open domain. It is energy neutral if it
is replaced by the skew-symmetric form
\[
\begin{aligned}
c_{\mathrm{skew}}
(\boldsymbol u^n;
 \boldsymbol u^{n+1},
 \boldsymbol{\tilde v})
=
\frac{\rho_f}{2}
\bigg(&
\left(
(\boldsymbol u^n\cdot\nabla)
\boldsymbol u^{n+1},
\boldsymbol{\tilde v}
\right)_{\Omega_f}
-
\left(
(\boldsymbol u^n\cdot\nabla)
\boldsymbol{\tilde v},
\boldsymbol u^{n+1}
\right)_{\Omega_f}
\bigg),
\end{aligned}
\]
because
\[
c_{\mathrm{skew}}
(\boldsymbol u^n;
 \boldsymbol u^{n+1},
 \boldsymbol u^{n+1})
=0.
\]
Alternatively, the standard form may be retained if the boundary flux in
equation~\eqref{convective_boundary_flux} is nonnegative or is treated as
prescribed kinetic-energy flux.

\begin{theorem}[Fixed-domain energy identity]
\label{thm:fixed_beta_stability}
Assume that the fluid domain and reference interface are fixed, that the
inlet condition is homogeneous so that $\boldsymbol u^{n+1}$ is an
admissible fluid test function. Assume also that the pulled-back fluid traction satisfies
$\boldsymbol\tau_f^n
\in
L^2(\omg;\mathbb R^3),~
n=0,\ldots,N$, and that the Robin interface relation
\eqref{fixed_beta_robin_relation} holds in
$L^2(\omg;\mathbb R^3)$. Consequently, the inverse-mass traction norm
\[
\|\boldsymbol\tau_f^n\|_{M_t^{-1}}^2
=
\int_\omg
\frac{|\boldsymbol\tau_f^n|^2}{m_t}
\,dx\,dy_{\mathrm{arc}}
\]
is well defined. Let $0\leq\beta\leq1$ and assume that $P_v$ is constant.
Then the splitting scheme satisfies
\begin{equation}
\begin{aligned}
\mathcal E^{n+1}
&+
\mathcal D_{\mathrm{phys}}^{n+1}
+
\mathcal D_{\mathrm{num}}^{n+1}
+
\mathcal D_\beta^{n+1}
+
\mathcal C_f^{n+1}
=
\mathcal E^n
+
\Delta t
\left(
\boldsymbol F_f^{n+1},
\boldsymbol u^{n+1}
\right)_{\Omega_f}
-
\Delta t\,P_vQ^{n+1}.
\end{aligned}
\label{fixed_beta_energy_identity}
\end{equation}
Here $\mathcal C_f^{n+1}$ denotes the contribution of the standard
semi-implicit convective term defined in equation~\eqref{convective_energy_contribution}. If the convective term is
omitted or replaced by the skew-symmetric form, the same identity holds
with $\mathcal C_f^{n+1}=0$, and if
$\boldsymbol F_f^{n+1}=\boldsymbol0$ and $P_v=0$, then
\begin{equation}
\mathcal E^{n+1}
+
\mathcal D_{\mathrm{phys}}^{n+1}
+
\mathcal D_{\mathrm{num}}^{n+1}
+
\mathcal D_\beta^{n+1}
\leq
\mathcal E^n.
\label{fixed_beta_energy_inequality}
\end{equation}
The same inequality holds for the standard semi-implicit convective form
whenever $\mathcal C_f^{n+1}\geq0$.
\end{theorem}

\begin{proof} Test the structure subproblem \eqref{structure_beta_subproblem} with
$(\boldsymbol v,\boldsymbol w)
=
(\dot{\boldsymbol\eta}^{n+1},
 \dot{\omb}^{n+1})$
and multiply the resulting equation by $\Delta t$. The translational
inertia term becomes
\[
\begin{aligned}
&\Delta t\,
M_t
\left(
\frac{
\dot{\boldsymbol\eta}^{n+1}
-\boldsymbol u_\Gamma^n
}{\Delta t},
\dot{\boldsymbol\eta}^{n+1}
\right)
=
\frac12
\left(
\|\dot{\boldsymbol\eta}^{n+1}\|_{M_t}^2
-
\|\boldsymbol u_\Gamma^n\|_{M_t}^2
+
\|\dot{\boldsymbol\eta}^{n+1}
-\boldsymbol u_\Gamma^n\|_{M_t}^2
\right).
\end{aligned}
\]
Similarly, the rotational inertia term gives
\[
\begin{aligned}
&\Delta t\,
M_r
\left(
\frac{
\dot{\omb}^{n+1}
-\dot{\omb}^n
}{\Delta t},
\dot{\omb}^{n+1}
\right)
=
\frac12
\left(
\|\dot{\omb}^{n+1}\|_{M_r}^2
-
\|\dot{\omb}^n\|_{M_r}^2
+
\|\dot{\omb}^{n+1}
-\dot{\omb}^n\|_{M_r}^2
\right).
\end{aligned}
\]
Using the structure updates
\[
\boldsymbol\eta^{n+1}-\boldsymbol\eta^n
=
\Delta t\,
\dot{\boldsymbol\eta}^{n+1},
\qquad
\omb^{n+1}-\omb^n
=
\Delta t\,
\dot{\omb}^{n+1},
\]
the bilinear term $A_s$ becomes
\[
\begin{aligned}
&\Delta t\,
A_s
\big(
(\boldsymbol\eta^{n+1},\omb^{n+1}),
(\dot{\boldsymbol\eta}^{n+1},
\dot{\omb}^{n+1})
\big)
=
\mathcal E_s^{n+1}
-
\mathcal E_s^n
+
\frac12
A_s
\big(
(\boldsymbol\eta^{n+1}-\boldsymbol\eta^n,
 \omb^{n+1}-\omb^n),
(\boldsymbol\eta^{n+1}-\boldsymbol\eta^n,
 \omb^{n+1}-\omb^n)
\big).
\end{aligned}
\]

Next, we test the fluid subproblem \eqref{fluid_beta_subproblem} with
$\boldsymbol{\tilde v}=\boldsymbol u^{n+1},~
q=p^{n+1}$,
and multiply it by $\Delta t$. The pressure term cancels with the divergence free term
\[
-\left(p^{n+1}, \nabla \cdot \boldsymbol{u}^{n+1}\right)
+\left(\nabla \cdot \boldsymbol{u}^{n+1}, p^{n+1}\right)
= 0.
\]
The fluid inertia term becomes
\[
\begin{aligned}
&\rho_f\Delta t
\left(
\frac{
\boldsymbol u^{n+1}-\boldsymbol u^n
}{\Delta t},
\boldsymbol u^{n+1}
\right)_{\Omega_f}
=
\frac{\rho_f}{2}
\left(
\|\boldsymbol u^{n+1}\|_{L^2(\Omega_f)}^2
-
\|\boldsymbol u^n\|_{L^2(\Omega_f)}^2
+
\|\boldsymbol u^{n+1}-\boldsymbol u^n\|_{L^2(\Omega_f)}^2
\right).
\end{aligned}
\]
The viscous term gives the first term in
$\mathcal D_{\mathrm{phys}}^{n+1}$, and the semi-implicit convective term
gives $\mathcal C_f^{n+1}$.

Regarding the interface terms, we have
\begin{equation}
\begin{aligned}
\mathcal I^{n+1}
={}&
M_t
\left(
\boldsymbol u_\Gamma^{n+1}
-
\dot{\boldsymbol\eta}^{n+1},
\boldsymbol u_\Gamma^{n+1}
\right)
+
\beta\Delta t
\int_\omg
\boldsymbol\tau_f^n
\cdot
\dot{\boldsymbol\eta}^{n+1}
\,dx\,dy_{\mathrm{arc}}
-
\beta\Delta t
\int_\omg
\boldsymbol\tau_f^n
\cdot
\boldsymbol u_\Gamma^{n+1}
\,dx\,dy_{\mathrm{arc}}
\\
={}&
M_t
\left(
\boldsymbol u_\Gamma^{n+1}
-
\dot{\boldsymbol\eta}^{n+1},
\boldsymbol u_\Gamma^{n+1}
\right)
-
\beta\Delta t
\int_\omg
\boldsymbol\tau_f^n
\cdot
\left(
\boldsymbol u_\Gamma^{n+1}
-
\dot{\boldsymbol\eta}^{n+1}
\right)
\,dx\,dy_{\mathrm{arc}}.
\end{aligned}
\label{combined_interface_terms}
\end{equation}
The Robin relation \eqref{fixed_beta_robin_relation} gives
\begin{equation}
\boldsymbol u_\Gamma^{n+1}
-
\dot{\boldsymbol\eta}^{n+1}
=
\frac{\Delta t}{m_t}
\left(
\beta\boldsymbol\tau_f^n
-
\boldsymbol\tau_f^{n+1}
\right).
\label{interface_velocity_difference}
\end{equation}
Apply the identity
\[
(\boldsymbol a-\boldsymbol b)\cdot\boldsymbol a
=
\frac12
\left(
|\boldsymbol a|^2
-
|\boldsymbol b|^2
+
|\boldsymbol a-\boldsymbol b|^2
\right)
\]
to the first term in
equation~\eqref{combined_interface_terms}. This gives
\[
\begin{aligned}
&M_t
\left(
\boldsymbol u_\Gamma^{n+1}
-
\dot{\boldsymbol\eta}^{n+1},
\boldsymbol u_\Gamma^{n+1}
\right)
=
\frac12
\left(
\|\boldsymbol u_\Gamma^{n+1}\|_{M_t}^2
-
\|\dot{\boldsymbol\eta}^{n+1}\|_{M_t}^2
\right)
+
\frac12
\left\|
\boldsymbol u_\Gamma^{n+1}
-
\dot{\boldsymbol\eta}^{n+1}
\right\|_{M_t}^2.
\end{aligned}
\]
Next, the second term in \eqref{combined_interface_terms} can be rewritten by using \eqref{interface_velocity_difference}, which gives
\[
\begin{aligned}
&\beta\Delta t
\int_\omg
\boldsymbol\tau_f^n
\cdot
\left(
\boldsymbol u_\Gamma^{n+1}
-
\dot{\boldsymbol\eta}^{n+1}
\right)
\,dx\,dy_{\mathrm{arc}}
=
\Delta t^2
\int_\omg
\frac{
\beta\boldsymbol\tau_f^n
\cdot
\left(
\beta\boldsymbol\tau_f^n
-
\boldsymbol\tau_f^{n+1}
\right)
}{m_t}
\,dx\,dy_{\mathrm{arc}}.
\end{aligned}
\]
Applying the same algebraic identity in the inverse-mass inner product
yields
\[
\begin{aligned}
&\beta\Delta t
\int_\omg
\boldsymbol\tau_f^n
\cdot
\left(
\boldsymbol u_\Gamma^{n+1}
-
\dot{\boldsymbol\eta}^{n+1}
\right)
\,dx\,dy_{\mathrm{arc}}
=
\frac{\Delta t^2}{2}
\left(
\|\beta\boldsymbol\tau_f^n\|_{M_t^{-1}}^2
-
\|\boldsymbol\tau_f^{n+1}\|_{M_t^{-1}}^2
\right)
+
\frac{\Delta t^2}{2}
\left\|
\beta\boldsymbol\tau_f^n
-
\boldsymbol\tau_f^{n+1}
\right\|_{M_t^{-1}}^2.
\end{aligned}
\]
Furthermore, equation~\eqref{interface_velocity_difference} implies
\[
\left\|
\boldsymbol u_\Gamma^{n+1}
-
\dot{\boldsymbol\eta}^{n+1}
\right\|_{M_t}^2
=
\int_\omg
m_t
\left|
\frac{\Delta t}{m_t}
\left(
\beta\boldsymbol\tau_f^n
-
\boldsymbol\tau_f^{n+1}
\right)
\right|^2
\,dx\,dy_{\mathrm{arc}}
=
\Delta t^2
\left\|
\beta\boldsymbol\tau_f^n
-
\boldsymbol\tau_f^{n+1}
\right\|_{M_t^{-1}}^2.
\]
Substituting these identities into
equation~\eqref{combined_interface_terms}, we obtain
\[
\begin{aligned}
\mathcal I^{n+1}
={}&
\frac12
\left(
\|\boldsymbol u_\Gamma^{n+1}\|_{M_t}^2
-
\|\dot{\boldsymbol\eta}^{n+1}\|_{M_t}^2
\right)
+
\frac12
\left\|
\boldsymbol u_\Gamma^{n+1}
-
\dot{\boldsymbol\eta}^{n+1}
\right\|_{M_t}^2
\\
&-
\frac{\Delta t^2}{2}
\left(
\|\beta\boldsymbol\tau_f^n\|_{M_t^{-1}}^2
-
\|\boldsymbol\tau_f^{n+1}\|_{M_t^{-1}}^2
\right)
-
\frac{\Delta t^2}{2}
\left\|
\beta\boldsymbol\tau_f^n
-
\boldsymbol\tau_f^{n+1}
\right\|_{M_t^{-1}}^2.
\end{aligned}
\]
Since the two mismatch terms cancel exactly, we have
\begin{equation}
\begin{aligned}
\mathcal I^{n+1}
={}&
\frac12
\left(
\|\boldsymbol u_\Gamma^{n+1}\|_{M_t}^2
-
\|\dot{\boldsymbol\eta}^{n+1}\|_{M_t}^2
\right)
+
\frac{\Delta t^2}{2}
\left(
\|\boldsymbol\tau_f^{n+1}\|_{M_t^{-1}}^2
-
\|\beta\boldsymbol\tau_f^n\|_{M_t^{-1}}^2
\right).
\end{aligned}
\label{interface_energy_identity}
\end{equation}
Furthermore, we can rewrite $\frac{\Delta t^2}{2}
\left(
\|\boldsymbol\tau_f^{n+1}\|_{M_t^{-1}}^2
-
\|\beta\boldsymbol\tau_f^n\|_{M_t^{-1}}^2
\right)$ as follows
\[
\begin{aligned}
\frac{\Delta t^2}{2}
\left(
\|\boldsymbol\tau_f^{n+1}\|_{M_t^{-1}}^2
-
\|\beta\boldsymbol\tau_f^n\|_{M_t^{-1}}^2
\right)
=
\frac{\Delta t^2}{2}
\left(
\|\boldsymbol\tau_f^{n+1}\|_{M_t^{-1}}^2
-
\|\boldsymbol\tau_f^n\|_{M_t^{-1}}^2
\right)
+
\frac{\Delta t^2}{2}
(1-\beta^2)
\|\boldsymbol\tau_f^n\|_{M_t^{-1}}^2.
\end{aligned}
\]
This produces the interface-traction energy and the nonnegative
$\beta$-splitting dissipation.

Regarding the work done by the outlet boundary condition, from \eqref{discrete_outlet_flow}
and \eqref{discrete_windkessel}, we can derive
\begin{equation}
\begin{aligned}
\Delta t\,
p_{\mathrm{out}}^{n+1}Q^{n+1}
&=
\frac{C}{2}
\left(
|P_B^{n+1}-P_v|^2
-
|P_B^n-P_v|^2
+
|P_B^{n+1}-P_B^n|^2
\right)
\\
&\quad
+
\frac{\Delta t}{R_d}
|P_B^{n+1}-P_v|^2
+
\Delta t\,R_p|Q^{n+1}|^2
+
\Delta t\,P_vQ^{n+1},
\end{aligned}
\label{discrete_windkessel_energy_identity}
\end{equation}
given that $P_v$ is constant.
The first line on the right hand side supplies the Windkessel storage
increment and its backward-Euler numerical dissipation. The next two terms
are the distal and proximal resistance dissipations. The prescribed distal
pressure term is moved to the right-hand side as external work.

Collecting the fluid, shell, interface, and Windkessel identities gives
equation~\eqref{fixed_beta_energy_identity}. Under the additional
zero-work and convection assumptions stated in the theorem, all remaining
terms are nonnegative, which yields
equation~\eqref{fixed_beta_energy_inequality}.
Finally, summing equation~\eqref{fixed_beta_energy_inequality} from $n=0$ to
$n=N-1$ gives
\[
\begin{aligned}
\mathcal E^N
&+
\sum_{n=0}^{N-1}
\left(
\mathcal D_{\mathrm{phys}}^{n+1}
+
\mathcal D_{\mathrm{num}}^{n+1}
+
\mathcal D_\beta^{n+1}
\right)
\leq
\mathcal E^0.
\end{aligned}
\]
Thus, the fixed-domain scheme is unconditionally stable with respect to
the added-mass coupling for every $\beta\in[0,1]$, subject to the
convection and boundary-work assumptions in
Theorem~\ref{thm:fixed_beta_stability}.
\end{proof}

\begin{remark}(Nonhomogeneous inlet data)
We want to point out that for the prescribed nonzero inlet velocity used in the full problem,
$\boldsymbol u^{n+1}$ is not an admissible homogeneous test function.
One must first introduce an $H^1$ lifting of
$\boldsymbol u_{\mathrm{in}}^{n+1}$ and test the fluid equation with the
homogeneous remainder. The resulting estimate contains additional work
terms involving the lifting and its discrete time difference. Therefore,
the prescribed inlet does not invalidate the energy method, but its work
cannot be omitted from the right-hand side of the estimate.
\end{remark}

\section{Estimation of the Windkessel parameters using data assimilation}
\label{sec:ukf_parameter_estimation}

To estimate the Windkessel parameters from pressure measurements, we use an
unscented Kalman filter applied to a low-dimensional parameter state. The finite-element variables for the fluid velocity and pressure and for the
shell displacement, rotation, and their velocities are not included in the
Kalman filter state. Instead, these unknowns are computed deterministically by fluid structure interaction solver and provide the nonlinear parameter-to-observation
map required by the filter. We therefore refer to this procedure as a
reduced-state unscented Kalman filter.

\paragraph{Parameter state and logarithmic transformation}

For the single-outlet Windkessel model considered in this work, the Kalman filter state contains only a vector of three parameters
\begin{equation}
\boldsymbol\theta
=
\begin{pmatrix}
R_p & R_d & C
\end{pmatrix}^{T}.
\label{ukf_physical_parameter_vector}
\end{equation}
Here, proximal resistance $R_p$ and distal resistance $R_d$ are expressed in
$\mathrm{Pa\,s\,m^{-3}}$, while compliance $C$ is expressed in
$\mathrm{m^3\,Pa^{-1}}$. Throughout the implementation, the logarithms
below are applied to the numerical values of these parameters expressed in
these fixed units. Because the $R_p$, $R_d$, and $C$ must remain strictly positive, the filter is reformulated in logarithmic
coordinates:
\begin{equation}
\boldsymbol\vartheta
=
\log\boldsymbol\theta
=
\begin{pmatrix}
\log R_p & \log R_d & \log C
\end{pmatrix}^{T},
\qquad
\boldsymbol\theta
=
\exp\boldsymbol\vartheta,
\label{ukf_log_transform}
\end{equation}
where the logarithm and exponential are applied componentwisely to ensure every value of
$\boldsymbol\vartheta\in\mathbb{R}^3$ produces strictly positive values of
$R_p$, $R_d$, and $C$.

Let
$
\widehat{\boldsymbol\vartheta}^{\,n}
\in\mathbb{R}^3,
~
\mathbf P_\vartheta^n
\in\mathbb{R}^{3\times3}$
denote the filtered mean of the logarithmic parameter vector and its
covariance at time $t^n$, respectively. No process-noise covariance is included in the parameter-evolution model.
Therefore, between two consecutive pressure measurements, the prior
parameter mean and covariance are taken directly from the preceding
posterior values:
\[
\widehat{\boldsymbol\vartheta}^{\,n+1,-}
=
\widehat{\boldsymbol\vartheta}^{\,n},
\qquad
\mathbf P_\vartheta^{n+1,-}
=
\mathbf P_\vartheta^n.
\]
This assumption treats $R_p$, $R_d$, and $C$ as unknown but constant in time. At each time step, the filter
compares the predicted inlet pressure with the measured inlet pressure and
uses their difference to update the estimates of $R_p$, $R_d$, and $C$.

\paragraph{Pressure observation}

Because the inlet velocity is prescribed strongly, we choose the inlet pressure as the output. In particular, the scalar model observation is defined as the averaged inlet pressure:
\begin{equation}
\mathcal H_{n+1}(\boldsymbol\theta)
=
\frac{1}
{\left|\Gamma_f^{\mathrm{in},n+1}\right|}
\int_{\Gamma_f^{\mathrm{in},n+1}}
p^{n+1}(\boldsymbol x;\boldsymbol\theta)
\,dS.
\label{ukf_observation_operator}
\end{equation}
Here,
$p^{n+1}(\cdot;\boldsymbol\theta)$ is obtained by solving the
fluid-Windkessel subproblem with the candidate parameter vector
$\boldsymbol\theta$.

The measured inlet pressure is modeled as
\begin{equation}
z_{\mathrm{obs}}^{n+1}
=
\mathcal H_{n+1}(\boldsymbol\theta)
+
\nu^{n+1},
\qquad
\mathbb E[\nu^{n+1}]
=
0,
\qquad
\mathbb E\left[(\nu^{n+1})^2\right]
=
R_{\mathrm{obs}},
\label{ukf_measurement_model}
\end{equation}
where $R_{\mathrm{obs}}>0$ is the prescribed measurement-noise variance.
In the present implementation, we choose
$R_{\mathrm{obs}}
=
(50\ \mathrm{Pa})^2,
\label{ukf_measurement_variance}$
and 
$z_{\mathrm{obs}}^{n+1}
=
1733.19\ \mathrm{Pa}$. For patient-specific calculations, this
constant can is replaced by the measured inlet-pressure waveform evaluated at
$t^{n+1}$.

\paragraph{Generation of sigma points}

The logarithmic parameter state has dimension
$N_\theta=3$.
Define
\begin{equation}
\lambda_{\mathrm{UKF}}
=
\alpha_{\mathrm{UKF}}^2
\left(
3+\kappa_{\mathrm{UKF}}
\right)
-
3.
\label{ukf_lambda}
\end{equation}
In our implementation, we choose
\begin{equation}
\alpha_{\mathrm{UKF}}
=
10^{-1},
\qquad
\beta_{\mathrm{UKF}}
=
2,
\qquad
\kappa_{\mathrm{UKF}}
=
0.
\label{ukf_scaling_parameters}
\end{equation}
We note that the parameter $\beta_{\mathrm{UKF}}$ is an unscented-transform parameter and
is unrelated to the fluid--structure splitting parameter $\beta$.

Let $\mathbf L^{n+1}$ be the lower-triangular Cholesky factor satisfying
\begin{equation}
\mathbf L^{n+1}
\left(
\mathbf L^{n+1}
\right)^T
=
\left(
3+\lambda_{\mathrm{UKF}}
\right)
\left(
\mathbf P_\vartheta^{n+1,-}
+
\epsilon_{\mathrm{chol}}\mathbf I_3
\right),
\label{ukf_cholesky_factor}
\end{equation}
where
$\epsilon_{\mathrm{chol}}
=
10^{-12}$
is a small numerical regularization parameter.
Because $N_\theta=3$, the unscented transform uses $2N_\theta+1=7$ sigma points, indexed by
$i=0,\ldots,6$
\begin{equation}
\begin{cases}
\begin{aligned}
\boldsymbol\chi_0^{n+1}
&=
\widehat{\boldsymbol\vartheta}^{\,n+1,-},
\\
\boldsymbol\chi_i^{n+1}
&=
\widehat{\boldsymbol\vartheta}^{\,n+1,-}
+
\mathbf L_{:,i}^{n+1},
\qquad
i=1,2,3,
\\
\boldsymbol\chi_{i+3}^{n+1}
&=
\widehat{\boldsymbol\vartheta}^{\,n+1,-}
-
\mathbf L_{:,i}^{n+1},
\qquad
i=1,2,3,
\end{aligned}
\end{cases}
\label{ukf_sigma_points}
\end{equation}
where $\mathbf L_{:,i}^{n+1}$ denotes the $i$th column of
$\mathbf L^{n+1}$. 

Therefore, the corresponding mean weights are
\begin{equation}
W_0^{(m)}
=
\frac{\lambda_{\mathrm{UKF}}}
{3+\lambda_{\mathrm{UKF}}},
\qquad
W_i^{(m)}
=
\frac{1}
{2\left(3+\lambda_{\mathrm{UKF}}\right)},
\quad
i=1,\ldots,6.
\label{ukf_mean_weights}
\end{equation}
The covariance weights are
\begin{equation}
\begin{cases}
\begin{aligned}
W_0^{(c)}
&=
\frac{\lambda_{\mathrm{UKF}}}
{3+\lambda_{\mathrm{UKF}}}
+
1-\alpha_{\mathrm{UKF}}^2
+
\beta_{\mathrm{UKF}},
\\
W_i^{(c)}
&=
\frac{1}
{2\left(3+\lambda_{\mathrm{UKF}}\right)},
\qquad
i=1,\ldots,6.
\end{aligned}
\end{cases}
\label{ukf_covariance_weights}
\end{equation}

\paragraph{Propagation of the sigma points}

At time $t^{n+1}$, the structure \eqref{structure_beta_subproblem} and ALE
subproblems are first computed. During the parameter-estimation step, the
resulting shell displacement, structure-stage velocity, fluid mesh, mesh
velocity, previous fluid state, and previous capacitor pressure $P_B^n$ are
held fixed.

For each sigma point $\boldsymbol\chi_i^{n+1}$, the corresponding positive
physical parameter vector is
\begin{equation}
\boldsymbol\theta_i^{n+1}
=
\exp
\left(
\boldsymbol\chi_i^{n+1}
\right)
=
\begin{pmatrix}
R_{p,i}^{n+1}
&
R_{d,i}^{n+1}
&
C_i^{n+1}
\end{pmatrix}^{T},
\qquad
i=0,\ldots,6.
\label{ukf_physical_sigma_point}
\end{equation}
And for each parameter vector $\boldsymbol\theta_i^{n+1}$, one candidate
fluid-Windkessel problem \eqref{fluid_beta_subproblem} is computed. 
Let $p_i^{n+1}$
denote the pressure obtained from the the $i_{th}$ sigma point, the
predicted observation associated with the $i_{th}$ sigma point is
\begin{equation}
z_i^{n+1}
=
\frac{1}
{\left|\Gamma_f^{\mathrm{in},n+1}\right|}
\int_{\Gamma_f^{\mathrm{in},n+1}}
p_i^{n+1}
\,dS,
\qquad
i=0,\ldots,6.
\label{ukf_sigma_observations}
\end{equation}

\paragraph{Unscented Kalman update}

Then, the predicted mean observation can be computed
\begin{equation}
\overline z^{n+1}
=
\sum_{i=0}^{6}
W_i^{(m)}z_i^{n+1},
\label{ukf_predicted_observation}
\end{equation}
with the corresponding innovation variance
\begin{equation}
S_z^{n+1}
=
\sum_{i=0}^{6}
W_i^{(c)}
\left(
z_i^{n+1}
-
\overline z^{n+1}
\right)^2
+
R_{\mathrm{obs}}.
\label{ukf_innovation_variance}
\end{equation}
We can also compute the parameter-observation cross covariance
\begin{equation}
\mathbf C_{\vartheta z}^{n+1}
=
\sum_{i=0}^{6}
W_i^{(c)}
\left(
\boldsymbol\chi_i^{n+1}
-
\widehat{\boldsymbol\vartheta}^{\,n+1,-}
\right)
\left(
z_i^{n+1}
-
\overline z^{n+1}
\right).
\label{ukf_cross_covariance}
\end{equation}
Because the observation is scalar,
\[
\mathbf C_{\vartheta z}^{n+1}
=
\begin{pmatrix}
\operatorname{Cov}(\log R_p,z)\\
\operatorname{Cov}(\log R_d,z)\\
\operatorname{Cov}(\log C,z)
\end{pmatrix}
\in\mathbb{R}^3.
\qquad
S_z^{n+1}
\in\mathbb{R}.
\]
The Kalman gain is therefore
\begin{equation}
\mathbf K^{n+1}
=
\mathbf C_{\vartheta z}^{n+1}
\left(
S_z^{n+1}
\right)^{-1}
\in\mathbb{R}^3.
\label{ukf_gain}
\end{equation}

The posterior logarithmic parameter mean is
\begin{equation}
\widehat{\boldsymbol\vartheta}^{\,n+1}
=
\widehat{\boldsymbol\vartheta}^{\,n+1,-}
+
\mathbf K^{n+1}
\left(
z_{\mathrm{obs}}^{n+1}
-
\overline z^{n+1}
\right).
\label{ukf_parameter_update}
\end{equation}
Its covariance is updated according to
\begin{equation}
\mathbf P_\vartheta^{n+1}
=
\mathbf P_\vartheta^{n+1,-}
-
\mathbf K^{n+1}
S_z^{n+1}
\left(
\mathbf K^{n+1}
\right)^T.
\label{ukf_covariance_update}
\end{equation}
To reduce the effect of roundoff error, the computed covariance is
symmetrized and regularized:
\begin{equation}
\mathbf P_\vartheta^{n+1}
\leftarrow
\frac{1}{2}
\left[
\mathbf P_\vartheta^{n+1}
+
\left(
\mathbf P_\vartheta^{n+1}
\right)^T
\right]
+
10^{-14}\mathbf I_3.
\label{ukf_covariance_regularization}
\end{equation}

The estimated physical Windkessel parameters are recovered through
\begin{equation}
\widehat{\boldsymbol\theta}^{\,n+1}
=
\exp
\left(
\widehat{\boldsymbol\vartheta}^{\,n+1}
\right)
=
\begin{pmatrix}
\widehat R_p^{\,n+1}
&
\widehat R_d^{\,n+1}
&
\widehat C^{\,n+1}
\end{pmatrix}^{T}.
\label{ukf_updated_physical_parameters}
\end{equation}

\paragraph{Accepted fluid-structure variable states update}

After the measurement update, the candidate sigma-point solutions are
discarded. The fluid--Windkessel subproblem is solved once more using the
posterior parameter estimate
$\widehat{\boldsymbol\theta}^{\,n+1}$. This final solve produces the
accepted fluid velocity $\boldsymbol u^{n+1}$, pressure $p^{n+1}$, outlet flow rate $Q^{n+1}$, outlet pressure $p_{\mathrm{out}}^{n+1}$, and
capacitor pressure $P_B^{n+1}$, and then
these accepted variables are carried into the next fluid-structure interaction time step.

\paragraph{Initialization}
Given an initial physical parameter estimate
$\boldsymbol\theta^0=(R_p^0, R_d^0, C^0)^T$, the initial logarithmic mean is
\begin{equation}
\widehat{\boldsymbol\vartheta}^{\,0}
=
\log\boldsymbol\theta^0 
=
\begin{pmatrix}
\log R_p^0\\
\log R_d^0\\
\log C^0
\end{pmatrix}.
\label{ukf_initial_logarithmic_mean}
\end{equation}
The initial logarithmic parameter covariance is
\begin{equation}
\mathbf P_\vartheta^0
=
\left[\log(1+0.20)\right]^2\mathbf I_3.
\label{ukf_initial_covariance}
\end{equation}
Thus, the three logarithmic parameters are initially assigned equal
variances and zero cross covariances. The logarithmic standard deviation
$\log(1.20)$ corresponds to a one-standard-deviation multiplicative range
from $1/1.20$ to $1.20$ around each initial physical parameter estimate.

\section{Numerical results}
\label{numerical_result}

\subsection{Straight-tube pressure-pulse benchmark}
\label{sec:straight_pressure_wave}

We first consider a straight compliant tube subjected to a short inlet-pressure
pulse. This test is based on the pressure-wave propagation benchmark used by
Buka\v{c} et al.~\cite{canic2015stability, bukac2013fluid, bukac2015partitioned}. The purpose of
this example is to verify the implementation of the coupled fluid-Naghdi-shell
model, including the FSI solver based on
the kinematically coupled $\beta$-splitting scheme.
We notice that the original benchmark couples a 2D fluid with a 1D Koiter shell wall model, while in the
present work, we consider the fluid occupying a 3D cylindrical domain and
the wall is modeled by the 2D Naghdi shell introduced in Section~\ref{NaghdiModel}. Therefore, the present test should be understood as
a 3D extension of the pressure-pulse benchmark rather than an exact reproduction. No Windkessel model or parameter estimator is considered in this example. 

Let $L>0$ and $R_0>0$ denote the length and reference midsurface radius of
the straight tube. The reference fluid domain is
\begin{equation}
\hat\Omega_f
=
\left\{
(\hat x_1,\hat x_2,\hat x_3)\in\mathbb R^3:
0<\hat x_1<L,
\quad
\hat x_2^2+\hat x_3^2<R_0^2
\right\}.
\label{benchmark_reference_fluid_domain}
\end{equation}
The fixed shell parameter domain is
$\omg
=
(0,L)\times(0,2\pi R_0)$,
where $x\in(0,L)$ is the axial coordinate and
$y_{\mathrm{arc}}\in(0,2\pi R_0)$ is the circumferential arc-length
coordinate. Therefore, the reference shell midsurface is parameterized by

\begin{equation}
\boldsymbol\varphi(x,y_{\mathrm{arc}})
=
\begin{pmatrix}
x,
R_0\cos\theta,
R_0\sin\theta
\end{pmatrix}.
\label{benchmark_reference_midsurface}
\end{equation}
The corresponding covariant tangent vectors are
\[
\boldsymbol a_1
=
\partial_1\boldsymbol\varphi
=
\begin{pmatrix}
1, 0, 0
\end{pmatrix},
\qquad
\boldsymbol a_2
=
\partial_2\boldsymbol\varphi
=
\begin{pmatrix}
0, -\sin\theta, \cos\theta
\end{pmatrix}.
\]
Consequently,
\[
\boldsymbol a_c
=
\begin{pmatrix}
1&0\\
0&1
\end{pmatrix},
\qquad
\sqrt a
=
\left|
\boldsymbol a_1\times\boldsymbol a_2
\right|
=
1.
\]

In this benchmark, the fluid is driven by a prescribed inlet pressure, with the corresponding fluid boundary conditions on the inlet and outlet as follows
\begin{align}
\begin{cases}
\boldsymbol\sigma_f(\boldsymbol u,p)\boldsymbol n_f
&=
-p_{\mathrm{in}}(t)\boldsymbol n_f
\qquad
\text{on }\Gamma_f^{\mathrm{in}}(t),\\
\boldsymbol\sigma_f(\boldsymbol u,p)\boldsymbol n_f
&=
-p_{\mathrm{out}}(t)\boldsymbol n_f
=
\boldsymbol0
\qquad
\text{on }\Gamma_f^{\mathrm{out}}(t).
\end{cases}
\label{benchmark_inout_pressure_condition}    
\end{align}
The inlet-pressure pulse is given by
\begin{equation}
p_{\mathrm{in}}(t)
=
\begin{cases}
\dfrac{p_{\max}}{2}
\left[
1-
\cos\left(
\dfrac{2\pi t}{t_{\max}}
\right)
\right],
&
0\leq t\leq t_{\max},
\\[2ex]
0,
&
t>t_{\max},
\end{cases}
\label{benchmark_inlet_pressure_pulse}
\end{equation}
where
$p_{\max}
=
2.0\times10^3\ {\rm Pa},
~
t_{\max}
=
5.0\times10^{-3}\ {\rm s}$.

The Naghdi shell variables are periodic in the circumferential direction and the two ends are clamped
\begin{equation}
\begin{aligned}
&\boldsymbol\eta(x,0,t)
=
\boldsymbol\eta(x,2\pi R_0,t),
&&
\omb(x,0,t)
=
\omb(x,2\pi R_0,t).\\
&\boldsymbol\eta(0,y_{\mathrm{arc}},t)
=
\boldsymbol0,
&&
\omb(0,y_{\mathrm{arc}},t)
=
\boldsymbol0,
\\
&\boldsymbol\eta(L,y_{\mathrm{arc}},t)
=
\boldsymbol0,
&&
\omb(L,y_{\mathrm{arc}},t)
=
\boldsymbol0.
\end{aligned}
\label{benchmark_shell_clamped_conditions}
\end{equation}

We remark that the original benchmark in \cite{canic2015stability, bukac2013fluid, bukac2015partitioned} employs absorbing conditions for the radial
wall displacement. Because those conditions were derived for a scalar
radial string model, they are not imposed directly on the vector-valued
Naghdi displacement and rotation fields. The clamped conditions may produce
localized boundary layers near the inlet and outlet. Therefore, comparisons
with the published benchmark are primarily made in the interior of the tube.

Regarding the initial condition, the fluid and shell are initially set at rest
\begin{equation}
\boldsymbol u(\cdot,0)
=
\boldsymbol0,~
p(\cdot,0)
=
0,
~\boldsymbol\eta(\cdot,0)
=
\boldsymbol0,
~\partial_t\boldsymbol\eta(\cdot,0)
=
\boldsymbol0,
~\omb(\cdot,0)
=
\boldsymbol0,
~\partial_t\omb(\cdot,0)
=
\boldsymbol0.
\label{benchmark_initial_conditions}
\end{equation}

The geometric, fluid, and shell parameters used in the benchmark are listed
in Table~\ref{tab:pressure_wave_parameters}. The original values are
converted from the centimeter-gram-second system to SI units.
\begin{table}[htbp]
\centering
\caption{Parameters for the straight-tube pressure-pulse benchmark.}
\label{tab:pressure_wave_parameters}
\renewcommand{\arraystretch}{1.2}
\setlength{\tabcolsep}{5pt}
\begin{tabular}{lll@{\hspace{0.5em}}lll}
\hline
Parameter name & Symbol & Value (SI units)
&
Parameter name & Symbol & Value (SI units)
\\
\hline
Tube length
&
$L$
&
$6.0\times10^{-2}\ \mathrm{m}$
&
Reference radius
&
$R_0$
&
$5.0\times10^{-3}\ \mathrm{m}$
\\
Shell thickness
&
$\varepsilon$
&
$1.0\times10^{-3}\ \mathrm{m}$
&
Fluid density
&
$\rho_f$
&
$1.0\times10^{3}\ \mathrm{kg\,m^{-3}}$
\\

Fluid dynamic viscosity
&
$\mu_f$
&
$3.5\times10^{-3}\ \mathrm{Pa\,s}$
&
Shell density
&
$\rho_s$
&
$1.1\times10^{3}\ \mathrm{kg\,m^{-3}}$
\\
Young's modulus
&
$E$
&
$7.5\times10^{4}\ \mathrm{Pa}$
&
Poisson ratio
&
$\nu_s$
&
$0.5$
\\
Maximum inlet pressure
&
$p_{\max}$
&
$2.0\times10^{3}\ \mathrm{Pa}$
&
Pulse duration
&
$t_{\max}$
&
$5.0\times10^{-3}\ \mathrm{s}$
\\
Final simulation time
&
$T$
&
$1.2\times10^{-2}\ \mathrm{s}$
&
Splitting parameter
&
$\beta$
&
$1$
\\
\hline
\end{tabular}
\end{table}
Here, the value $\nu_s=0.5$ should not be substituted into the 3D
Lam\'e coefficient
\[
\lambda_s
=
\frac{E\nu_s}
{(1+\nu_s)(1-2\nu_s)},
\]
because this expression is singular at $\nu_s=0.5$. For our thin shell, since the normal stress through the shell thickness is assumed to vanish, by eliminating the thickness strain from the 3D constitutive law will produce the following
plane-stress coefficients:
\begin{equation}
\mu_s
=
\frac{E}{2(1+\nu_s)},
\qquad
\lambda_s^{\mathrm{ps}}
=
\frac{E\nu_s}{1-\nu_s^2}.
\label{benchmark_plane_stress_coefficients}
\end{equation}
Based on the parameter values given in
Table~\ref{tab:pressure_wave_parameters}, these coefficients are
\[
\mu_s
=
2.5\times10^4\ {\rm Pa},
\qquad
\lambda_s^{\mathrm{ps}}
=
5.0\times10^4\ {\rm Pa}.
\]

We notice that the original benchmark problem also contains a structural damping term, but our Naghdi shell model considered does not include this additional damping. Therefore, exact point-wise agreement with the published viscoelastic-wall solution is not expected. Instead, quantitative discretization errors are evaluated relative to a numerically converged solution of the same 3D
fluid-Naghdi-shell problem. In particular, the coupled problem is advanced using the splitting scheme introduced in
Section~\ref{sec:betascheme}, with $\beta=1$. 
For the spatial convergence study, three successive mesh grids were used. Here, $n_x$ and $n_\theta$ denote the numbers of
axial and circumferential shell element respectively, and $\ell_c$
denotes the target interior fluid mesh spacing. The coarse mesh, generated
with
\[
(n_x,n_\theta,\ell_c)
=
(24,16,2.0\times10^{-3}\ {\rm m}),
\]
contained $850$ fluid vertices, $3,600$ tetrahedrons, and $768$ shell triangles.
The medium mesh, generated with
\[
(n_x,n_\theta,\ell_c)
=
(40,24,1.0\times10^{-3}\ {\rm m}),
\]
contained $3,690$ fluid vertices, $18,480$ tetrahedrons, and $1,920$ shell triangles. The fine mesh, generated with
\[
(n_x,n_\theta,\ell_c)
=
(60,32,5.0\times10^{-4}\ {\rm m}),
\]
contained $19,520$ fluid vertices, $109,080$ tetrahedrons, and $3,840$ shell triangles. The corresponding polygonal-cylinder volumes were
\[
4.59\times10^{-6},\quad
4.66\times10^{-6},\quad
4.68\times10^{-6}\ {\rm m}^3,
\]
respectively, approaching the exact cylindrical volume
$\pi R_0^2L
=
4.712389\times10^{-6}\ {\rm m}^3$.
At every refinement level, the fluid and shell meshes coincided on the fluid structure interface. 
For the temporal convergence study, a sequence of successively refined time
steps may be used, 
$\Delta t
=
\left\{
1.0\times10^{-4},
\;
5.0\times10^{-5},
\;
2.5\times10^{-5}
\right\}
\ {\rm s}$.
\begin{figure}[htbp]
    \centering
    \begin{subfigure}[t]{0.8\textwidth}
        \centering
     \includegraphics[width=\linewidth]{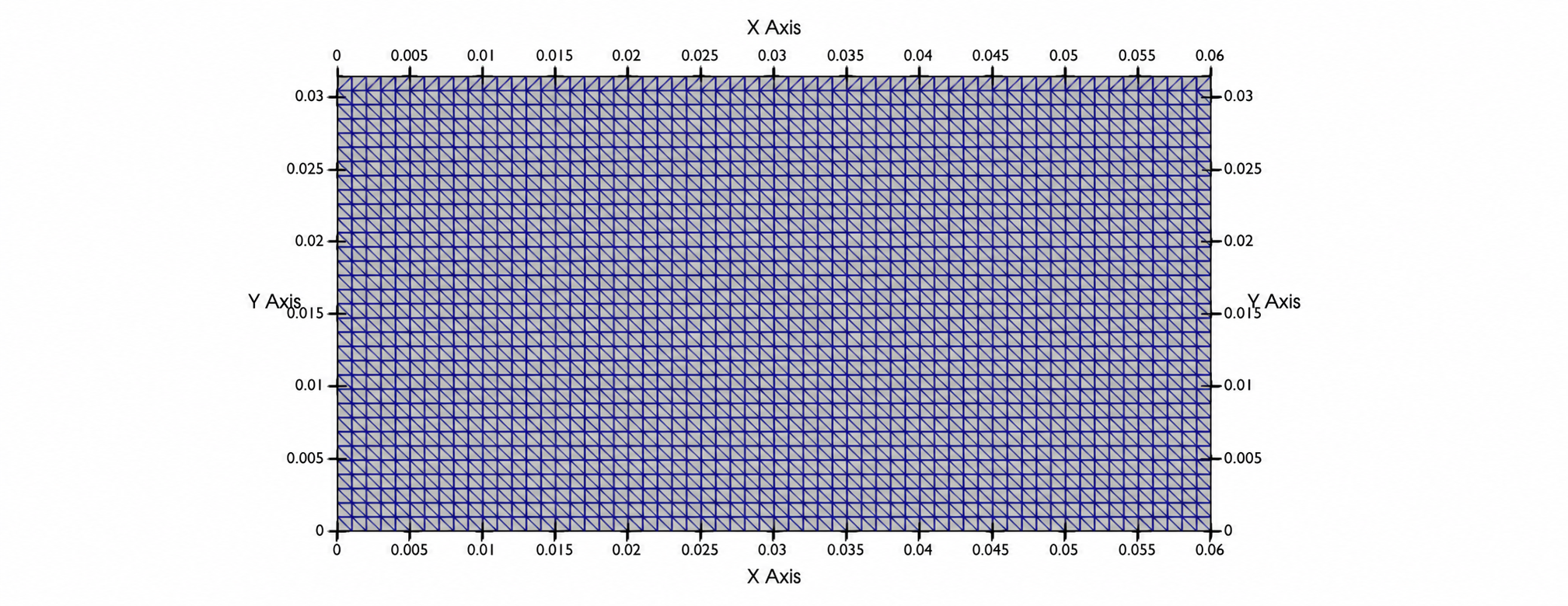}
        \caption{Shell mesh}
        \label{fig:shell_mesh}
    \end{subfigure}
    \hfill
    \begin{subfigure}[t]{0.75\textwidth}
        \centering
        \includegraphics[width=\linewidth]{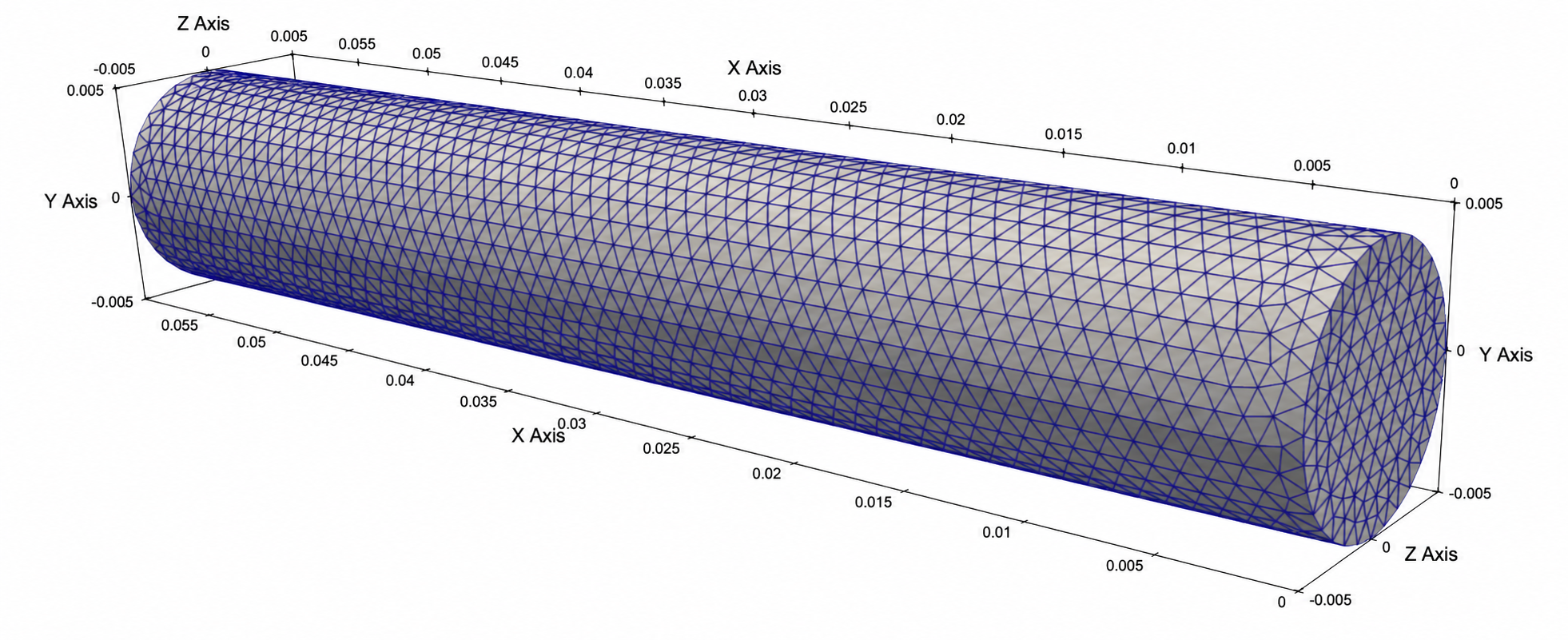}
        \caption{Fluid mesh}
     \label{fig:fluid_mesh}
    \end{subfigure}
\caption{Computational meshes used for the pressure-wave benchmark. The shell mesh consists of 3,840 triangular elements, while the fluid mesh consists of 109,080 tetrahedral elements.}
\label{fig:ex1_shell_fluid_meshes}
\end{figure}

We report the cross-section $S(\boldsymbol{x},t)$ averaged pressure which is defined as below
\begin{equation}
\overline p(x,t)
=
\frac{1}
{\left|\mathcal S(\boldsymbol{x},t)\right|}
\int_{\mathcal S(\boldsymbol{x},t)}
p(\boldsymbol{x},t)
\,dA.
\label{benchmark_mean_pressure}
\end{equation}
We also report the mean diameter based on the following formula. We partition the shell parameter domain $\omg$ into $N$ nonoverlapping vertical strips, denoted by $\omg_i=\langle x_{i-1},x_i\rangle \times \langle 0, 2\pi R_0 \rangle$, $i=1,\ldots,N$. For each strip $\omg_i$, the mean radial displacement is calculated as
\begin{equation}
\overline{\eta}_{r,i}(t)
=
\frac{
\displaystyle
\int_{\omg_i}
\boldsymbol\eta(x,y_{\mathrm{arc}},t)
\cdot
\boldsymbol e_r(y_{\mathrm{arc}})
\,dx\,dy_{\mathrm{arc}}
}{
\displaystyle
\int_{\omg_i}
1
\,dx\,dy_{\mathrm{arc}}
},
\label{mean_radial_displacement}
\end{equation}
where the outward radial unit vector  $\boldsymbol e_r(y_{\mathrm{arc}})$ of the reference cylinder is
\[
\boldsymbol e_r(y_{\mathrm{arc}})
=
\begin{pmatrix}
0, 
\cos\left(y_{\mathrm{arc}}/R_0\right), 
\sin\left(y_{\mathrm{arc}}/R_0\right).
\end{pmatrix}
\]
Then, the corresponding mean diameter is then defined by
\begin{equation}
\overline D_i(t)
=
2
\left(
R_0+\overline{\eta}_{r,i}(t)
\right).
\label{mean_diameter}
\end{equation}
We studied the temporal convergence and spatial convergence by reporting the averaged pressure profiles of $\overline p(x,t)$ and the mean diameter $\overline D_i(t)$ at
$t=2, 4, 6, 8, 10, 12~ {\rm ms}$. The results are presented in Figures~\ref{fig:spatial_axial_pressure_convergence}, \ref{fig:spatial_axial_diameter_convergence}, \ref{fig:temporal_axial_pressure_convergence}, and \ref{fig:temporal_axial_diameter_convergence} respectively. Both temporal and spatial results provide an evidence of convergence.
\begin{figure}[!htbp]
    \centering
    \includegraphics[width=0.99\textwidth]
{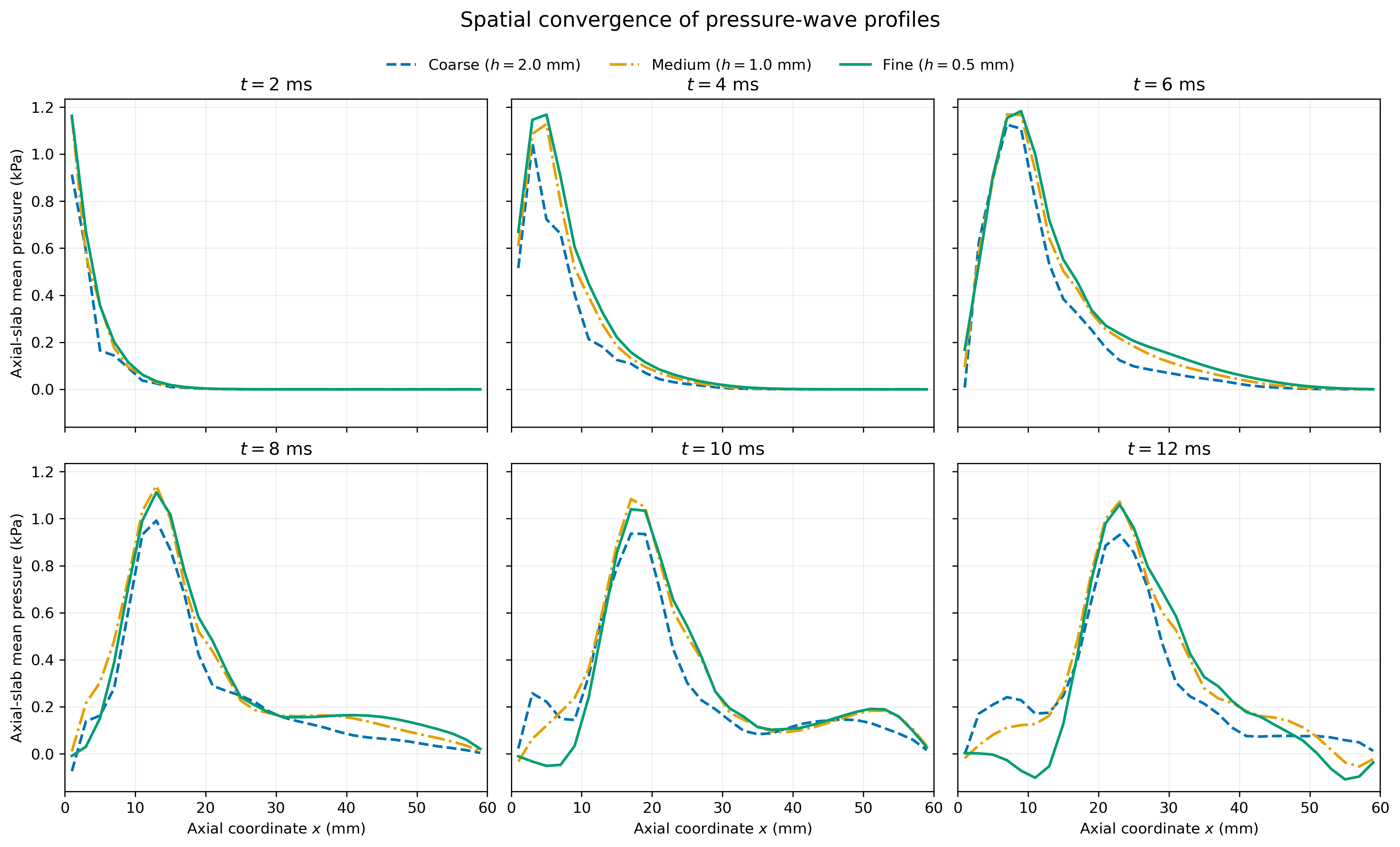}
    \caption{Spatial convergence comparison of the axial averaged pressure
    profiles at $t=2$, $4$, $6$, $8$, $10$, and $12~\mathrm{ms}$ for the original Naghdi shell model.
    The curves correspond to the coarse, medium, and fine spatial
    discretizations, while the time-step size is held fixed with $\Delta t= 2.5\times10^{-5}$.}
    \label{fig:spatial_axial_pressure_convergence}
\end{figure}

\begin{figure}[!htbp]
    \centering
    \includegraphics[width=0.99\textwidth]
{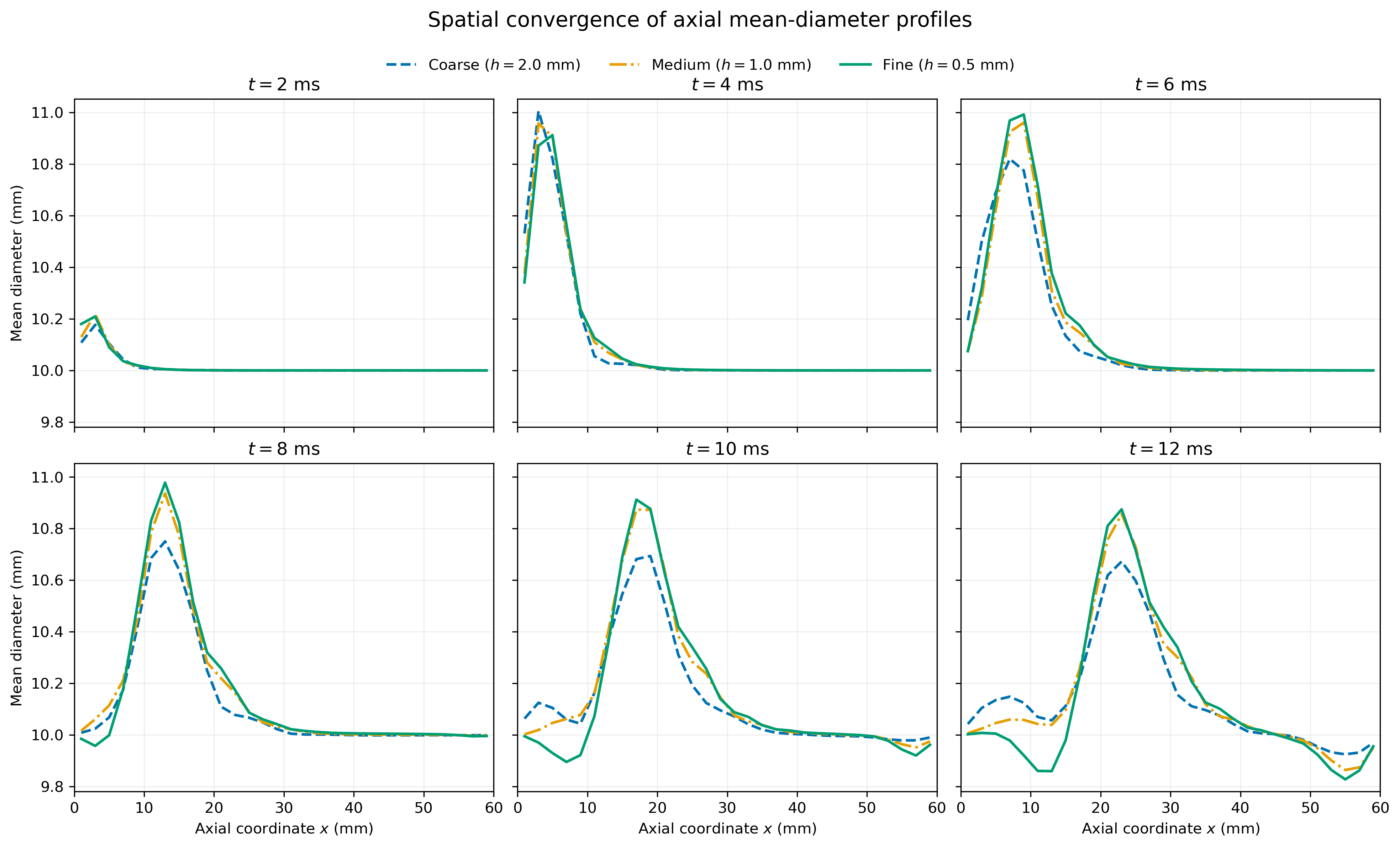}
    \caption{Spatial convergence comparison of the axial mean diameter
    profiles at $t=2$, $4$, $6$, $8$, $10$, and $12~\mathrm{ms}$ for the original Naghdi shell model.
    The curves correspond to the coarse, medium, and fine spatial
    discretizations, while the time-step size is held fixed $\Delta t= 2.5\times10^{-5}$.}
\label{fig:spatial_axial_diameter_convergence}
\end{figure}

\begin{figure}[!htbp]
    \centering
    \includegraphics[width=0.99\textwidth]
{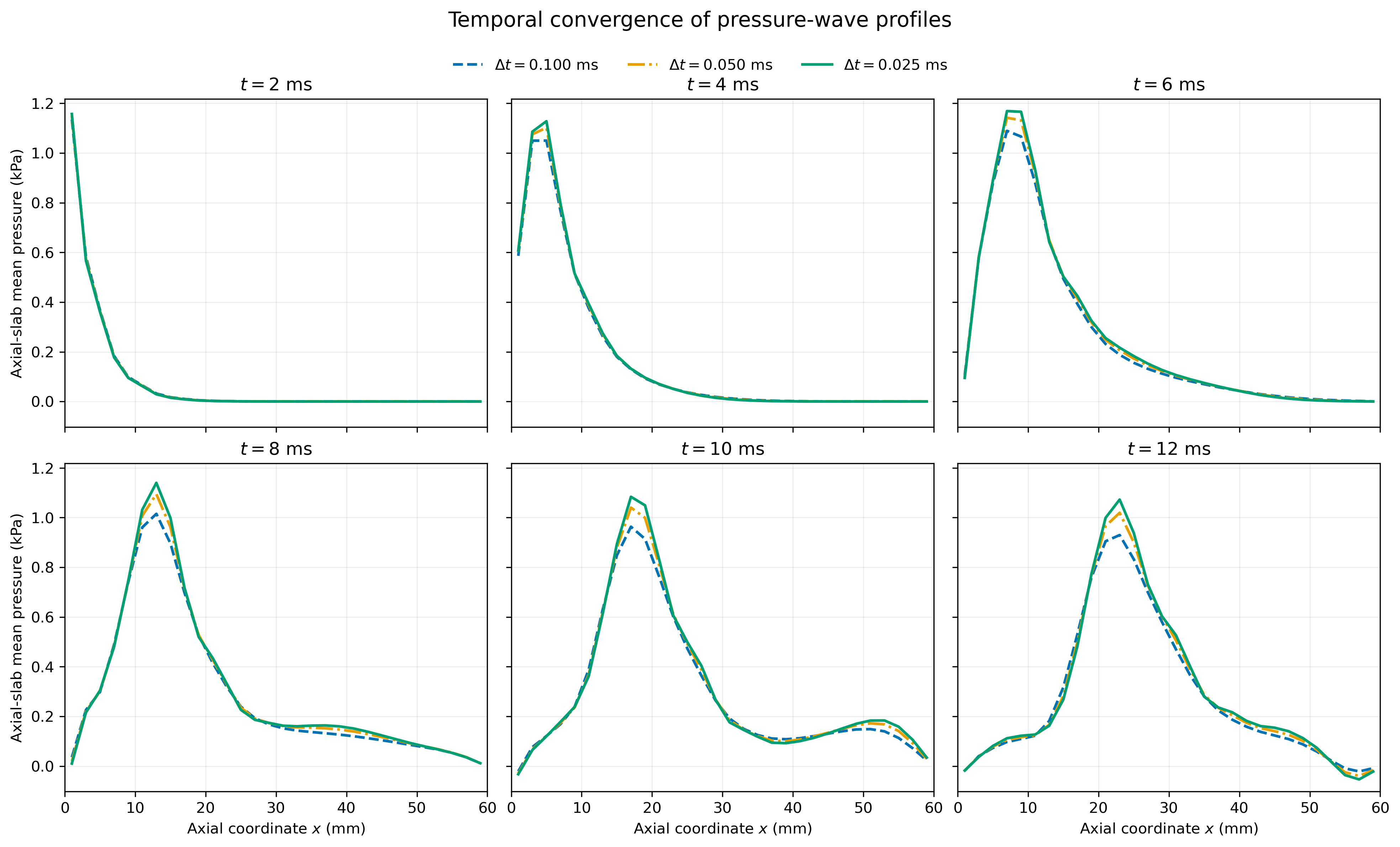}
    \caption{Temporal convergence comparison of the axial averaged pressure
    profiles at $t=2$, $4$, $6$, $8$, $10$, and $12~\mathrm{ms}$ for the original Naghdi shell model.
    The curves correspond to successively refined time-step sizes,
    while the spatial discretization is fixed at the fine-mesh level.}
    \label{fig:temporal_axial_pressure_convergence}
\end{figure}

\begin{figure}[!htbp]
    \centering
    \includegraphics[width=0.99\textwidth]
{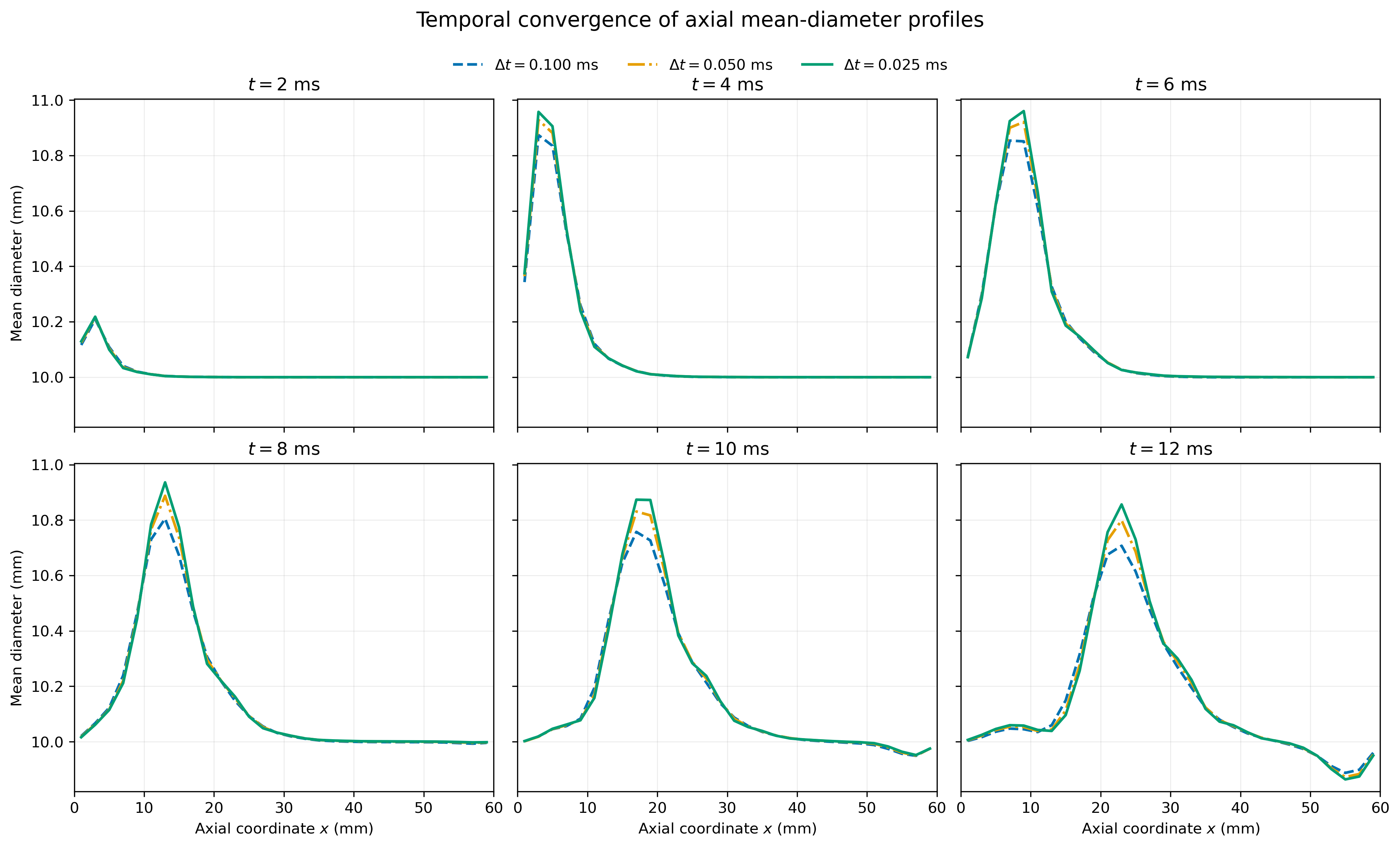}
    \caption{Temporal convergence comparison of the axial mean diameter
    profiles at $t=2$, $4$, $6$, $8$, $10$, and $12~\mathrm{ms}$ for the original Naghdi shell model.
    The curves correspond to successively refined time-step sizes,
    while the spatial discretization is fixed at the fine-mesh level.}
    \label{fig:temporal_axial_diameter_convergence}
\end{figure}

\subsection{Comparison with Koiter shell with radial displacement only}
\label{sec:straight_pressure_wave_penalized}
For this comparison, we modified our Naghdi shell constitutive equation and included a few penalty terms. 
{\color{black}Specifically, to suppress the shear (transverse and drilling) and axial motion, we introduce three penalty parameters} 
\[
\gamma_{\mathrm{sh}}\geq 0,
\qquad
\gamma_x\geq 0,
\qquad
\gamma_{\mathrm{sym}}\geq 0
\]
denote the transverse-shear, bending, axial-displacement, and tangential
symmetry parameters, respectively.

\paragraph{Scaling of the transverse-shear term}

The transverse-shear contribution
involving
\[
\mathcal B_m=\mu_s\boldsymbol a^c
\]
is modified. Specifically, the original constitutive contraction $\mathcal C_m
\widehat{\mathbf E}_{\mathrm{mem}}
(\boldsymbol\eta,\omb)
\cdot
\widehat{\mathbf D}_{\mathrm{mem}}
(\boldsymbol v,\boldsymbol w)$ defined in \eqref{membrane_tensor}
is modified by
\begin{equation}
\begin{aligned}
\mathcal C_m^{(\gamma_{\mathrm{sh}})}
\widehat{\mathbf E}_{\mathrm{mem}}
(\boldsymbol\eta,\omb)
\cdot
\widehat{\mathbf D}_{\mathrm{mem}}
(\boldsymbol v,\boldsymbol w)
={}&
\frac{2\lambda_s\mu_s}
{\lambda_s+2\mu_s}
\operatorname{tr}
\left(
\mathbf A(\boldsymbol\eta,\omb)
\right)
\operatorname{tr}
\left(
\mathbf D(\boldsymbol v,\boldsymbol w)
\right)
\\
&+
2\mu_s
\left(
\boldsymbol a_c
\mathbf A(\boldsymbol\eta,\omb)
\boldsymbol a^c
\right)
:
\mathbf D(\boldsymbol v,\boldsymbol w)
+
\gamma_{\mathrm{sh}}\,
\mathcal B_m
\boldsymbol e(\boldsymbol\eta,\omb)
\cdot
\boldsymbol d(\boldsymbol v,\boldsymbol w),
\end{aligned}
\label{scaled_transverse_shear_constitutive_form}
\end{equation}
where, the in-plane membrane contributions remain unchanged. Choosing
{\color{black}$\gamma_{\mathrm{sh}}\gg 1$ reduces the transverse-shear,}
whereas \(\gamma_{\mathrm{sh}}=1\) recovers the original constitutive model.

Let
$
B_m^{(\gamma_{\mathrm{sh}})}
\left(
(\boldsymbol\eta,\omb),
(\boldsymbol v,\boldsymbol w)
\right)
$
denote the membrane--shear bilinear form obtained using
equation~\eqref{scaled_transverse_shear_constitutive_form}.



\paragraph{Axial-displacement penalty}

Let
$\boldsymbol e_x=(1,0,0)^T$
denote the unit vector along the axial direction of the vessel, and define
\[
\eta_x
=
\boldsymbol\eta\cdot\boldsymbol e_x,
\qquad
v_x
=
\boldsymbol v\cdot\boldsymbol e_x.
\]
To suppress axial displacement, we add the symmetric penalty term
\begin{equation}
P_x(\boldsymbol\eta,\boldsymbol v)
=
\gamma_x
\frac{E_s\varepsilon}{R_0^2}
\int_\omg
\eta_x v_x
\sqrt a
\,dx\,dy_{\mathrm{arc}}.
\label{axial_displacement_penalty}
\end{equation}
The factor \(E_s\varepsilon/R_0^2\) supplies the appropriate stiffness
scaling, while \(\gamma_x\) is dimensionless. This term is the first
variation of the nonnegative penalty energy
\[
\mathcal E_x(\boldsymbol\eta)
=
\frac{\gamma_x}{2}
\frac{E_s\varepsilon}{R_0^2}
\int_\omg
\eta_x^2
\sqrt a
\,dx\,dy_{\mathrm{arc}}.
\]
Consequently, increasing \(\gamma_x\) increasingly suppresses
\(\eta_x\), although a finite penalty parameter does not impose
\(\eta_x=0\) exactly.

\paragraph{Penalty on the off-diagonal difference}

To encourage symmetry of the local tangential membrane matrix, we penalize
the difference between its two off-diagonal entries. For the trial fields,
these entries are
\[
\begin{aligned}
A_{12}(\boldsymbol\eta,\omb)
&=
\boldsymbol a^1\cdot\partial_2\boldsymbol\eta
+
\sqrt a\,a^{11}
\omb\cdot\boldsymbol a_3,
\\
A_{21}(\boldsymbol\eta,\omb)
&=
\boldsymbol a^2\cdot\partial_1\boldsymbol\eta
-
\sqrt a\,a^{22}
\omb\cdot\boldsymbol a_3.
\end{aligned}
\]
The corresponding entries for the test fields are
\[
\begin{aligned}
D_{12}(\boldsymbol v,\boldsymbol w)
&=
\boldsymbol a^1\cdot\partial_2\boldsymbol v
+
\sqrt a\,a^{11}
\boldsymbol w\cdot\boldsymbol a_3,
\\
D_{21}(\boldsymbol v,\boldsymbol w)
&=
\boldsymbol a^2\cdot\partial_1\boldsymbol v
-
\sqrt a\,a^{22}
\boldsymbol w\cdot\boldsymbol a_3.
\end{aligned}
\]
Therefore, the added symmetry penalty is
\begin{equation}
\begin{aligned}
P_{\mathrm{sym}}
\left(
(\boldsymbol\eta,\omb),
(\boldsymbol v,\boldsymbol w)
\right)
={}&
\gamma_{\mathrm{sym}}E_s\varepsilon
\int_\omg
\left[
A_{12}(\boldsymbol\eta,\omb)
-
A_{21}(\boldsymbol\eta,\omb)
\right]
\times
\left[
D_{12}(\boldsymbol v,\boldsymbol w)
-
D_{21}(\boldsymbol v,\boldsymbol w)
\right]
\sqrt a
\,dx\,dy_{\mathrm{arc}}.
\end{aligned}
\label{off_diagonal_symmetry_penalty}
\end{equation}
This term is the first variation of
\[
\mathcal E_{\mathrm{sym}}
(\boldsymbol\eta,\omb)
=
\frac{\gamma_{\mathrm{sym}}E_s\varepsilon}{2}
\int_\omg
\left[
A_{12}(\boldsymbol\eta,\omb)
-
A_{21}(\boldsymbol\eta,\omb)
\right]^2
\sqrt a
\,dx\,dy_{\mathrm{arc}}.
\]
It is therefore symmetric and nonnegative. A sufficiently large value of
\(\gamma_{\mathrm{sym}}\) enforces
$A_{12}(\boldsymbol\eta,\omb)
\approx
A_{21}(\boldsymbol\eta,\omb)$, whereas exact equality is obtained only in the limiting case or by imposing
symmetry as a hard constraint.

For the straight, unbanded cylindrical geometry,
$\sqrt a=1,
\quad
a^{11}=a^{22}=1$,
so the off-diagonal difference reduces to
$
A_{12}-A_{21}
=
\boldsymbol a^1\cdot\partial_2\boldsymbol\eta
-
\boldsymbol a^2\cdot\partial_1\boldsymbol\eta
+
2\omb\cdot\boldsymbol a_3$.

\paragraph{Modified shell weak equation}

With these modifications, the shell subproblem becomes: find
\((\boldsymbol\eta,\omb)\in V_N\) such that, for every
\((\boldsymbol v,\boldsymbol w)\in V_N\),
\begin{equation}
\begin{aligned}
&\varepsilon\rho_s
\int_\omg
\partial_{tt}\boldsymbol\eta
\cdot\boldsymbol v
\sqrt a
\,dx\,dy_{\mathrm{arc}}
+
\frac{\varepsilon^3\rho_s}{12}
\int_\omg
\partial_{tt}\omb
\cdot\boldsymbol w
\sqrt a
\,dx\,dy_{\mathrm{arc}}
\\
&\quad
+
B_m^{(\gamma_{\mathrm{sh}})}
\left(
(\boldsymbol\eta,\omb),
(\boldsymbol v,\boldsymbol w)
\right)
+
\gamma_b
B_b
\left(
(\boldsymbol\eta,\omb),
(\boldsymbol v,\boldsymbol w)
\right)
+
P_x(\boldsymbol\eta,\boldsymbol v)
+
P_{\mathrm{sym}}
\left(
(\boldsymbol\eta,\omb),
(\boldsymbol v,\boldsymbol w)
\right)
\\
&=
\int_\omg
\boldsymbol f_s^{\mathrm{ref}}
\cdot\boldsymbol v
\,dx\,dy_{\mathrm{arc}}.
\end{aligned}
\label{modified_shell_weak_equation}
\end{equation}

The original shell formulation is recovered by taking
$\gamma_{\mathrm{sh}}=1,
~\gamma_x=0,
~\gamma_{\mathrm{sym}}=0.$
To suppress transverse shear, 
axial motion and
the antisymmetric tangential component, we choose
$
\gamma_{\mathrm{sh}}=10, 
~\gamma_x=10,
~\gamma_{\mathrm{sym}}=10$.

Similar as we did in Section \ref{sec:straight_pressure_wave}, we report the temporal convergence and spatial convergence studies for the averaged pressure profiles and the mean diameter at
$t=2, 4, 6, 8, 10, 12~ {\rm ms}$. The results are presented in Figures~\ref{fig:penalized_spatial_axial_pressure_convergence}, \ref{fig:penalized_spatial_axial_diameter_convergence}, \ref{fig:penalized_temporal_axial_pressure_convergence}, and \ref{fig:penalized_temporal_axial_diameter_convergence} respectively. Again, both the temporal and spatial results provide evidence of convergence. The pressure and diameter profiles are consistent with those reported in \cite{canic2015stability, bukac2013fluid, bukac2015partitioned}, except for a difference in the wave speed, which is addressed in Remark~\ref{remark7_1}.

\begin{figure}[!htbp]
    \centering
    \includegraphics[width=0.99\textwidth]
{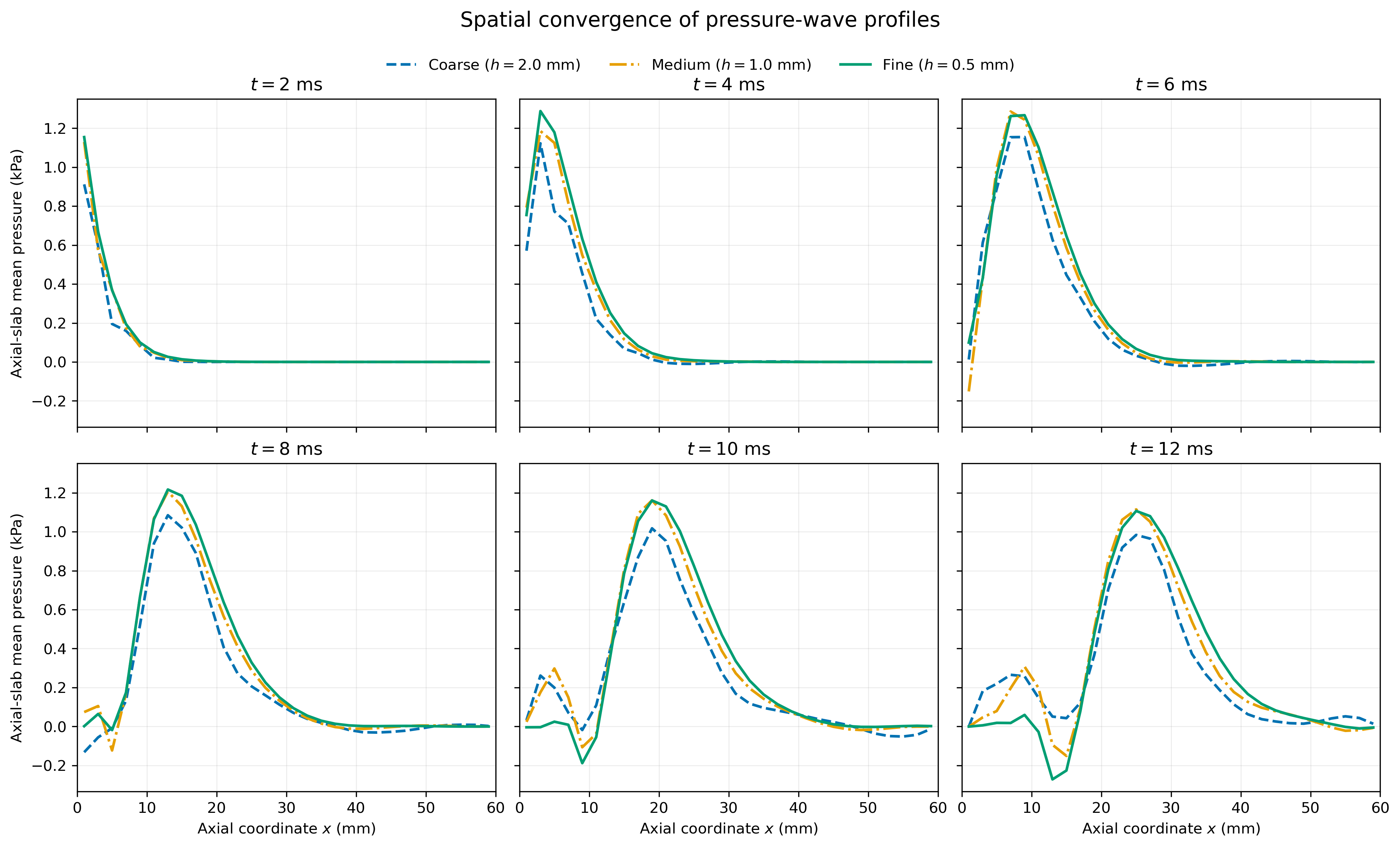}
    \caption{Spatial convergence comparison of the axial averaged pressure
    profiles at $t=2$, $4$, $6$, $8$, $10$, and $12~\mathrm{ms}$ for the modified Naghdi shell model.
    The curves correspond to the coarse, medium, and fine spatial
    discretizations, while the time-step size is held fixed with $\Delta t= 2.5\times10^{-5}$.}
    \label{fig:penalized_spatial_axial_pressure_convergence}
\end{figure}

\begin{figure}[!htbp]
    \centering
    \includegraphics[width=0.99\textwidth]
{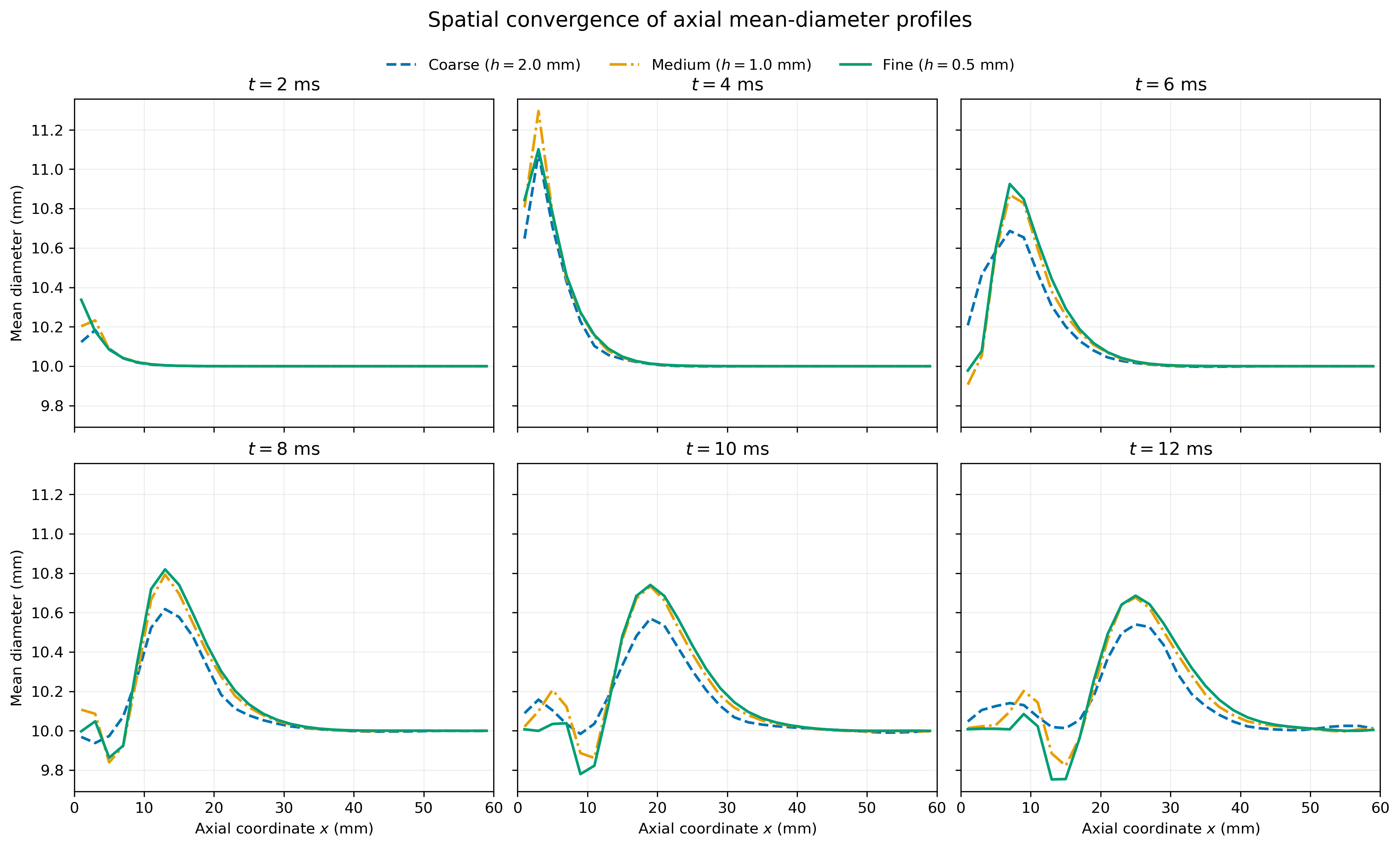}
    \caption{Spatial convergence comparison of the axial mean diameter
    profiles at $t=2$, $4$, $6$, $8$, $10$, and $12~\mathrm{ms}$ for the modified Naghdi shell model.
    The curves correspond to the coarse, medium, and fine spatial
    discretizations, while the time-step size is held fixed $\Delta t= 2.5\times10^{-5}$.}
\label{fig:penalized_spatial_axial_diameter_convergence}
\end{figure}

\begin{figure}[!htbp]
    \centering
    \includegraphics[width=0.99\textwidth]
{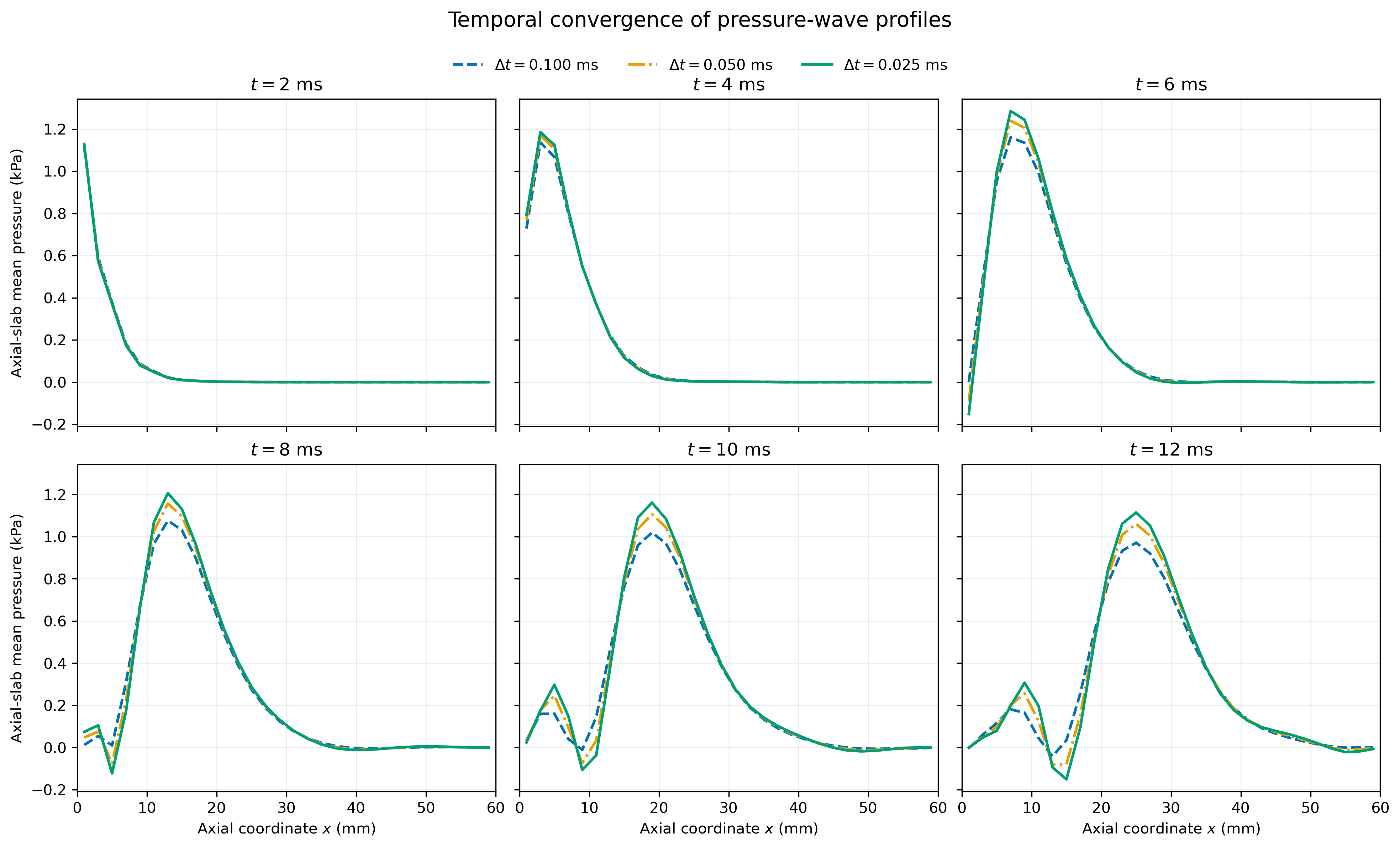}
    \caption{Temporal convergence comparison of the axial averaged pressure
    profiles at $t=2$, $4$, $6$, $8$, $10$, and $12~\mathrm{ms}$ for the modified Naghdi shell model.
    The curves correspond to successively refined time-step sizes,
    while the spatial discretization is fixed at the fine-mesh level.}
    \label{fig:penalized_temporal_axial_pressure_convergence}
\end{figure}

\begin{figure}[!htbp]
    \centering
    \includegraphics[width=0.99\textwidth]
{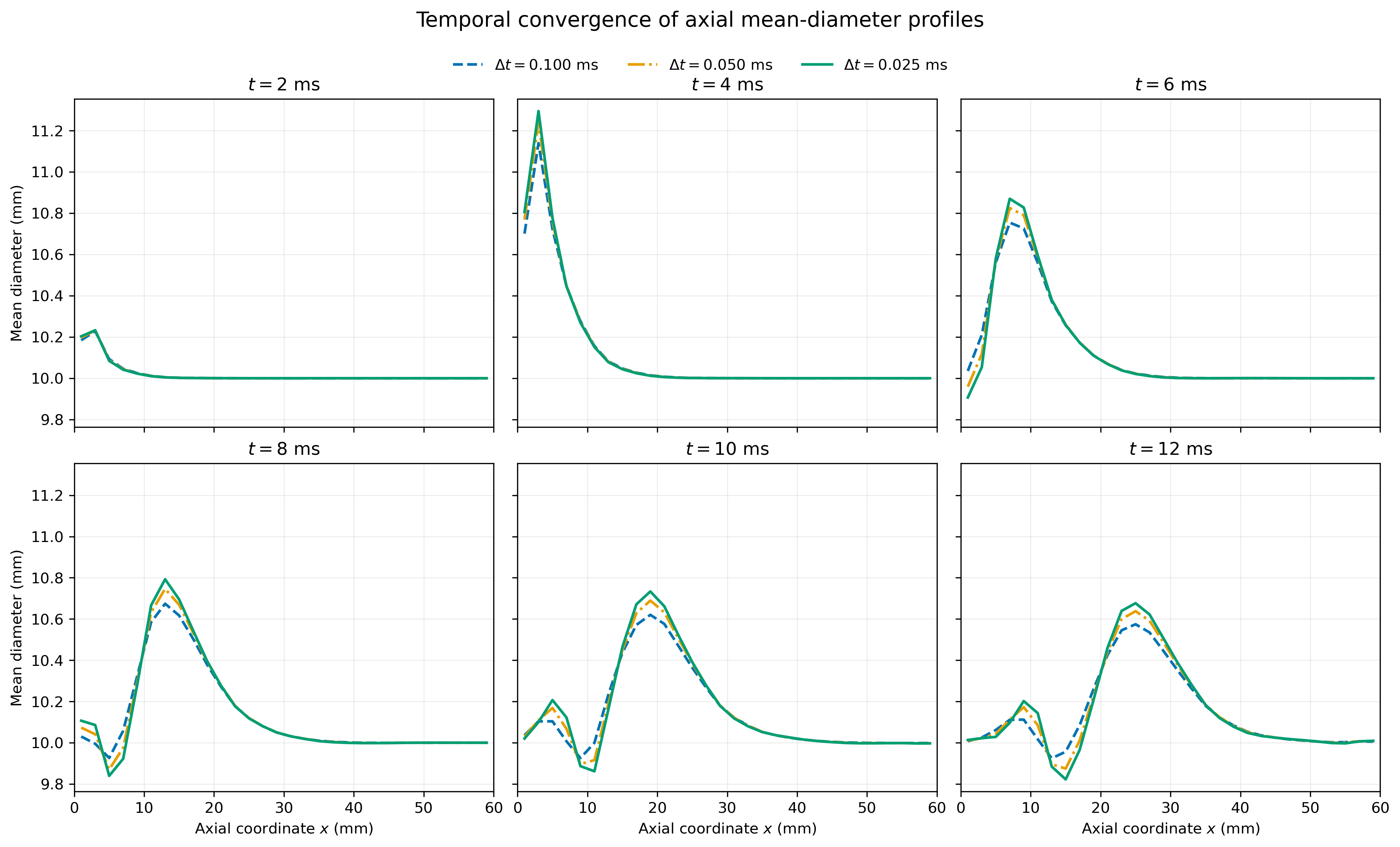}
    \caption{Temporal convergence comparison of the axial mean diameter
    profiles at $t=2$, $4$, $6$, $8$, $10$, and $12~\mathrm{ms}$ for the modified Naghdi shell model.
    The curves correspond to successively refined time-step sizes,
    while the spatial discretization is fixed at the fine-mesh level.}
    \label{fig:penalized_temporal_axial_diameter_convergence}
\end{figure}

\begin{remark}
[Comparison of wave speeds in the 2D and 3D models]
\label{remark7_1}
For small-amplitude, long-wavelength disturbances in a compliant vessel, the
wave speed is approximately
\[
c^2
\approx
\frac{A_0}
{\rho_f
\left.
\dfrac{dA}{dp}
\right|_{A=A_0}},
\]
where $A_0$ is the undeformed cross-sectional area and
$\left.dA/dp\right|_{A=A_0}$ is the tangent cross-sectional compliance.

To isolate the effect of dimensionality, suppose that the 2D and 3D models
satisfy the same linear pressure radial displacement relation, we have
\[
p=K\eta_r,
\qquad
\frac{d\eta_r}{dp}=\frac{1}{K},
\]
where $K$ is the radial stiffness per unit displacement. For the 2D case, the cross-sectional area per unit out of plane depth is
\[
A_{\mathrm{2D}}=R+\eta_r.
\]
Therefore,
\[
\frac{dA_{\mathrm{2D}}}{dp}
=
\frac{dA_{\mathrm{2D}}}{d\eta_r}
\frac{d\eta_r}{dp}
=
\frac{1}{K},
\]
and thus
\[
c_{\mathrm{2D}}^2
\approx
\frac{KR}{\rho_f}.
\]
While for the 3D case, the cross-sectional area of the cylindrical tube is given by
\[
A_{\mathrm{3D}}
=
\pi(R+\eta_r)^2.
\]
Linearization about the undeformed configuration $\eta_r=0$ gives
\[
A_{\mathrm{3D}}
=
\pi R^2
+
2\pi R\eta_r
+
\mathcal O(\eta_r^2).
\]
Hence,
\[
\frac{dA_{\mathrm{3D}}}{dp}
=
2\pi R\frac{d\eta_r}{dp}
=
\frac{2\pi R}{K}.
\]
It follows that
\[
c_{\mathrm{3D}}^2
\approx
\frac{\pi R^2}
{\rho_f(2\pi R/K)}
=
\frac{KR}{2\rho_f}
=
\frac{1}{2}c_{\mathrm{2D}}^2,
\]
and therefore the wave speeds satisfy that
\[
\boxed{
c_{\mathrm{3D}}
\approx
\frac{1}{\sqrt{2}}\,
c_{\mathrm{2D}}.
}
\]
This means that under the assumed common radial stiffness $K$, the relative cross-sectional area change in the 3D model is twice that in the 2D model.
The 3D model consequently has a
lower pressure-wave speed, which is about $\sqrt{2}$ times slower.
\end{remark}

\begin{figure}[!htbp]
    \centering
    \captionsetup[subfigure]{font=small,skip=2pt}
    \begin{subfigure}[t]{0.49\textwidth}
        \centering
    \includegraphics[width=\linewidth]{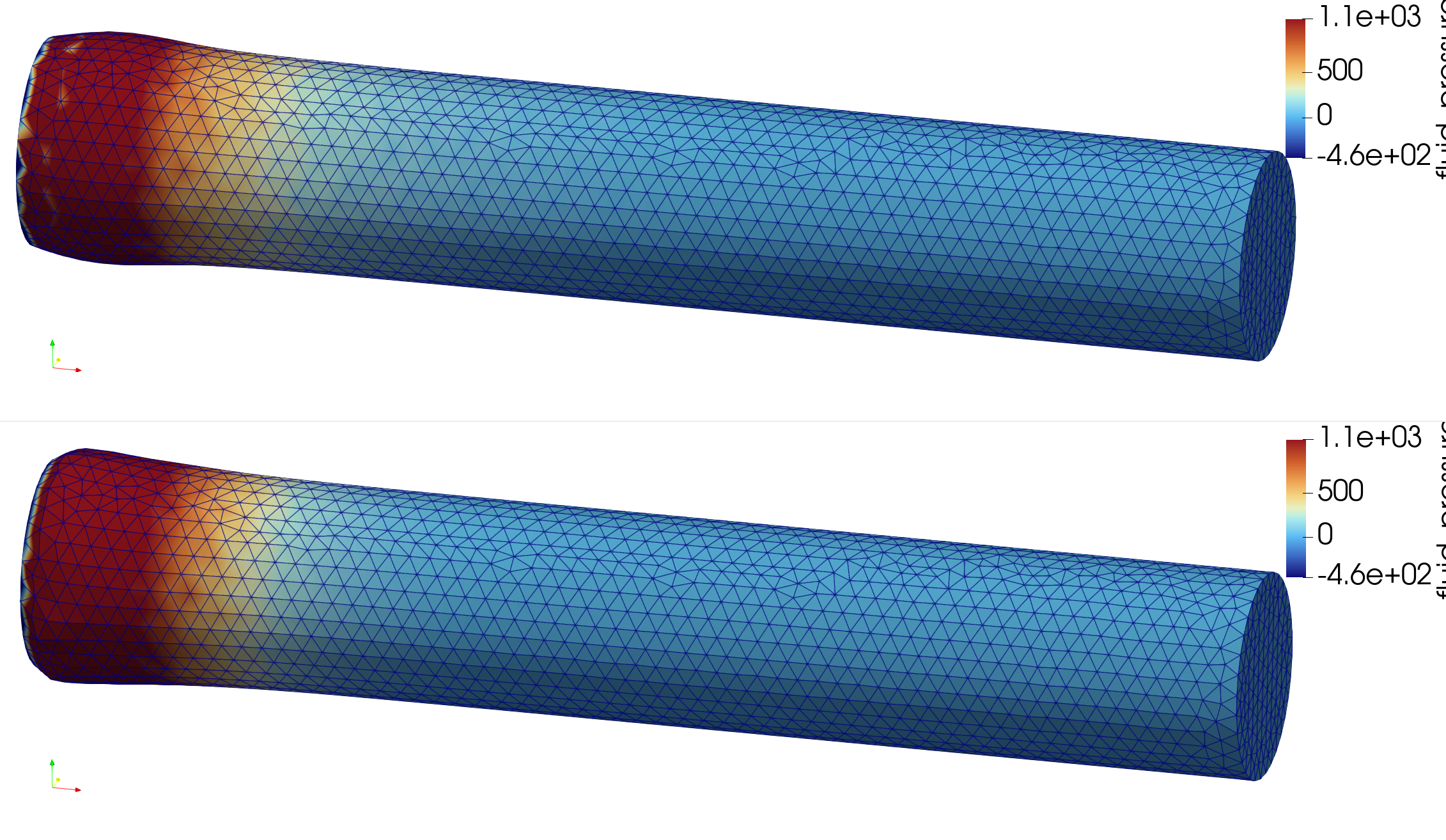}
        \caption{$t=0.004~\mathrm{s}$}
        \label{fig:pressure_wave_004}
    \end{subfigure}
    \hfill
    \begin{subfigure}[t]{0.49\textwidth}
        \centering
    \includegraphics[width=\linewidth]{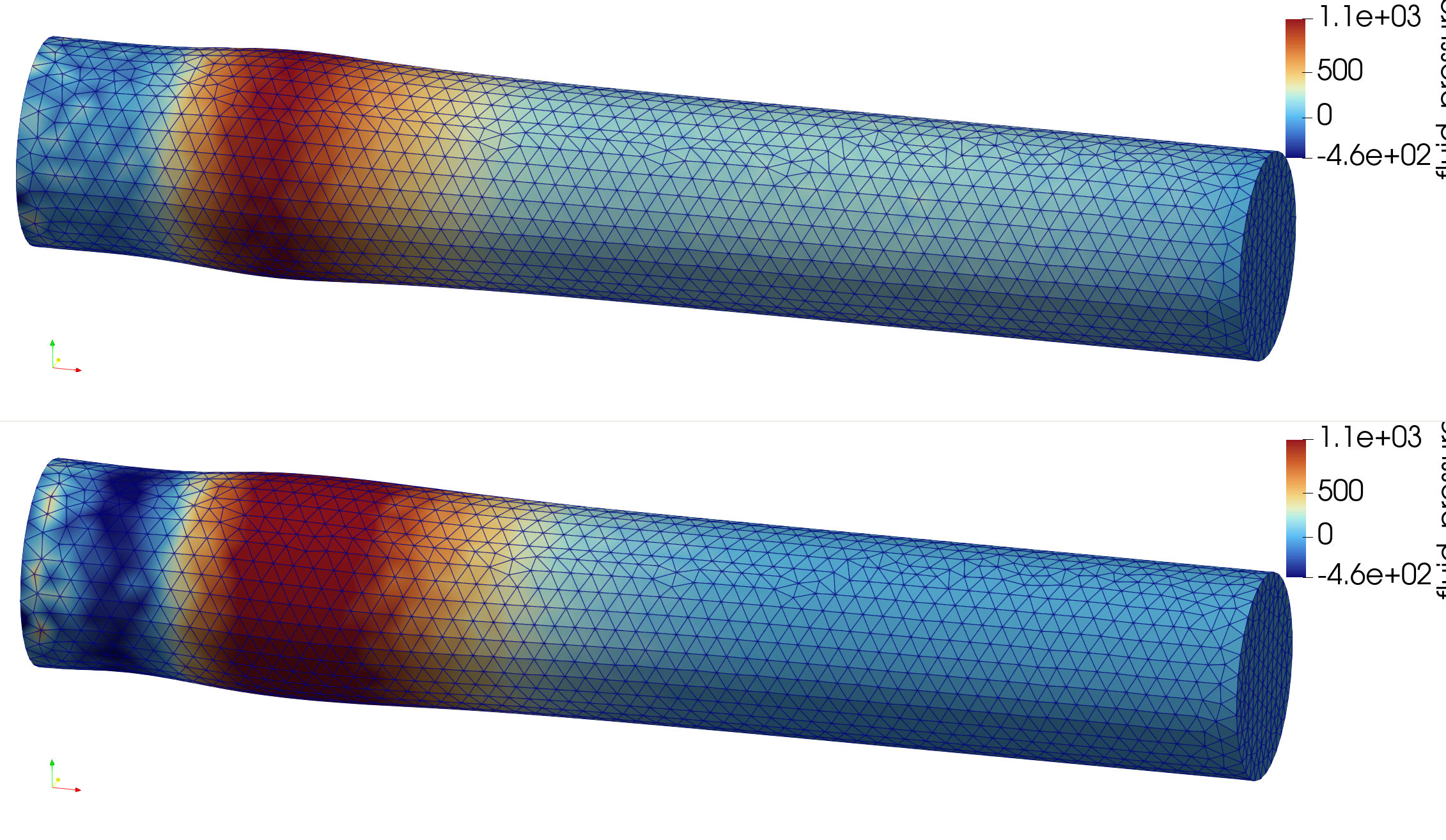}
        \caption{$t=0.008~\mathrm{s}$}
        \label{fig:pressure_wave_008}
    \end{subfigure}

    \par\medskip

    \begin{subfigure}[t]{0.49\textwidth}
        \centering
        \includegraphics[width=\linewidth]{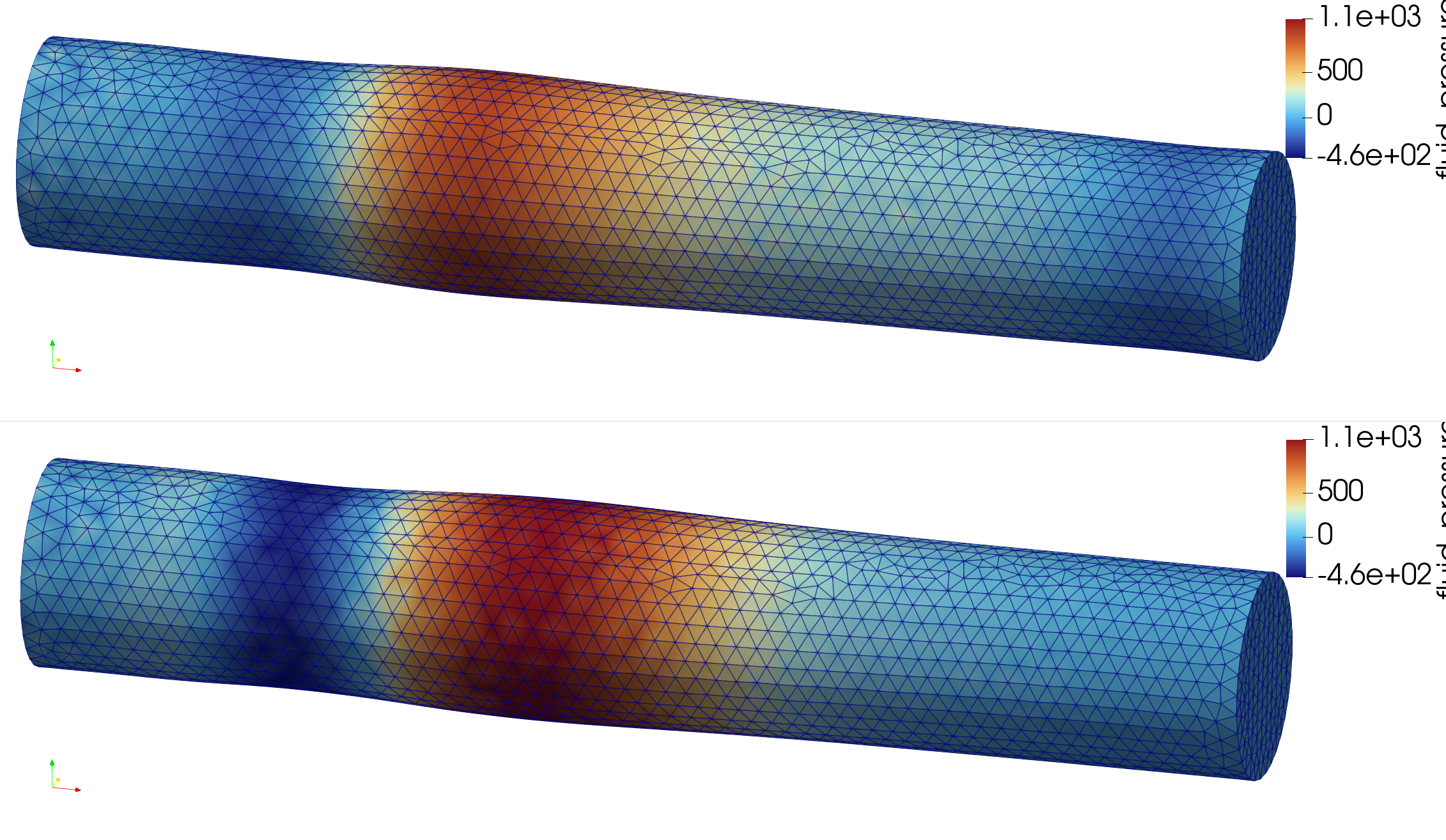}
        \caption{$t=0.012~\mathrm{s}$}
        \label{fig:pressure_wave_012}
    \end{subfigure}
    \hfill
    \begin{subfigure}[t]{0.49\textwidth}
        \centering
        \includegraphics[width=\linewidth]{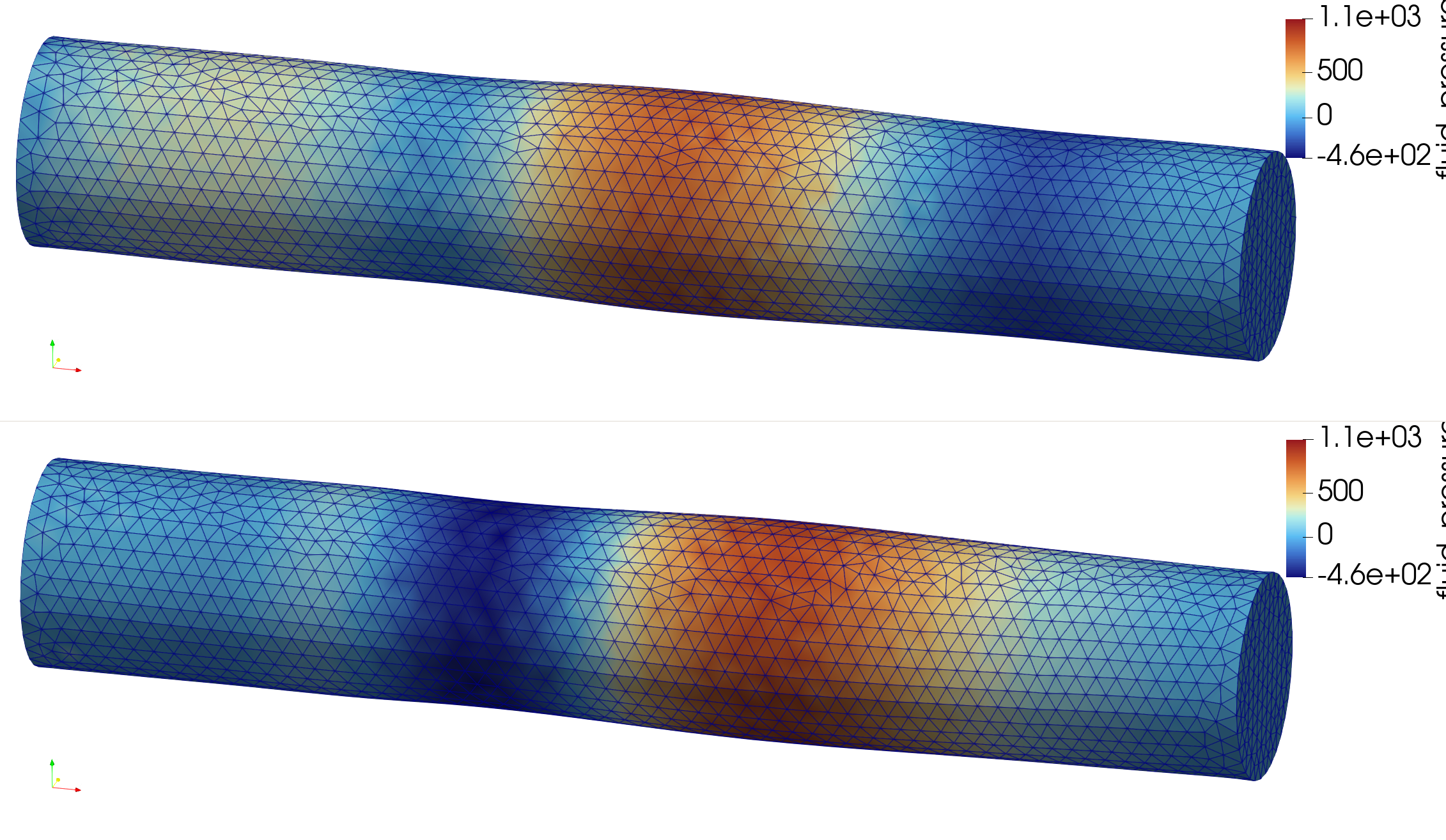}
        \caption{$t=0.016~\mathrm{s}$}
        \label{fig:pressure_wave_016}
    \end{subfigure}
    \par\medskip
    \begin{subfigure}[t]{0.49\textwidth}
        \centering
        \includegraphics[width=\linewidth]{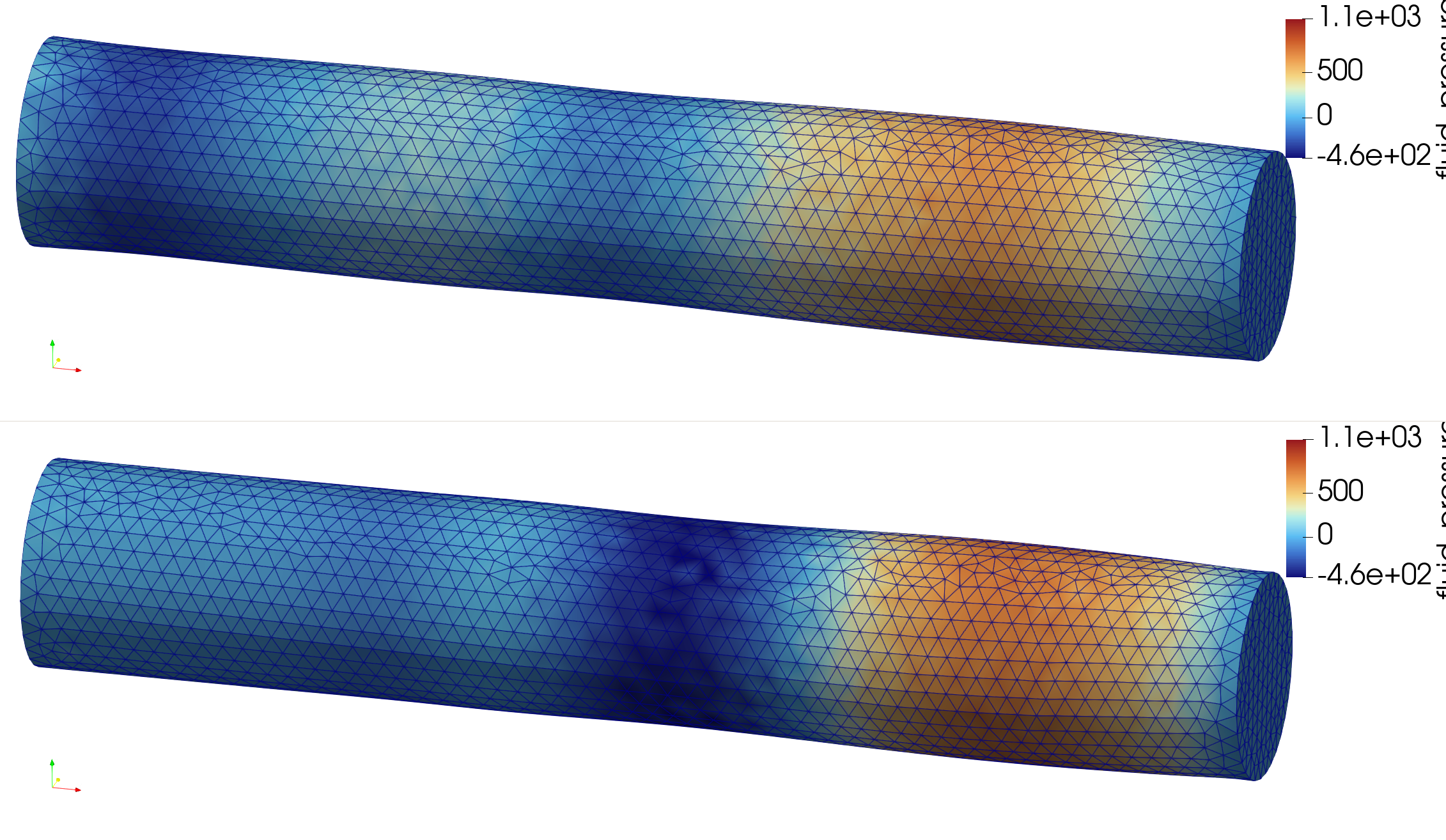}
        \caption{$t=0.020~\mathrm{s}$}
        \label{fig:pressure_wave_020}
    \end{subfigure}
    \hfill
    \begin{subfigure}[t]{0.49\textwidth}
        \centering
        \includegraphics[width=\linewidth]{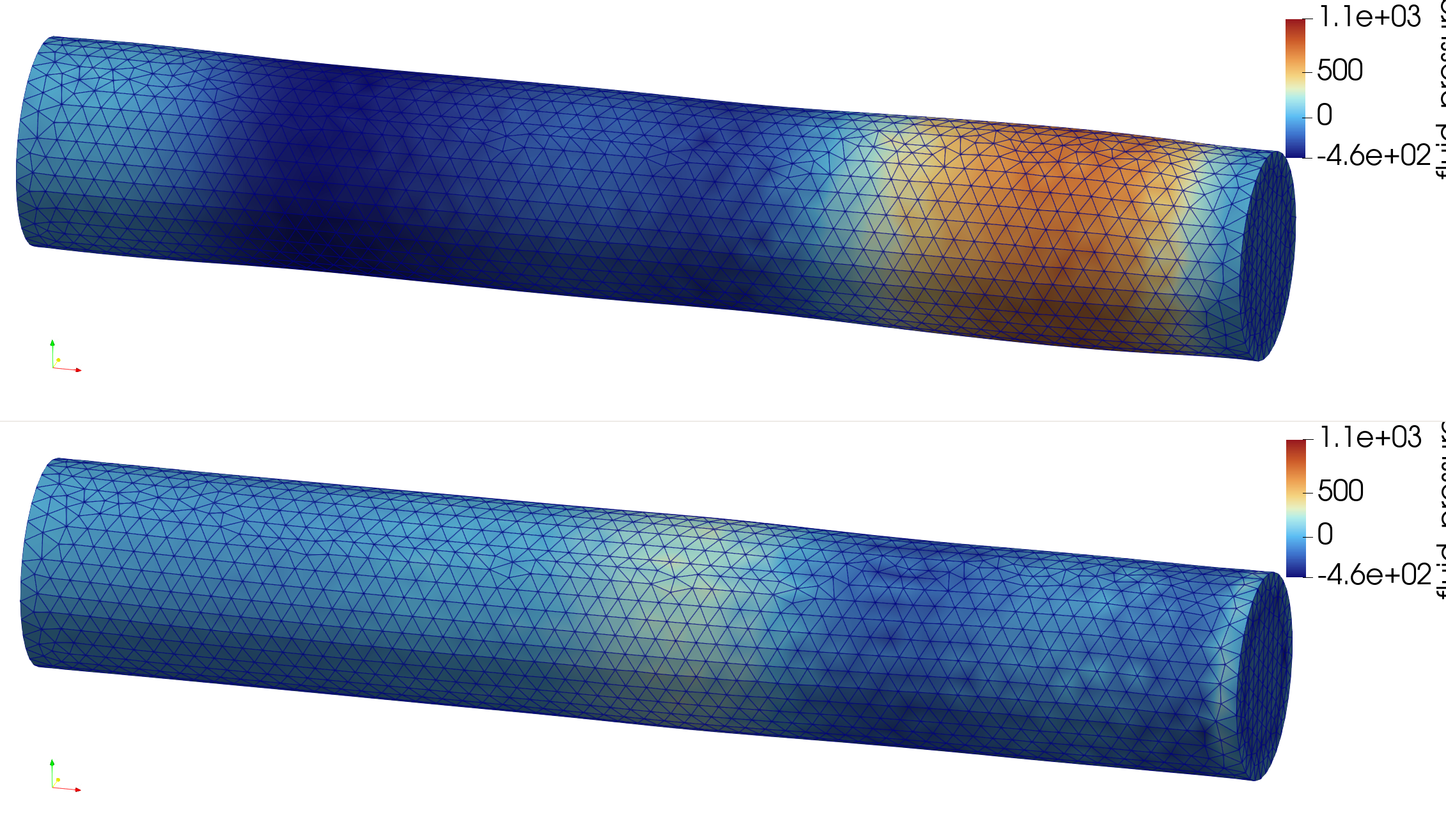}
        \caption{$t=0.024~\mathrm{s}$}
        \label{fig:pressure_wave_024}
    \end{subfigure}

    \caption{
        Snapshots of pressure-wave propagation at six time levels. In each
        subfigure, the upper panel shows the standard Naghdi-shell result from
        numerical example~1 in Section~\ref{sec:straight_pressure_wave}, whereas the lower panel shows the penalized
        Naghdi-shell result from numerical example~2 in Section~ \ref{sec:straight_pressure_wave_penalized}, which is designed to
        produce a Koiter-like shell response. The same pressure color scale is
        used in both panels. 
    }
\label{fig:pressure_wave_naghdi_comparison}
\end{figure}
Figure~\ref{fig:pressure_wave_naghdi_comparison} compares the propagation of
the pressure pulse obtained with the standard and penalized Naghdi-shell
models. The pressure-wave front in the penalized model is located farther downstream, demonstrating that it propagates faster than the wave obtained with the standard Naghdi model. Moreover,
the 3D penalized-shell wave speed remains approximately a factor of $\sqrt{2}$ below the corresponding 2D
Koiter-shell wave speed, as discussed in Remark ~\ref{remark7_1}.

\subsection{Pulmonary artery banding with Windkessel parameter estimation}
\label{sec:pulmonary_banding_parameters}
The physiological scales used in this numerical experiment are guided by the
multiscale circulation model of Baker et al.~\cite{baker2013effects}. Their
model couples a 3D representation of the bilateral branch
pulmonary arteries to a lumped-parameter model of the complete circulation.
By contrast, the present example considers one idealized compliant
pulmonary-artery segment coupled to a single downstream Windkessel outlet.
Therefore, the present setup should be only viewed as a proof-of-concept parameter estimation problem rather than a direct reproduction of the studied in~\cite{baker2013effects}.

\paragraph{Geometry and parameter setting}
We adapted the reference pulmonary and systemic vascular resistance indices as follows:
\[
\mathrm{PVRI}
=
2.3~
\mathrm{mmHg\,min\,L^{-1}\,m^2},
\qquad
\mathrm{SVRI}
=
21.6~
\mathrm{mmHg\,min\,L^{-1}\,m^2}.
\]
The reference heart rate is
$HR=120~\mathrm{beats/min}$, corresponding to the nominal cardiac-cycle
duration of $T_c={60}/{HR}=0.5~\mathrm{s}$.
The total simulation time of $1~\mathrm{s}$ therefore corresponds to two
nominal cardiac cycles. The inlet condition used below is time independent,
however, so $T_c$ will only serve as a physiological reference time in this example. The systemic vascular resistance is also reported only for context
and does not enter the fluid boundary conditions.

Blood is modeled as an incompressible Newtonian fluid with density of
$\rho_f=1060~\mathrm{kg\,m^{-3}}$
and dynamic viscosity 
$\mu_f=0.0035~\mathrm{Pa\,s}$.
The computational domain is an idealized single banded pulmonary-artery
segment, as shown in Figure~\ref{fig:banded_pa_geometry}. 
\begin{figure}[htbp]
    \centering
    \begin{subfigure}[t]{0.9\textwidth}
        \centering
     \includegraphics[width=\linewidth]{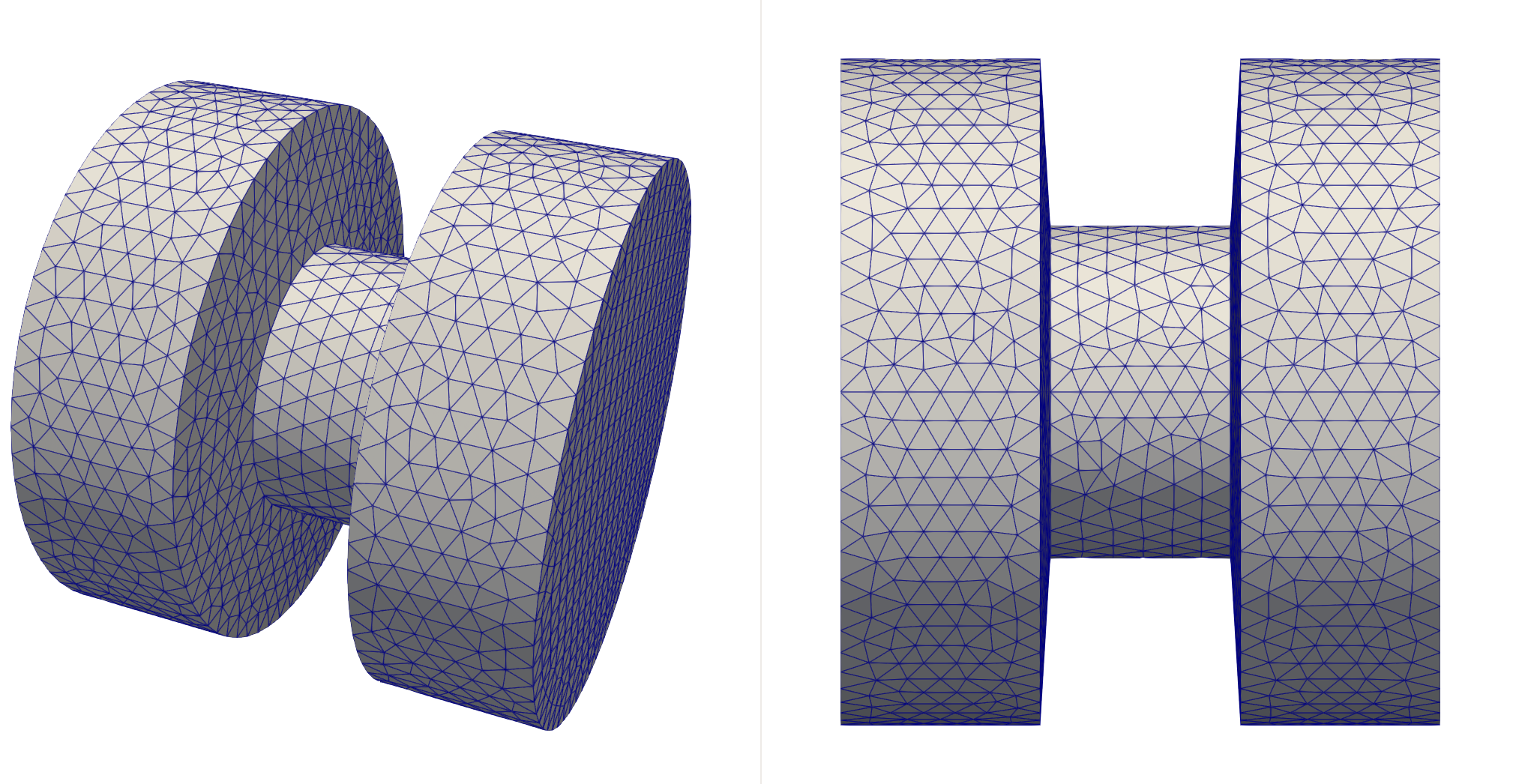}

     \includegraphics[width=\linewidth]{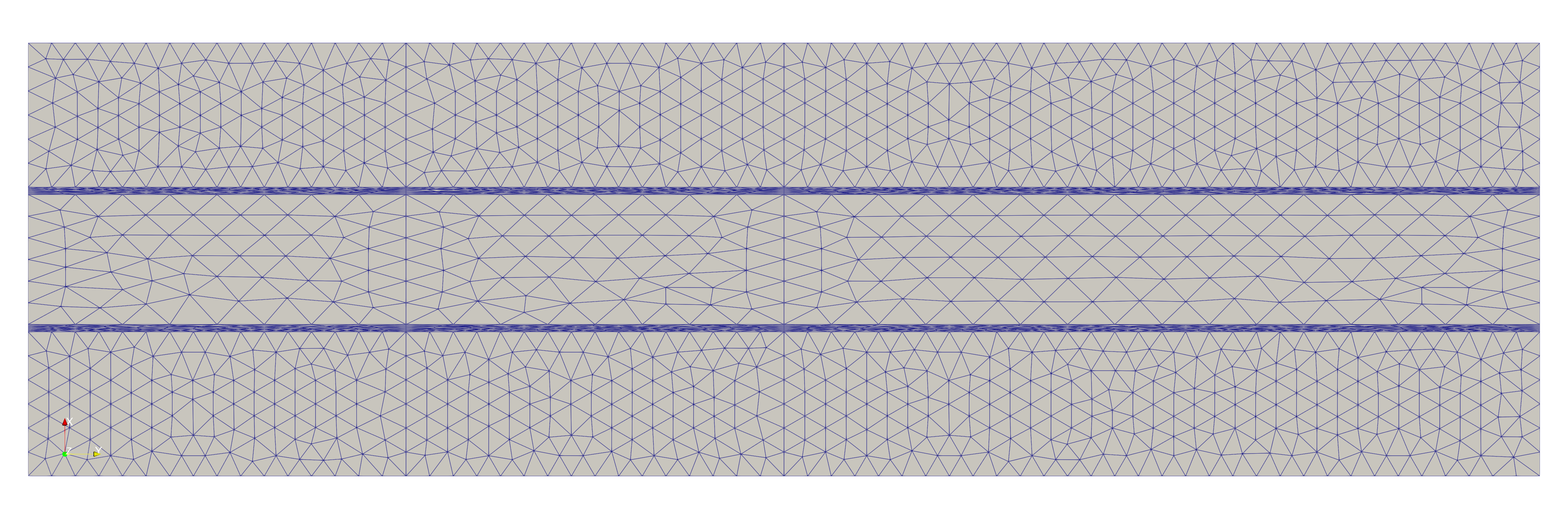}
    \end{subfigure}
\caption{Computational meshes used for the pulmonary banding example with the five-region banded geometry. The shell mesh consists of 5,920 triangular elements, while the fluid mesh consists of 20,245 tetrahedral elements. Fluid mesh: (a) side view; (b) the front view to show the five regions. Shell mesh: (c) horizontal view of Naghdi Shell mesh, showing five regions. The fluid and shell mesh are conformed at the interface.  }
\label{fig:banded_pa_geometry}
\end{figure}
The native radius, minimum band radius, and axial length of the fully constricted region are chosen as
\[
R_0=5~\mathrm{mm},
\qquad
R_b=2.5~\mathrm{mm},
\qquad
L_b=3~\mathrm{mm}.
\]
Here, $R_b$ denotes a radius and therefore corresponds to a minimum band
diameter of $5~\mathrm{mm}$. 
The length of each ramp region is chosen as
\(\ell_r=0.00015~\mathrm{m}=0.15~\mathrm{mm}\), which corresponds to
\(L/60\) for the total vessel length \(L=0.009~\mathrm{m}\). The vessel
radius is therefore defined by the continuous piecewise-affine function
\begin{equation}
R(x)=
\begin{cases}
R_{\mathrm{end}},
& 0\leq x\leq 0.003,
\\[1mm]
R_{\mathrm{end}}
+\displaystyle\frac{x-0.003}{0.00015}
\left(R_{\mathrm{mid}}-R_{\mathrm{end}}\right),
& 0.003<x<0.00315,
\\[3mm]
R_{\mathrm{mid}},
& 0.00315\leq x\leq 0.00585,
\\[1mm]
R_{\mathrm{mid}}
+\displaystyle\frac{x-0.00585}{0.00015}
\left(R_{\mathrm{end}}-R_{\mathrm{mid}}\right),
& 0.00585<x<0.006,
\\[3mm]
R_{\mathrm{end}},
& 0.006\leq x\leq 0.009,
\end{cases}
\label{eq:five_region_radius}
\end{equation}
where all axial coordinates are measured in meters.

\begin{remark}
For long-time simulations, a conforming mesh between the fluid and shell meshes at the fluid–structure interface is necessary in the present study. When the fluid and shell meshes are generated independently, repeated projection and interpolation of displacement, velocity, and traction between non-conforming interface nodes may fail to preserve circumferential symmetry and the discrete balance of forces and moments. Although the resulting transfer error may be small at each time step, it can accumulate over time and excite a spurious internal rotational mode (Our simulations with non-conforming meshes has observed this effects). This artificial rotation can contaminate the computed displacement, stress, and flow fields and, through the ALE update, progressively distort the fluid mesh. The resulting deterioration in mesh quality may eventually cause solver stagnation or premature termination of the simulation. Constructing the shell mesh directly from the corresponding fluid-interface facets provides consistent geometric and topological correspondence, suppresses artificial interface torque, and substantially improves the accuracy and robustness of long-time FSI simulations.
\end{remark}

The corresponding undeformed cross-sectional areas are
\[
A_0
=
\pi R_0^2
=
7.854\times10^{-5}~\mathrm{m^2},
\qquad
A_b
=
\pi R_b^2
=
1.963\times10^{-5}~\mathrm{m^2}.
\]
Thus, the band produces a
\[
1-\frac{A_b}{A_0}
=
1-\left(\frac{R_b}{R_0}\right)^2
=
1-\left(\frac{0.0025}{0.005}\right)^2
=
75\%
\]
reduction in the reference cross-sectional area.

We consider the infant's nominal pulmonary flow rate as
\[
Q_0
=
2.0~\mathrm{L/min}
=
3.333\times10^{-5}~\mathrm{m^3/s}.
\]
The associated cross-sectionally averaged inlet velocity is
\[
\overline U_{\mathrm{in}}
=
\frac{Q_0}{A_0}
=
\frac{3.333\times10^{-5}}
     {7.854\times10^{-5}}
\approx
0.424~\mathrm{m/s}.
\]

At the inlet, a steady parabolic velocity profile is imposed:
\[
\boldsymbol u_{\mathrm{in}}(y,z)
=
2\overline U_{\mathrm{in}}
\left(
1-\frac{y^2+z^2}{R_0^2}
\right)
\boldsymbol e_x,
\qquad
y^2+z^2\leq R_0^2.
\]
Based on these parameters, the inlet Reynolds number is approximately
\[
\mathrm{Re}_0
=
\frac{\rho_f\overline U_{\mathrm{in}}(2R_0)}{\mu_f}
\approx
1.28\times10^3.
\]
The mean velocity based on the undeformed band area and the corresponding
local Reynolds number are
\[
\overline U_b
=
\frac{Q_0}{A_b}
\approx
1.70~\mathrm{m/s},
\qquad
\mathrm{Re}_b
=
\frac{\rho_f\overline U_b(2R_b)}{\mu_f}
\approx
2.57\times10^3.
\]
Because the inlet velocity is imposed, the inlet pressure is an output of the
simulation and is compared with the pressure observation. For the
present data-assimilation test, the constant inlet pressure observation is
$$
z_{\mathrm{obs}}
=
13~\mathrm{mmHg}
=
1733.19~\mathrm{Pa}.
$$
This value is supplied to each Kalman-filter measurement update. In a
patient-specific calculation, it would be replaced by a time-dependent
inlet-pressure waveform together with its measurement uncertainty.
The downstream circulation is represented by a three-element Windkessel model
with venous reference pressure
\[
P_v
=
5~\mathrm{mmHg}
\approx
666.61~\mathrm{Pa}.
\]
The Windkessel parameters used to initialize the simulation were obtained
using the representative body-surface area
\[
\mathrm{BSA}_{\mathrm{ref}}
=
1.8~\mathrm{m^2}.
\]
The corresponding non-indexed pulmonary vascular resistance is
$$
\mathrm{PVR}_{\mathrm{ref}}
=
\frac{\mathrm{PVRI}}
     {\mathrm{BSA}_{\mathrm{ref}}}
=
\frac{2.3}{1.8}
\approx
1.278~
\mathrm{mmHg\,min/L}.
$$
For the single-outlet model, this value is used to establish the initial scale
of the total downstream Windkessel resistance. With a proximal-resistance
fraction of $\alpha=0.1$ and a target distal time constant of
$\tau=0.2~\mathrm{s}$, the corresponding parameter values are
\[
R_p^0
\approx
1.0\times10^6~
\mathrm{Pa\,s\,m^{-3}},
\qquad
R_d^0
\approx
9.1\times10^6~
\mathrm{Pa\,s\,m^{-3}},
\qquad
C^0
\approx
2.1\times10^{-8}~
\mathrm{m^3\,Pa^{-1}}.
\]
which are used for the logarithmic Kalman-filter state initialization.

The prescribed inlet pressure, flow rate, and venous pressure imply the
effective inlet-to-venous resistance
$$
R_{\mathrm{eff,obs}}
=
\frac{z_{\mathrm{obs}}-P_v}{Q_0}
=
\frac{13-5}{2.0}
=
4.00~\mathrm{mmHg\,min/L}.
$$
This effective resistance includes both the pressure loss resolved within the
3D banded segment and the contribution of the downstream Windkessel model. It
is therefore not expected to equal $\mathrm{PVR}_{\mathrm{ref}}$, which is
assigned to the initial downstream Windkessel resistance. Under the initial
scaling, their difference,
\[
R_{\mathrm{eff,obs}}-\mathrm{PVR}_{\mathrm{ref}}
\approx
2.72~\mathrm{mmHg\,min/L},
\]
represents the effective resistance contribution required from the resolved
3D domain. Using
\[
1~\mathrm{mmHg\,min/L}
=
7.99932\times10^6~
\mathrm{Pa\,s\,m^{-3}},
\]
we can obtain the initial total Windkessel resistance
$$
R_{\mathrm{RCR}}^0
=
R_p^0+R_d^0
=
1.01\times10^7~
\mathrm{Pa\,s\,m^{-3}}
\approx
1.263~
\mathrm{mmHg\,min/L}.
$$
We notice that the small difference between $R_{\mathrm{RCR}}^0$ and the reference
value $\mathrm{PVR}_{\mathrm{ref}}$ results from rounding the implemented
Windkessel parameters.
And the initial proximal-resistance fraction and distal time constant can be computed as
$$
\alpha^0
=
\frac{R_p^0}{R_p^0+R_d^0}
\approx
0.099,
\qquad
\tau^0
=
R_d^0C^0
\approx
0.191~\mathrm{s}.
$$

\paragraph{Numerical results}
We run the simulation for a total of 
$t=1s$. Figure~\ref{fig:final_time_results} presents the shell displacement magnitude and the fluid velocity magnitude at the final
simulation time.
\begin{figure}[htbp]
    \centering
\includegraphics[width=0.55\textwidth]{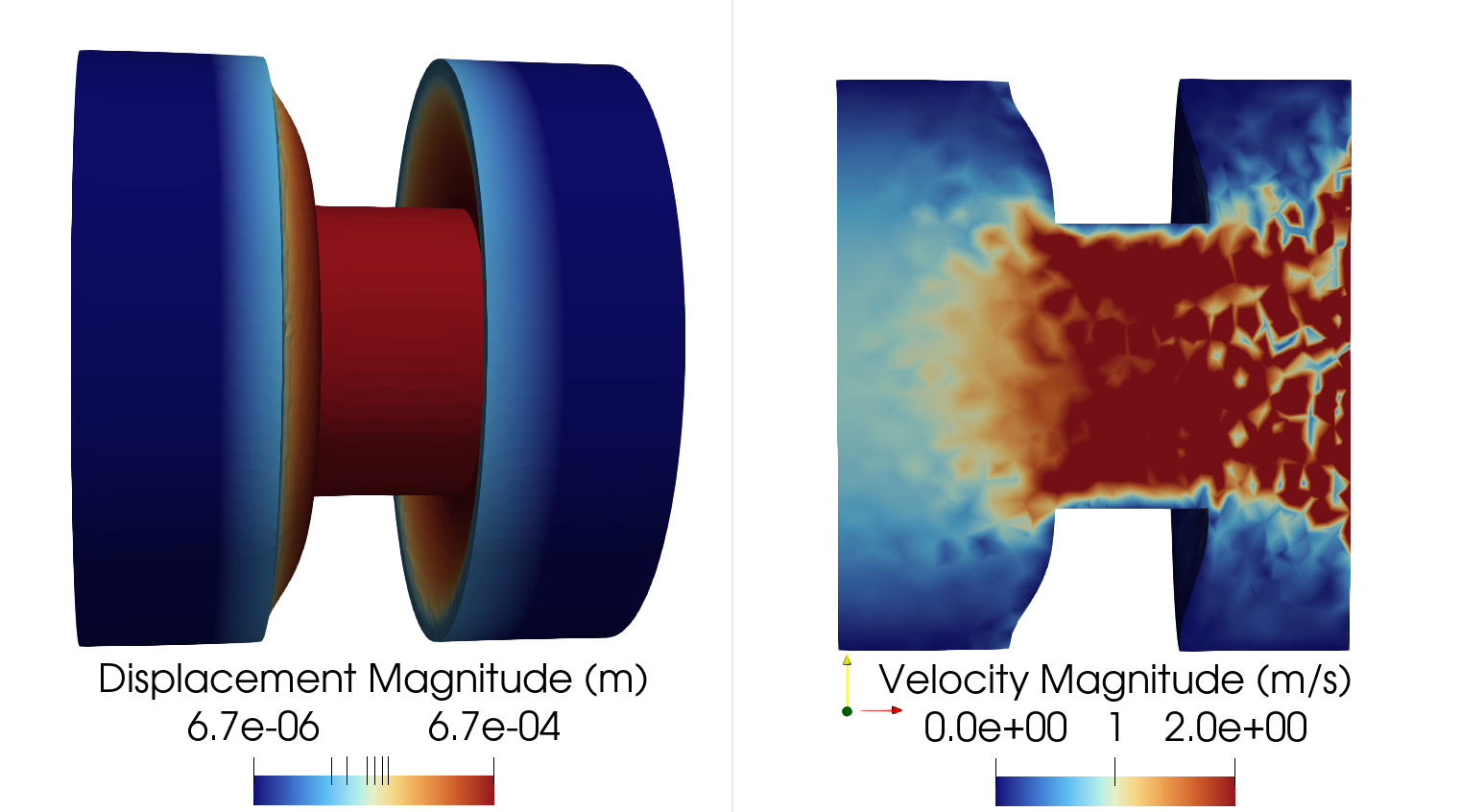}
\includegraphics[width=0.42\textwidth]{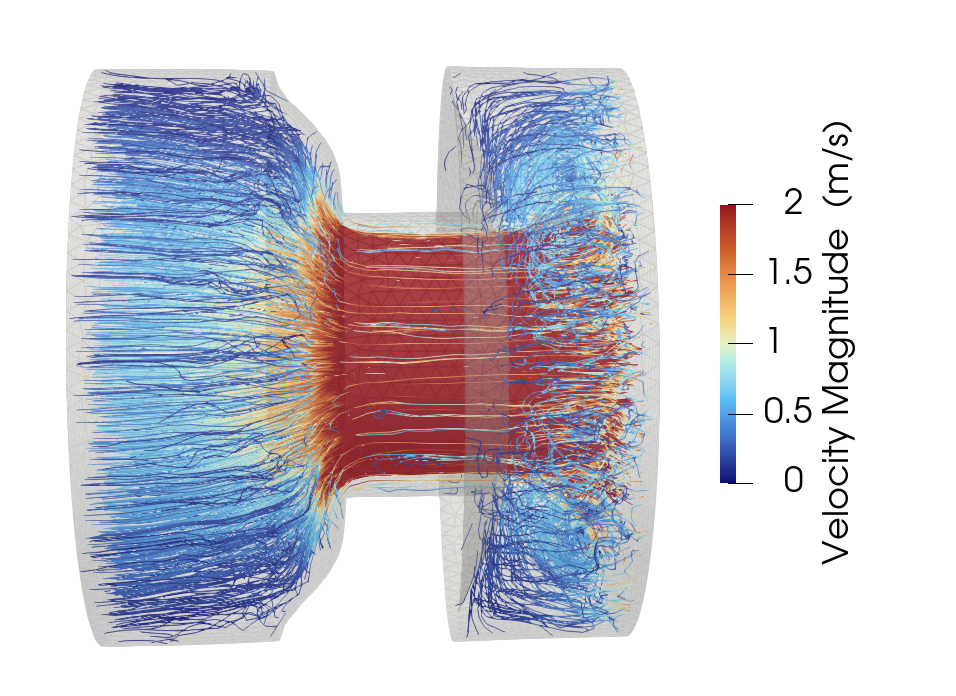}    \caption{Simulation results at the final time $t=1 s$, showing the shell displacement magnitude, fluid velocity magnitude, and fluid velocity streamlines. The shell displacement is predominantly axial (horizontal), with the banded region and the transition (ramp) regions on either side displaced toward downstream direction.}
\label{fig:final_time_results}
\end{figure}

At the final simulation time, the estimated Windkessel parameters are 
$$ R_p^{\mathrm{final}} = 2.214\times10^7~ \mathrm{Pa\,s\,m^{-3}}, \qquad R_d^{\mathrm{final}} = 9.146\times10^6~ \mathrm{Pa\,s\,m^{-3}}, \qquad C^{\mathrm{final}} = 2.085\times10^{-8}~ \mathrm{m^3\,Pa^{-1}}.$$ 
The corresponding total resistance is 
$$
R_{\mathrm{RCR}}^{\mathrm{final}} = R_p^{\mathrm{final}}+R_d^{\mathrm{final}} = 3.129\times10^7~ \mathrm{Pa\,s\,m^{-3}} \approx 3.911~ \mathrm{mmHg\,min/L},
$$
with a proximal-resistance fraction and distal time constant of 
$$
\alpha^{\mathrm{final}} = \frac{R_p^{\mathrm{final}}} {R_p^{\mathrm{final}}+R_d^{\mathrm{final}}} \approx 0.708, \qquad \tau^{\mathrm{final}} = R_d^{\mathrm{final}}C^{\mathrm{final}} \approx 0.191~\mathrm{s}. 
$$ Relative to their initial values, $R_p$ increases by a factor of approximately $22.1$, whereas $R_d$ and $C$ change by approximately $0.51\%$ and $0.72\%$, respectively. Accordingly, the proximal-resistance fraction increases from $0.099$ to $0.708$, while the distal time constant remains nearly unchanged over the simulation. These numerical results show that the largest parameter variation occurs in $R_p$, with only minor changes observed in $R_d$ and $C$.

\begin{figure}[htbp]
    \centering
    \begin{subfigure}[t]{0.9\textwidth}
        \centering
        \includegraphics[width=\linewidth]{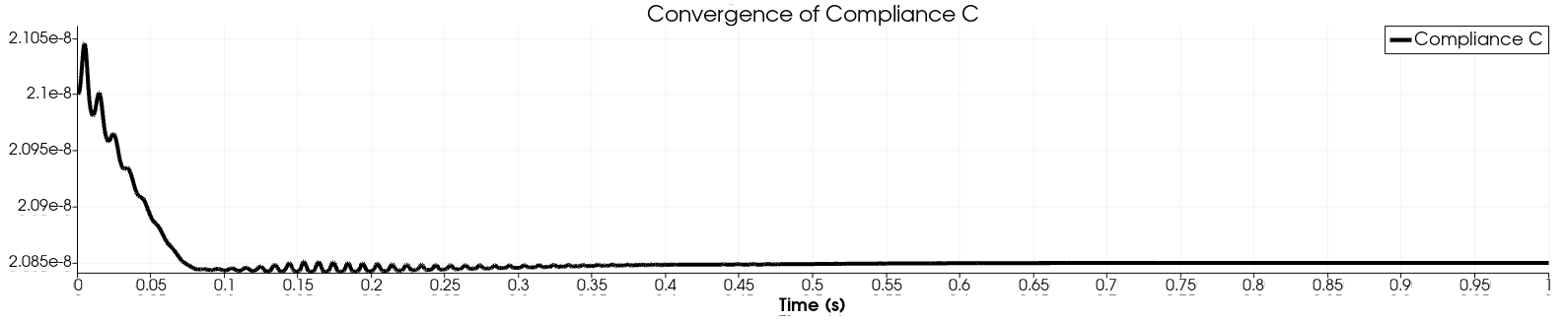}
        \caption{Convergence of the estimated compliance.}
        \label{fig:conv_C}
    \end{subfigure}
    \hfill
    \begin{subfigure}[t]{0.9\textwidth}
        \centering
        \includegraphics[width=\linewidth]{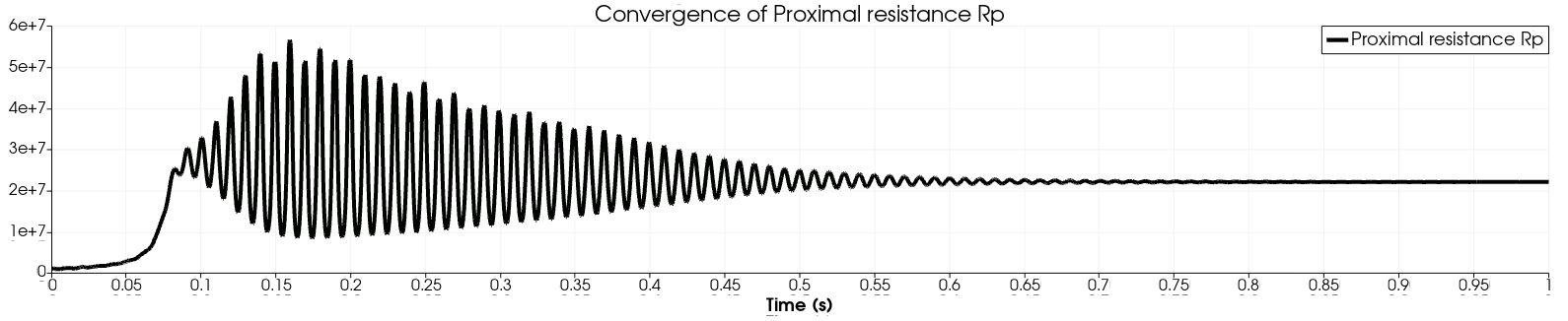}
        \caption{Convergence of the estimated proximal resistance.}
        \label{fig:conv_Rp}
    \end{subfigure}

    \medskip

    \begin{subfigure}[t]{0.9\textwidth}
        \centering
        \includegraphics[width=\linewidth]{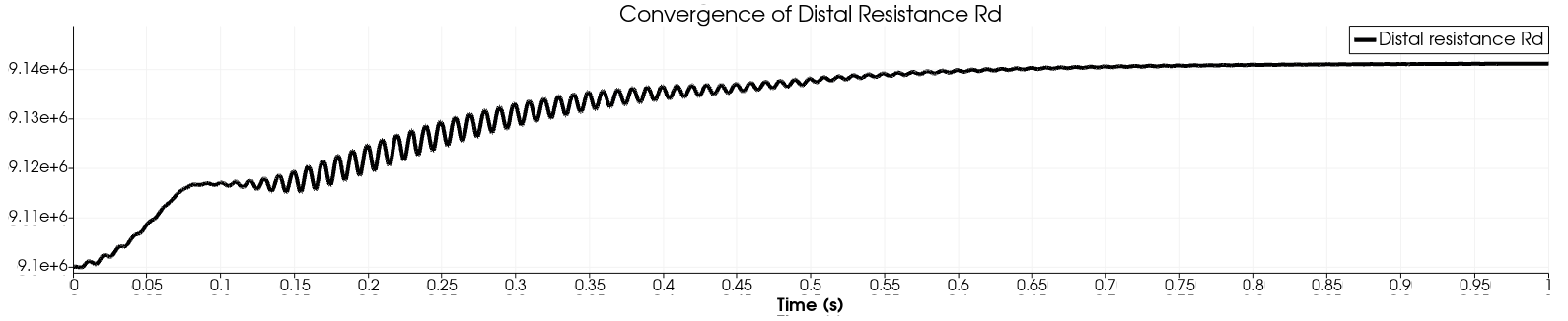}
        \caption{Convergence of the estimated distal resistance.}
        \label{fig:conv_Rd}
    \end{subfigure}
    \hfill
    \begin{subfigure}[t]{0.9\textwidth}
        \centering
        \includegraphics[width=\linewidth]{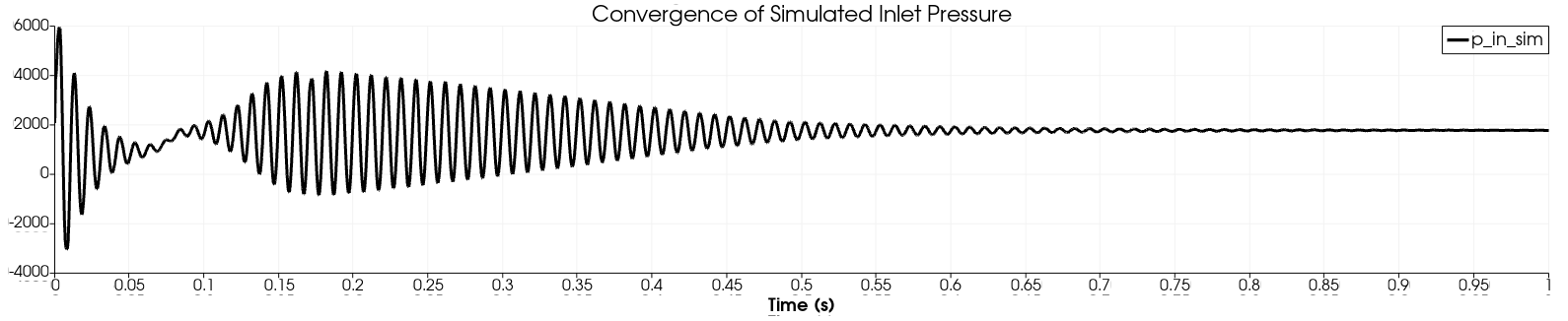}
        \caption{Convergence of the simulated inlet pressure.}
        \label{fig:conv_Pin}
    \end{subfigure}

    \caption{Convergence histories of the estimated Windkessel parameters
    and simulated inlet pressure over the full simulation interval
    \(0\leq t\leq 1\,\mathrm{s}\). At the final simulation time $t=1$ s, the converged estimates are $R_p^{\mathrm{final}}=2.214\times10^7~\mathrm{Pa\,s\,m^{-3}},~
R_d^{\mathrm{final}} =9.146\times10^6~
\mathrm{Pa\,s\,m^{-3}},~
C^{\mathrm{final}} = 2.085\times10^{-8}~
\mathrm{m^3\,Pa^{-1}}.$}
    \label{fig:parameter_convergence}
\end{figure}

Figure~\ref{fig:parameter_convergence} shows that the simulated inlet pressure converges to a stable value of $13.35~
\mathrm{mmHg}$ at the final time, but the filter does not reach to the pressure observation exactly. It differs from the prescribed observation by approximately
$0.35~\mathrm{mmHg}$, corresponding to a relative
discrepancy of $2.69\%$. This remaining innovation may due to the measurement-noise and
process-noise covariances, the initial parameter covariance, and the available observation window. 

The final Windkessel parameters permit a conditional estimate of the inlet to outlet pressure difference. Considering that the outlet flow has approached $Q_0=2.0~\mathrm{L/min}$ and that the Windkessel capacitor pressure is approximately steady, the outlet pressure satisfies
\[
P_{\mathrm{out}}
\approx P_v+Q_0
\left(R_p^{\mathrm{final}}+R_d^{\mathrm{final}}\right)
=5+2\times(3.911)~\mathrm{mmHg}
\approx 12.82~\mathrm{mmHg}.
\]
Together with the simulated inlet pressure of $13.35~\mathrm{mmHg}$, this gives
\[
P_{\mathrm{in}}-P_{\mathrm{out}}
\approx 0.53~\mathrm{mmHg},
\]
which gives the estimation of the pressure drop across the entire 3D vessel segment. Therefore, the downstream Windkessel contributes approximately $7.82~\mathrm{mmHg}$ to the total inlet to venous pressure difference of $8.35~\mathrm{mmHg}$. 

It is worth noting that, because a single pressure observable is used to estimate three Windkessel parameters, convergence of the parameter traces alone may not fully characterize the independent identifiability of all three parameters. The relatively small changes in $R_d$ and $C$, together with the nearly unchanged product of $R_dC$,
may indicate correlations among the estimated parameters. A more detailed investigation of parameter correlation and identifiability will therefore be considered in future work.

\section{Conclusion}

In this work, we developed a coupled computational framework for blood flow
through a deformable, banded pulmonary artery. The model combines the
three-dimensional incompressible Navier-Stokes equations, formulated on a
moving domain using an arbitrary Lagrangian-Eulerian mapping, with a Naghdi shell model for the arterial wall and a three-element Windkessel model for the
downstream circulation. The Naghdi formulation retains membrane, bending, and
transverse-shear effects and is therefore suitable for describing arterial
segments with spatially varying radius and localized constrictions. The
fluid, shell, and Windkessel components are connected through kinematic and
dynamic interface conditions and through the outlet flow and pressure.

A non-iterative partitioned algorithm was introduced for the coupled
fluid-structure-Windkessel problem. At each time step, the method advances
the shell, updates the ALE mapping, and solves the fluid and downstream
circulation without subiterations between subproblems. This separation retains the modularity of the individual fluid and structural solvers while avoiding the computational cost associated with a monolithic formulation. 
For the linearized problem on a fixed domain, a discrete energy estimate was derived to identify the principal energy-transfer mechanisms and to establish the stability of the partitioned coupling under the stated assumptions.

The numerical studies examined several aspects of the proposed framework.
The straight compliant-tube example demonstrated stable pressure-wave
propagation and the expected behavior under temporal and spatial refinement.
Comparisons between the Naghdi shell and reduced shell models showed that the
dominant pressure and radial-displacement responses were similar, while the
choice of structural model and scaling affected the pressure-wave speed and
the local displacement field. Although the axial displacement was much
smaller than the radial displacement, it was not identically zero, and the
results also illustrated the influence of the clamped boundaries on the
propagating pressure waves. These observations support the use of the Naghdi
model when transverse shear, axial deformation, and localized geometric
changes are relevant. The final pulmonary-artery banding example was included as a small
proof-of-concept application of the framework. It demonstrated how the
fluid-Naghdi shell model can be coupled to a three-element Windkessel outlet
and combined with a reduced-order unscented Kalman filter to update the
positive parameters \(R_p\), \(R_d\), and \(C\) from an inlet-pressure
observation. In this illustrative calculation, the estimated parameters and simulated inlet pressure approached stable final time values. 

Future work will consider extensions of applying to pulsatile
patient-specific pressure and flow data and image-based arterial geometries.
Such applications may include systematic investigations of band radius, band width, and arterial-wall properties, as well as uncertainty and parameter
sensitivity analyses. These extensions will help determine the range of biomedical applications for which the proposed coupled
fluid-structure-Windkessel framework can provide reliable quantitative
predictions.

\bibliographystyle{unsrt}
\bibliography{ref}

@article{masud1997space,
  author    = {Masud, Arif and Hughes, Thomas J. R.},
  title     = {A Space-Time Galerkin/Least-Squares Finite Element
               Formulation of the {Navier--Stokes} Equations for
               Moving Domain Problems},
  journal   = {Computer Methods in Applied Mechanics and Engineering},
  volume    = {146},
  number    = {1--2},
  pages     = {91--126},
  year      = {1997},
  month     = jul,
  publisher = {Elsevier},
  doi       = {10.1016/S0045-7825(96)01222-4}
}

@misc{chaudhary_pulmonary_artery_banding,
  author       = {Vimlendra Chaudhary},
  title        = {Pulmonary Artery Banding---Why? When? How?},
  howpublished = {SlideShare, PowerPoint presentation},
  url          = {https://www.slideshare.net/slideshow/pulmonary-artery-banding-why-when-how/284155723},
  note         = {Accessed: September 21, 2026}
}

@incollection{Naghdi1972,
  author    = {Naghdi, P. M.},
  title     = {The Theory of Shells and Plates},
  booktitle = {Handbuch der Physik},
  editor    = {Fl\"ugge, Siegfried},
  volume    = {VIa/2},
  pages     = {425--640},
  publisher = {Springer-Verlag},
  address   = {Berlin},
  year      = {1972}
}

@article{BlouzaLeDret2001,
  author  = {Blouza, Adel and Le Dret, Herv\'e},
  title   = {Nagdhi's Shell Model: Existence, Uniqueness and Continuous Dependence on the Midsurface},
  journal = {Journal of Elasticity},
  volume  = {64},
  number  = {2--3},
  pages   = {199--216},
  year    = {2001},
  doi     = {10.1023/A:1015270504666}
}

@article{TambacaTutek2016,
  author  = {Tamba\v{c}a, Josip and Tutek, Zvonimir},
  title   = {A New Linear {Naghdi} Type Shell Model for Shells with Little Regularity},
  journal = {Applied Mathematical Modelling},
  volume  = {40},
  number  = {23--24},
  pages   = {10549--10562},
  year    = {2016},
  doi     = {10.1016/j.apm.2016.07.007}
}

@inproceedings{Koiter1960,
  author    = {Koiter, W. T.},
  title     = {A Consistent First Approximation in the General Theory of Thin Elastic Shells},
  booktitle = {Proceedings of the {IUTAM} Symposium on the Theory of Thin Elastic Shells},
  pages     = {12--33},
  publisher = {North-Holland},
  address   = {Amsterdam},
  year      = {1960}
}

@book{Ciarlet2000,
  author    = {Ciarlet, Philippe G.},
  title     = {Mathematical Elasticity, Volume {III}: Theory of Shells},
  series    = {Studies in Mathematics and Its Applications},
  volume    = {29},
  publisher = {North-Holland},
  address   = {Amsterdam},
  year      = {2000},
  isbn      = {9780444828910}
}

@article{BukacEtAl2013,
  author  = {Buka\v{c}, Martina and \v{C}ani\'{c}, Sun\v{c}ica and Glowinski, Roland and Tamba\v{c}a, Josip and Quaini, Annalisa},
  title   = {Fluid--Structure Interaction in Blood Flow Capturing Non-Zero Longitudinal Structure Displacement},
  journal = {Journal of Computational Physics},
  volume  = {235},
  pages   = {515--541},
  year    = {2013},
  doi     = {10.1016/j.jcp.2012.08.033}
}

@article{MuhaCanic2013,
  author  = {Muha, Boris and \v{C}ani\'{c}, Sun\v{c}ica},
  title   = {Existence of a Weak Solution to a Nonlinear Fluid--Structure Interaction Problem Modeling the Flow of an Incompressible, Viscous Fluid in a Cylinder with Deformable Walls},
  journal = {Archive for Rational Mechanics and Analysis},
  volume  = {207},
  number  = {3},
  pages   = {919--968},
  year    = {2013},
  doi     = {10.1007/s00205-012-0585-5}
}

@article{QuarteroniTuveriVeneziani2000,
  author  = {Quarteroni, Alfio and Tuveri, Massimiliano and Veneziani, Alessandro},
  title   = {Computational Vascular Fluid Dynamics: Problems, Models and Methods},
  journal = {Computing and Visualization in Science},
  volume  = {2},
  number  = {4},
  pages   = {163--197},
  year    = {2000},
  doi     = {10.1007/s007910050039}
}

@book{FormaggiaQuarteroniVeneziani2009,
  editor    = {Formaggia, Luca and Quarteroni, Alfio and Veneziani, Alessandro},
  title     = {Cardiovascular Mathematics: Modeling and Simulation of the Circulatory System},
  series    = {MS\&A: Modeling, Simulation and Applications},
  volume    = {1},
  publisher = {Springer},
  address   = {Milan},
  year      = {2009},
  doi       = {10.1007/978-88-470-1152-6}
}

@article{OshimaEtAl2001,
  author  = {Oshima, Marie and Torii, Ryo and Kobayashi, Toshio and Taniguchi, Nobuyuki and Takagi, K.},
  title   = {Finite Element Simulation of Blood Flow in the Cerebral Artery},
  journal = {Computer Methods in Applied Mechanics and Engineering},
  volume  = {191},
  pages   = {661--671},
  year    = {2001},
  doi     = {10.1016/S0045-7825(01)00307-3}
}

@article{FigueroaEtAl2006,
  author  = {Figueroa, C. Alberto and Vignon-Clementel, Irene E. and Jansen, Kenneth E. and Hughes, Thomas J. R. and Taylor, Charles A.},
  title   = {A Coupled Momentum Method for Modeling Blood Flow in Three-Dimensional Deformable Arteries},
  journal = {Computer Methods in Applied Mechanics and Engineering},
  volume  = {195},
  number  = {41--43},
  pages   = {5685--5706},
  year    = {2006},
  doi     = {10.1016/j.cma.2005.11.011}
}

@techreport{PHC4_2019_PediatricHeartSurgery,
  author      = {{Pennsylvania Health Care Cost Containment Council}},
  title       = {Pediatric and Congenital Heart Surgery: 2014--2017 Data},
  institution = {Pennsylvania Health Care Cost Containment Council},
  address     = {Harrisburg, Pennsylvania},
  year        = {2019},
  month       = jun,
  url         = {https://www.phc4.org/wp-content/uploads/cardiac_pediatric2017report.pdf},
  note        = {Accessed September 5, 2026}
}

@article{angeli2012pulmonary,
  author  = {Angeli, Emanuela and Pace Napoleone, Carlo and
             Turci, Simone and Oppido, Guido and Gargiulo, Gaetano},
  title   = {Pulmonary Artery Banding},
  journal = {Multimedia Manual of Cardio-Thoracic Surgery},
  volume  = {2012},
  pages   = {mms010},
  year    = {2012},
  doi     = {10.1093/mmcts/mms010},
  url     = {https://doi.org/10.1093/mmcts/mms010}
}

@book{ciarlet2000shells,
  author    = {Ciarlet, P. G.},
  title     = {Mathematical Elasticity, Volume {III}: Theory of Shells},
  publisher = {North-Holland},
  address   = {Amsterdam},
  year      = {2000},
  isbn      = {978-0-444-82891-0},
  url       = {https://shop.elsevier.com/books/theory-of-shells/ciarlet/978-0-444-82891-0}
}

@book{chapelle2003finite,
  author    = {Chapelle, D. and Bathe, K.-J.},
  title     = {The Finite Element Analysis of Shells -- Fundamentals},
  publisher = {Springer},
  address   = {Berlin},
  year      = {2003},
  doi       = {10.1007/978-3-662-05229-7},
  url       = {https://doi.org/10.1007/978-3-662-05229-7}
}

@article{canic2017dimension,
  author  = {{\v{C}}ani{\'c}, S. and Galovi{\'c}, M. and Ljulj, M. and Tamba{\v{c}}a, J.},
  title   = {A dimension-reduction based coupled model of mesh-reinforced shells},
  journal = {SIAM Journal on Applied Mathematics},
  year    = {2017},
  volume  = {77},
  number  = {2},
  pages   = {744--769},
  doi     = {10.1137/16M1088181},
  url     = {https://doi.org/10.1137/16M1088181}
}

@article{bukac2015partitioned,
  author  = {Buka{\v{c}}, M. and {\v{C}}ani{\'c}, S. and Muha, B.},
  title   = {A partitioned scheme for fluid--composite structure interaction problems},
  journal = {Journal of Computational Physics},
  year    = {2015},
  volume  = {281},
  pages   = {493--517},
  doi     = {10.1016/j.jcp.2014.10.045},
  url     = {https://doi.org/10.1016/j.jcp.2014.10.045}
}

@article{canic2015stability,
  author  = {{\v{C}}ani{\'c}, S. and Muha, B. and Buka{\v{c}}, M.},
  title   = {Stability of the kinematically coupled beta-scheme for fluid--structure interaction problems in hemodynamics},
  journal = {International Journal for Numerical Analysis and Modeling},
  year    = {2015},
  volume  = {12},
  number  = {1},
  pages   = {54--80},
  url     = {https://www.global-sci.com/ijnam/article/view/10188}
}

@article{sharma2012pulmonary,
  author  = {Sharma, R.},
  title   = {Pulmonary artery banding: Rationale and possible indications in the current era},
  journal = {Annals of Pediatric Cardiology},
  year    = {2012},
  volume  = {5},
  number  = {1},
  pages   = {40--43},
  pmid    = {22529600},
  url     = {https://pubmed.ncbi.nlm.nih.gov/22529600/}
}

@article{trusler1972method,
  author  = {Trusler, G. A. and Mustard, W. T.},
  title   = {A method of banding the pulmonary artery for large isolated ventricular septal defect with and without transposition of the great arteries},
  journal = {The Annals of Thoracic Surgery},
  year    = {1972},
  volume  = {13},
  number  = {4},
  pages   = {351--355},
  doi     = {10.1016/S0003-4975(10)64866-7},
  url     = {https://doi.org/10.1016/S0003-4975(10)64866-7}
}

@article{durandy2013pulmonary,
  author  = {Durandy, Y.},
  title   = {Pulmonary artery banding is still a valuable option},
  journal = {Frontiers in Pediatrics},
  year    = {2013},
  volume  = {1},
  pages   = {17},
  doi     = {10.3389/fped.2013.00017},
  url     = {https://doi.org/10.3389/fped.2013.00017}
}

@article{baker2013effects,
  author  = {Baker, C. E. and Corsini, C. and Cosentino, D. and Dubini, G. and Pennati, G. and Migliavacca, F. and Hsia, T. Y. and {{Modeling of Congenital Hearts Alliance Investigators}}},
  title   = {Effects of pulmonary artery banding and retrograde aortic arch obstruction on the hybrid palliation of hypoplastic left heart syndrome},
  journal = {The Journal of Thoracic and Cardiovascular Surgery},
  year    = {2013},
  volume  = {146},
  number  = {6},
  pages   = {1341--1348},
  doi     = {10.1016/j.jtcvs.2013.01.038},
  url     = {https://doi.org/10.1016/j.jtcvs.2013.01.038}
}

@article{corno2006noncircular,
  author  = {Corno, A. F. and Prosi, M. and Fridez, P. and Zunino, P. and Quarteroni, A. and von Segesser, L. K.},
  title   = {The non-circular shape of {FloWatch-PAB} prevents the need for pulmonary artery reconstruction after banding: Computational fluid dynamics and clinical correlations},
  journal = {European Journal of Cardio-Thoracic Surgery},
  year    = {2006},
  volume  = {29},
  number  = {1},
  pages   = {93--99},
  doi     = {10.1016/j.ejcts.2005.10.029},
  url     = {https://doi.org/10.1016/j.ejcts.2005.10.029}
}

@incollection{liu2015effects,
  author    = {Liu, J. L. and Zhang, W. M. and Yan, Q. and Hong, H. F. and Liu, J. F.},
  title     = {Effects of band position on hemodynamics of pulmonary artery: A numerical study of patient-specific virtual procedure},
  booktitle = {World Congress on Medical Physics and Biomedical Engineering, June 7--12, 2015, Toronto, Canada},
  editor    = {Jaffray, David A.},
  series    = {IFMBE Proceedings},
  volume    = {51},
  pages     = {893--896},
  publisher = {Springer},
  address   = {Cham},
  year      = {2015},
  doi       = {10.1007/978-3-319-19387-8_218},
  url       = {https://doi.org/10.1007/978-3-319-19387-8_218}
}

@article{bukac2013fluid,
  author  = {Buka{\v{c}}, M. and {\v{C}}ani{\'c}, S. and Glowinski, R. and Tamba{\v{c}}a, J. and Quaini, A.},
  title   = {Fluid--structure interaction in blood flow capturing non-zero longitudinal structure displacement},
  journal = {Journal of Computational Physics},
  year    = {2013},
  volume  = {235},
  pages   = {515--541},
  doi     = {10.1016/j.jcp.2012.08.033},
  url     = {https://doi.org/10.1016/j.jcp.2012.08.033}
}

@article{vignonclementel2010outflow,
  author  = {Vignon-Clementel, I. E. and Figueroa, C. A. and Jansen, K. E. and Taylor, C. A.},
  title   = {Outflow boundary conditions for three-dimensional simulations of non-periodic blood flow and pressure fields in deformable arteries},
  journal = {Computer Methods in Biomechanics and Biomedical Engineering},
  year    = {2010},
  volume  = {13},
  number  = {5},
  pages   = {625--640},
  doi     = {10.1080/10255840903413565},
  url     = {https://doi.org/10.1080/10255840903413565}
}

@article{zambrano2018image,
  author  = {Zambrano, B. A. and McLean, N. A. and Zhao, X. and others},
  title   = {Image-based computational assessment of vascular wall mechanics and hemodynamics in pulmonary arterial hypertension patients},
  journal = {Journal of Biomechanics},
  year    = {2018},
  volume  = {68},
  pages   = {84--92},
  doi     = {10.1016/j.jbiomech.2017.12.022},
  url     = {https://doi.org/10.1016/j.jbiomech.2017.12.022}
}

@article{pant2014methodological,
  author  = {Pant, S. and Fabr{\`e}ges, B. and Gerbeau, J. F. and Vignon-Clementel, I. E.},
  title   = {A methodological paradigm for patient-specific multi-scale computational fluid dynamics simulations: From clinical measurements to parameter estimates for individual analysis},
  journal = {International Journal for Numerical Methods in Biomedical Engineering},
  year    = {2014},
  volume  = {30},
  number  = {12},
  pages   = {1614--1648},
  doi     = {10.1002/cnm.2692},
  url     = {https://doi.org/10.1002/cnm.2692}
}

@article{moireau2011reduced,
  author  = {Moireau, P. and Chapelle, D.},
  title   = {Reduced-order unscented {Kalman} filtering with application to parameter identification in large-dimensional systems},
  journal = {ESAIM: Control, Optimisation and Calculus of Variations},
  year    = {2011},
  volume  = {17},
  number  = {2},
  pages   = {380--405},
  doi     = {10.1051/cocv/2010006},
  url     = {https://doi.org/10.1051/cocv/2010006}
}

@article{bertoglio2012sequential,
  author  = {Bertoglio, C. and Moireau, P. and Gerbeau, J. F.},
  title   = {Sequential parameter estimation for fluid--structure problems: Application to hemodynamics},
  journal = {International Journal for Numerical Methods in Biomedical Engineering},
  year    = {2012},
  volume  = {28},
  number  = {4},
  pages   = {434--455},
  doi     = {10.1002/cnm.1476},
  url     = {https://doi.org/10.1002/cnm.1476}
}

@article{arthurs2020flexible,
  author  = {Arthurs, C. J. and Xiao, N. and Moireau, P. and Schaeffter, T. and Figueroa, C. A.},
  title   = {A flexible framework for sequential estimation of model parameters in computational hemodynamics},
  journal = {Advanced Modeling and Simulation in Engineering Sciences},
  year    = {2020},
  volume  = {7},
  pages   = {48},
  doi     = {10.1186/s40323-020-00186-x},
  url     = {https://doi.org/10.1186/s40323-020-00186-x}
}

@article{deus2024efficient,
  author  = {Deus, J. and Martin, E.},
  title   = {Efficient cardiovascular parameters estimation for fluid--structure simulations using gappy proper orthogonal decomposition},
  journal = {Annals of Biomedical Engineering},
  year    = {2024},
  volume  = {52},
  pages   = {3037--3052},
  doi     = {10.1007/s10439-024-03568-z},
  url     = {https://doi.org/10.1007/s10439-024-03568-z}
}

@article{BukacEtAl2019Stents,
  author  = {Buka{\v{c}}, Martina and
             {\v{C}}ani{\'c}, Sun{\v{c}}ica and
             Tamba{\v{c}}a, Josip and Wang, Yifan},
  title   = {Fluid--structure interaction between pulsatile
             blood flow and a curved stented coronary artery
             on a beating heart: A four stent computational study},
  journal = {Computer Methods in Applied Mechanics and Engineering},
  year    = {2019},
  volume  = {350},
  pages   = {679--700},
  doi     = {10.1016/j.cma.2019.03.034}
}

@article{CanicEtAl2019FluidMeshShell,
  author  = {{\v{C}}ani{\'c}, Sun{\v{c}}ica and
             Gali{\'c}, Marija and Ljulj, Matko and Muha, Boris and
             Tamba{\v{c}}a, Josip and Wang, Yifan},
  title   = {Analysis of a linear {3D}
             fluid--mesh--shell interaction problem},
  journal = {Zeitschrift f{\"u}r angewandte Mathematik und Physik},
  year    = {2019},
  volume  = {70},
  number  = {2},
  pages   = {44},
  doi     = {10.1007/s00033-019-1087-1}
}

@misc{CanicEtAl2024Stenosis,
  author        = {{\v{C}}ani{\'c}, Sun{\v{c}}ica and
                   Guo, Shihan and Wang, Yifan and
                   Yue, Xiaohe and Zheng, Haibiao},
  title         = {Extended one-dimensional reduced model
                   for blood flow within a stenotic artery},
  year          = {2024},
  eprint        = {2409.16262},
  archivePrefix = {arXiv},
  primaryClass  = {math.NA},
  doi           = {10.48550/arXiv.2409.16262},
  note          = {arXiv preprint arXiv:2409.16262}
}

@article{GuoEtAl2025FPSI,
  author  = {Guo, Shihan and Sun, Yizhong and Wang, Yifan and
             Yue, Xiaohe and Zheng, Haibiao},
  title   = {A Fully Parallelizable Loosely Coupled Scheme
             for Fluid-Poroelastic Structure Interaction Problems},
  journal = {SIAM Journal on Scientific Computing},
  year    = {2025},
  volume  = {47},
  number  = {4},
  pages   = {B951--B975},
  doi     = {10.1137/24M1695713}
}

@misc{GuoEtAl2025RobinRobin,
  author        = {Guo, Shihan and Lin, Ping and Wang, Yifan and
                   Yue, Xiaohe and Zheng, Haibiao},
  title         = {An Unconditionally Stable Explicit
                   {Robin--Robin} Partitioned Scheme
                   for Fluid--Structure Interaction},
  year          = {2025},
  eprint        = {2510.13096},
  archivePrefix = {arXiv},
  primaryClass  = {math.NA},
  doi           = {10.48550/arXiv.2510.13096},
  note          = {arXiv preprint arXiv:2510.13096}
}

@misc{WangLeeCanic2026ErrorAnalysis,
  author        = {Wang, Yifan and Lee, Jeonghun and
                   {\v{C}}ani{\'c}, Sun{\v{c}}ica},
  title         = {Error Analysis of the Explicit Splitting Scheme
                   for Fluid-Poroelastic Structure Interaction Problems},
  year          = {2026},
  eprint        = {2603.21088},
  archivePrefix = {arXiv},
  primaryClass  = {math.NA},
  doi           = {10.48550/arXiv.2603.21088},
  note          = {arXiv preprint arXiv:2603.21088}
}

@misc{WangLeeCanic2026SecondOrder,
  author        = {Wang, Yifan and Lee, Jeonghun and
                   {\v{C}}ani{\'c}, Sun{\v{c}}ica},
  title         = {Second order explicit splitting scheme
                   for fluid-poroelastic structure interaction problems},
  year          = {2026},
  eprint        = {2606.19811},
  archivePrefix = {arXiv},
  primaryClass  = {math.NA},
  doi           = {10.48550/arXiv.2606.19811},
  note          = {arXiv preprint arXiv:2606.19811}
}
\end{document}